\documentclass{amsart}

\usepackage[colorlinks,citecolor=niceblue,linkcolor=niceblue]{hyperref}
\usepackage{comment}
\usepackage{esint}
\usepackage{amssymb}
\usepackage[alphabetic]{amsrefs}
\usepackage{tikz}
\usetikzlibrary {3d}
\usetikzlibrary{calc}
\usetikzlibrary{patterns}
\usepackage{graphicx}
\usepackage{enumitem}
\usepackage{array}
\usepackage[export]{adjustbox}

\newtheorem{theorem}{Theorem}
\newtheorem{lemma}[theorem]{Lemma}
\newtheorem{corollary}[theorem]{Corollary}
\newtheorem{proposition}[theorem]{Proposition}
\theoremstyle{definition}

\newtheorem{remark}[theorem]{Remark}

\newtheorem{definition}[theorem]{Definition}

\newcommand{\eref}[1]{(\ref{e.#1})}
\newcommand{\tref}[1]{Theorem \ref{t.#1}}
\newcommand{\lref}[1]{Lemma \ref{l.#1}}
\newcommand{\pref}[1]{Proposition \ref{p.#1}}
\newcommand{\cref}[1]{Corollary \ref{c.#1}}
\newcommand{\fref}[1]{Figure \ref{f.#1}}
\newcommand{\sref}[1]{Section \ref{s.#1}}

\newcommand{\deflink}[2]{\hyperref[#1]{#2}}

\numberwithin{theorem}{section}
\numberwithin{equation}{section}

\newcommand{\R}{\mathbb{R}}

\newcommand{\grad}{\nabla}

\def\XXint#1#2#3{{\setbox0=\hbox{$#1{#2#3}{\int}$ }
\vcenter{\hbox{$#2#3$ }}\kern-.6\wd0}}

\newcommand{\ep}{\varepsilon}
\newcommand{\osc}{\mathop{\textup{osc}}}

\newcommand{\Diss}{\operatorname{Diss}}

\newcommand{\FS}{\Sigma^{f}} % free surface
\newcommand{\PS}{\Sigma^{p}} % plate contact surface
\newcommand{\CS}{\Sigma^{o}} % container contact surface
\newcommand{\PL}{\gamma^{p}} % plate contact line
\newcommand{\CL}{\gamma^{o}} % container contact line
\newcommand{\THP}[1]{\theta_{\gamma^p_{#1}}} %plate contact angle
\newcommand{\THC}[1]{\theta_{\gamma^o_{#1}}} %container contact angle
\newcommand{\THG}[1]{\theta_{\gamma_{#1}}} %generic contact angle
\newcommand{\Cyl}{\mathrm{Cyl}^d} % d-dim annular cylinder, the region liquid fills
\newcommand{\thinCyl}{\mathrm{Cyl}^{d-1}} % the (d-1)-dim product of a (d-2)-sphere and an interval
\newcommand{\YP}{\theta^{p}_{Y}} % Young angle, plate
\newcommand{\YC}{\theta^{o}_{Y}} % Young angle, container
\newcommand{\relAreaP}[1]{\lvert{#1}\rvert_{\mathrm{rel},p}} % relative area of a subset of the plate boundary
\newcommand{\relAreaC}[1]{\lvert{#1}\rvert_{\mathrm{rel},o}} % relative area of a subset of the container boundary
\newcommand{\vol}{\operatorname{vol}} % d-dimensional volume of a set

\newcommand{\one}{\mathbf{1}}

\definecolor{niceblue}{rgb}{0,0,0.7}

\def\strikethrough#1{\setbox0\hbox{#1}\rlap{#1}\hbox to \wd0{\hss\strikebox\hss}}
\def\strikebox{\vrule height 0.6\ht0 depth -0.4\ht0 width 1.1\wd0}
\newcommand\numberthis{\stepcounter{equation}\tag{\theequation}}

\begin{document}

% \title{Wilhelmy plate model with gravity}
\title{Rate independent capillary motion on a narrow Wilhelmy plate}
\author{Carson Collins}
\address{University of California, Los Angeles, Department of Mathematics, Los Angeles, CA}
\author{William M Feldman}
\address{University of Utah, Department of Mathematics, Salt Lake City, UT}
\subjclass{49Q20, 35Q35, 35R35}
\keywords{Rate-independent evolutions, gradient systems, capillary boundary conditions, contact angle hysteresis}
\begin{abstract}
    We study a rate independent energetic model of the Wilhelmy plate experiment in capillarity. The evolution is driven by vertical motions of the plate. We show stability of energy solutions to the evolution, in the sense used in the rate-independent systems literature, as the ratio between container width and plate width goes to infinity. In particular, we show that the volume-constraint for the finite-ratio problem disappears in the limit. This leads to a volume-unconstrained Dirichlet-forced evolution, a setting where monotonicity, uniqueness, and contact line regularity properties have been established in previous literature.

    Our result is based on using comparison principle techniques for the prescribed mean curvature equation with capillary contact angle condition that characterizes the liquid surface at equilibrium. Through barrier arguments, we are able to develop asymptotics for the energy which give us control independent of the container-to-plate ratio.
\end{abstract}

\maketitle

\setcounter{tocdepth}{1}
\tableofcontents

\section{Introduction}
In this paper we study a model for the evolution of a capillary surface in the scenario of the Wilhelmy plate experiment. In this experiment a solid plate is dipped into a liquid bath and then slowly lifted in order to measure dynamic contact angles. See \fref{intro-fig}. We will work with a rate-independent energetic evolution incorporating the effects of contact angle hysteresis on the plate.  Our main result is a limit theorem in the case where the ratio $R>1$ of the container radius to the plate radius becomes very large.  We focus on the evolution of the contact line on the plate, and show how a monotone evolution arises in the limit from volume constrained evolutions as $R \to \infty$.

From a mathematical perspective one of the issues that we want to explore in this paper is the connection between volume-forced evolutions, as studied in \cite{AlbertiDeSimone}, and Dirichlet-forced evolutions as studied in \cites{FeldmanKimPozarEnergy,FeldmanKimPozarViscosity,CollinsFeldman}. The volume constrained case is more relevant to typical experimental setups; however, its spatial regularity of solutions and uniqueness properties are very challenging to address. On the other hand, with monotonically varying Dirichlet data, there has been progress toward describing the motion law of arbitrary energy solutions \cites{FeldmanKimPozarEnergy}, and showing regularity and uniqueness of solutions constructed by discrete-time schemes \cite{CollinsFeldman}. Intuitively speaking one may expect Dirichlet driven evolutions to arise in a setting where the length scale defined by the volume constraint is quite large relative to the length scale of displacements of the contact line. Our main result will make this idea rigorous in a model of the Wilhelmy plate experiment.  We show that the Dirichlet-forced evolution can be derived as a suitable scaling limit of the volume-forced evolution.

%\begin{figure}
%    \centering
%    \includegraphics[width=0.5\linewidth]{}
%    \caption{Caption}
%    \label{fig:placeholder}
%\end{figure}

We work in the abstract framework of \cite{mielke2015book} for energetic rate-independent systems. To introduce this framework, we require a state space $X$, an energy functional $\mathcal{E}:X\to (-\infty, \infty]$ (representing the potential energy stored in the system), a dissipation functional $\Diss:X\times X\to [0,\infty)$ (representing the energy lost to friction when moving from one state to another), and a flux $Q:X\times [0,T]\to \R$ (representing the energy added or removed to the system by external forces over time). Then the notion of energy solution is:
\begin{definition}
A measurable path $L: [0,T]\to X$ is an \emph{energy solution} if it satisfies
\begin{enumerate}
    \item (Global stability) For each $t\in [0, T]$ and for all $L'\in X$,
    \[ \mathcal{E}[L(t)] \leq \mathcal{E}[L'] + \Diss[L(t), L']. \]
    \item (Energy-dissipation inequality) For each $t_0, t_1$ with $0 \leq t_0 < t_1 \leq T$,
    \[ \mathcal{E}[L(t_0)] - \mathcal{E}[L(t_1)] + \int_{t_0}^{t_1} Q[L(t),t] dt \geq \Diss[L; [t_0, t_1]] \]
\end{enumerate}
\end{definition}
\begin{remark}
    Energy solutions actually satisfy energy-dissipation balance:
    \[ \mathcal{E}[L(t_0)] - \mathcal{E}[L(t_1)] + \int_{t_0}^{t_1} Q[L(t),t] dt = \overline{\Diss}[L; [t_0, t_1]] \]
    where the \emph{total dissipation} $\overline{\Diss}$ is defined as the supremum over partitions of $\sum \Diss[L(t_i), L(t_{i+1})]$. The reverse inequality follows from global stability (see \lref{reverse_edi}).
\end{remark}

In our setting, the energy $\mathcal{E}$ accounts for the gravitational energy of the liquid, as well as the capillary energy associated to the free surface of the liquid, the liquid-plate interface, and the liquid-container interface. The state space $X$ incorporates the volume-constraint for the liquid. The dissipation $\Diss$ is modeled as arising from a constant frictional force per unit length applied along the contact line between the liquid and the inserted plate. The flux $Q$ represents the energy put into the system by moving the plate vertically, and depends on both the state of the system and the rate at which the plate is moved. We state all of these quantities more precisely in \sref{energy_intro}, and we subsequently restate the definition of energy solution adapted to our problem in Definition \ref{def:constrained_energy_sln}.

Our main result addresses stability of this notion of solution as the plate becomes small compared to the container. In the sequel, we represent this by introducing a parameter $R\in (1,\infty]$ for the ratio of the container width to the plate width; we annotate the state space $X_R$ and energy $\mathcal{E}_R$ accordingly. Then, the result is as follows:

\begin{theorem}[see \cref{compactness_of_energy_solutions_in_r}]\label{t.main}
    Let $(R_n)$ be given with $R_n\to \infty$ as $n\to\infty$, and let $L_{R_n}(t)$ be a family of energy solutions for $\mathcal{E}_{R_n}$ under the same external forcing. Then there exists a subsequence $n_k$ and a measurable path $L_\infty(t)$ such that
    \begin{itemize}
        \item $L_\infty(t)$ is an energy solution for $\mathcal{E}_\infty$.
        \item Relative energies converge: for any $t_0, t_1$: \[\mathcal{E}_{R_{n_k}}[L_{R_{n_k}}(t_0)] - \mathcal{E}_{R_{n_k}}[L_{R_{n_k}}(t_1)]\to \mathcal{E}_\infty[L_\infty(t_0)] - \mathcal{E}_{\infty}[L_\infty(t_1)] \]
        \item The liquid-plate contact sets of the $L_{R_{n_k}}(t)$ converge to those of the $L_\infty(t)$, pointwise in time.
        \item For each $t$, $L_\infty(t)$ is an $L^1_{\mathrm{loc}}$-subsequential limit of the $L_{R_{n_k}}(t)$ (for a subsequence which may depend on $t$).
        \item The fluxes $Q_{R_n}[L_{R_n}(t),t]$ converge in measure on $[0,T]$ to $Q_\infty[L_\infty(t),t]$.
    \end{itemize}

    %Conceptual result: limit of energy solutions $L_R(t)$ is an energy solution. More details
    %\begin{itemize}
    %    \item Existence of solutions $L_R(t)$ for all $R$
    %    \item All subsequences $L_R(t)$ as $R \to \infty$ have a further subsequence converging to an energy solution of $\infty$ problem (at the level of contact sets).
    %    \item Convergence of profiles, relativized energies, and pressure along $t$-dependent subsequences, to a measurable limit.
    %\end{itemize}
\end{theorem}

We also present a proof that solutions for any $R$ and any initial data can be constructed by a discrete scheme.

The proof of \tref{main} takes advantage of frameworks developed in previous literature for limits of energetic rate independent systems, especially \cite{AlbertiDeSimone}. The main new ideas in the proof center around the introduction of comparison principle techniques to obtain sufficiently strong quantitative control on the solutions in order to pass to the limit $R \to \infty$. In particular this includes a uniform bound on the vertical variations of the free surface and exponential decay of oscillations, due to gravity, away from the inner and outer containers.  One key point is that we are able to estimate the Lagrange multiplier associated with the volume constraint to within order $R^{1-d}$. This allows us to obtain strong asymptotics on the solutions in an outer region away from the plate. In particular global approximation errors in each term of the energy decay in $R$, even as the energy itself diverges in $R$; thus, we can show that a suitably relativized energy converges. At the end of the introduction, we state \tref{main_quantitative} as a summary of the quantitative content of this argument.

Though we do not present details, we believe that the analysis of \cite{CollinsFeldman} can be adapted in a straightforward manner to describe minimizing movements solutions of the $R=\infty$ system. These arguments depend in an essential way on comparison arguments that only become available once the volume constraint disappears. In summary, this should imply:
\begin{itemize}
    \item If $F(t)$ is strictly monotone in time, then there is a unique minimizing movements solution to the evolution. If $F(t)$ is piecewise strictly monotone, then minimizing movements solutions can be classified by their values at monotonicity changes.
    \item Initial regularity of the liquid-solid-vapor contact line is preserved in time for minimizing movements solutions, according to the optimal $C^{1,1/2-}$ contact line regularity for minimizers of the capillary energy with obstacle \cite{DePhilippisFuscoMorini}.
\end{itemize}

\begin{figure}[h]
    \centering
    \begin{tabular}{cc}
       \begin{tikzpicture}[xscale = .7, yscale = .23, line cap=round, line join=round]

  %things which need to go behind the inner cylinder
% Parameters
  \def\R{4} 
  \def\H{15} 
    %circles
  \draw[black!20!white,dashed] (0,0) ++(-\R,0) arc[start angle=180, end angle=0, radius=\R];
    \draw[black] (0,.45*\H) ++(-\R,0) arc[start angle=180, end angle=0, radius=\R];

\def\R{.5} 
\foreach \r in {1,...,13}
  \draw[black] (0,{.5*\H-.05*(\r-6)}) ++(-\R-.25*\r,0) arc[start angle=180, end angle=0, radius=\R+.25*\r];

    % Parameters
  \def\R{.5} 
  \def\H{15} 

     \filldraw[white,semitransparent] (-\R,0) rectangle (\R,\H);

       % Parameters
  \def\R{4} 
  \def\H{15} 
 \filldraw[black!10!white,semitransparent] (0,0) ++(-\R,0) arc[start angle=180, end angle=360, radius=\R] -- (\R,.45*\H) arc [start angle = 360,end angle=180, radius=\R] ;

       \def\R{.5} 
  \def\H{15} 

  \filldraw[black!10!white,semitransparent] (0,0) ++(-\R,0) arc[start angle=180, end angle=360, radius=\R] -- (\R,.55*\H) arc [start angle = 360,end angle=180, radius=\R] ;
    %circles
  \draw[dashed] (0,0) ++(-\R,0) arc[start angle=180, end angle=0, radius=\R];
  \draw        (0,0) ++(-\R,0) arc[start angle=180, end angle=360, radius=\R];
  \draw (0,\H) ++(-\R,0) arc[start angle=180, end angle=0, radius=\R];
  \draw         (0,\H) ++(-\R,0) arc[start angle=180, end angle=360, radius=\R];

\draw[black]         (0,.55*\H) ++(-\R,0) arc[start angle=180, end angle=360, radius=\R];
\foreach \r in {1,...,13}
  \draw[black] (0,{.5*\H-.05*(\r-6)}) ++(-\R-.25*\r,0) arc[start angle=180, end angle=360, radius=\R+.25*\r];

    %sides
  \draw (-\R,0) -- (-\R,\H);
  \draw ( \R,0) -- ( \R,\H);

  % Parameters
  \def\R{4} 
  \def\H{15} 

    %circles
  %\draw[dashed] (0,0) ++(-\R,0) arc[start angle=180, end angle=0, radius=\R];
  \draw        (0,0) ++(-\R,0) arc[start angle=180, end angle=360, radius=\R];
  \draw (0,\H) ++(-\R,0) arc[start angle=180, end angle=0, radius=\R];
  \draw         (0,\H) ++(-\R,0) arc[start angle=180, end angle=360, radius=\R];

\draw[black]         (0,.45*\H) ++(-\R,0) arc[start angle=180, end angle=360, radius=\R];

    %sides
  \draw (-\R,0) -- (-\R,\H);
  \draw ( \R,0) -- ( \R,\H);

\node at (-.5*\R,0*\H) {$L$};
\node at (-.5*\R,.85*\H) {$V$};
 \node at (0,.85*\H) {$P$};
 
\end{tikzpicture}  &  \begin{tikzpicture}
        \begin{scope}[rotate = 90]
            \draw[dotted] (0,0) circle (3);
        \draw (-3,0)--(3,0);
        \draw (-3/1.41,3/1.41)--(0,0);
        \node at (1,1.75) {$V$};
        \node at (-1.75,.75) {$L$};
        \node at (0,-1.5) {$P$};
        %\node[above] at (3,0) {$B_r$};
        \node[left] at (-1.6/1.41,1.85/1.41) {$\FS_L$};
        \node[right] at (-1.5,0) {$\PS_L$};
        \filldraw (0,0) circle (1.5pt);
        \node[right] at (0,0) {$\PL_L$};
        \draw (-.35,0) arc (180:130:.35);
        \node at (-1.7*.6/1.41,0.65*.6/1.41) {$\THP{L}$};
        \end{scope}
    \end{tikzpicture}
    \end{tabular}
    \caption{\emph{Left}: a 3D view of the liquid-vapor-plate configuration. \emph{Right}: a cross-sectional view of the configuration near the liquid-vapor-plate contact line, labeling the free surface $\FS_L$, the liquid-plate contact surface $\PS_L$, the liquid-vapor-plate contact line $\PL_L$, and the liquid-plate contact angle $\theta_{\PL_L}$, which may differ from the optimal Young angle $\YP$ due to hysteresis.}
    \label{f.intro-fig}
\end{figure}
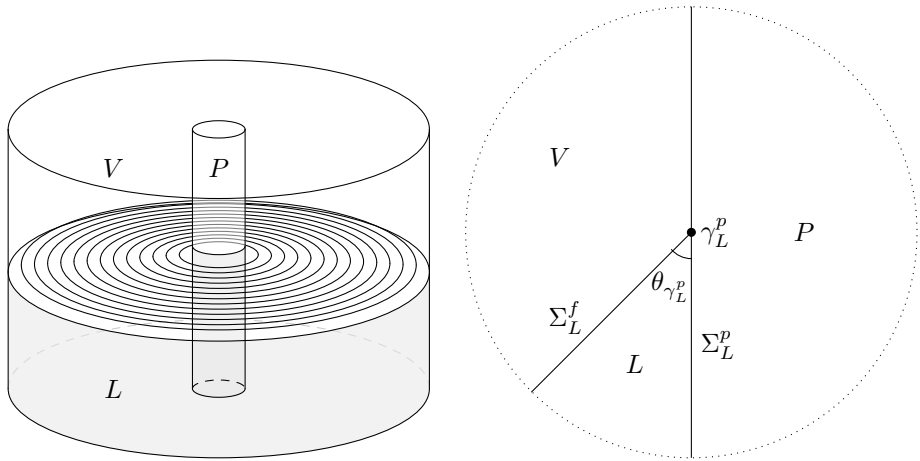

%\subsection{Literature} 

\subsection{Notation for Wilhelmy plate setting}\label{s.energy_intro}

Throughout, we write $(x,z)\in \R^{d-1}\times \R$ to distinguish the vertical component, which plays a special role due to the presence of gravity. In our setup of the problem, we will use the following notation for common lower-dimensional sets, where $R$ is a parameter in $(0, \infty]$:
\begin{align*}
    B_R &= \{ x\in \R^{d-1} : |x| < R\}, \\
    \partial B_R &= \{ x\in \R^{d-1} : |x| = R\}, \\
    \thinCyl_R &= \partial B_R\times \R, \ \hbox{ and } \\
    \thinCyl_{R,a} &= \thinCyl_R\cap \{z \leq a \}.
\end{align*}
In our setting, the liquid will be allowed to occupy the solid annular cylinder bounded by $\thinCyl_{1}$ and $\thinCyl_R$, for $R\in (1, \infty]$. We use the notation:
\begin{align*}
    \Cyl_R &= (B_R\times \R)\setminus \overline{B_1\times \R} \\
    \Cyl_{R,a} &= \Cyl_R\cap \{ z \leq a \}   
\end{align*}
for such $d$-dimensional regions.

To avoid confusion between sets of different dimension, we preference $| \cdot |$ to refer to $d-1$-dimensional areas, i.e. surface areas. We write $\vol(U)$ for the $d$-dimensional volume of a set $U$, when relevant, and use Hausdorff measure explicitly notated with dimension for $(d-2)$-dimensional objects.

We set the plate as the cylinder $\overline{B_1\times \R}$, and the container as the cylinder $\overline{B_R\times \R}$, for $R > 1$. Thus, $\Cyl_R$ is the open region that the liquid may occupy. Let $F$ denote the vertical displacement of the container relative to the plate, and let $\cos\theta_{Y_c}, \cos\theta_{Y_p}, g$ be constant parameters.

Recall the Caccioppoli subsets of an open set $U \subset \R^d$
\[\textup{Cacc}(U) := \left\{ V \subset U : \underset{|\vec{\varphi}| \leq 1}{\sup_{\vec{\varphi}\in C^1_c(U; \R^d)}}\int_V \nabla\cdot\vec{\varphi} < \infty \right\}.\]
We say $V$ is locally Caccioppoli in $U$ if it is Caccioppoli in any compactly contained open subset of $U$. We say that $V$ is locally Caccioppoli in the closed set $F$ if $V\cap F\in \mathrm{Cacc}_{\mathrm{loc}}(\R^d)$.

In what follows, we will always assume that we are considering locally Caccioppoli sets in $\overline{\Cyl_R}$. When referring to boundary operators such as $\partial L$, we mean reduced boundary or $BV$-trace (written $\mathrm{tr}_S(L)$, when taken on a $d-1$ surface $S$) as appropriate; in practice, the sets we consider will have enough regularity that this does not play an important role.

We now develop the notation for the different elements of a liquid profile $L$. For the liquid-plate contact set, we write
\[\PS_L = \textup{tr}_{\thinCyl_1}(L).\]
For the liquid-container contact set, we write:
\[ \CS_{L} = \textup{tr}_{\thinCyl_R}(L). \]
For the free surface of the liquid, we write:
\[ \FS_L = \partial L\cap \Cyl_R. \]
We also write $\nu^f_L$ for the outward unit normal to $L$ along the free surface.

For the two contact lines on the plate and container, we write respectively $\PL_L = \PS_L\cap \FS_L$ and $ \CL_L = \CS_L\cap \FS_L$.

We measure the area of the liquid-plate contact set relative to height $z = 0$:
\[ \relAreaP{\PS} = |\PS\setminus \thinCyl_{1,0}| - |\thinCyl_{1,0} \setminus \PS| \]
and the area of the liquid-container contact set relative to height $z = F$
\[ \relAreaC{\CS} = |\CS\setminus \thinCyl_{R,F}| - |\thinCyl_{R,F} \setminus \CS|.\]
We note the dependence on $R,F$, though they are suppressed from the left side notation.

\subsection{Energy formulation} Our energy is composed of contributions from each liquid-vapor and liquid-solid interface, as well as gravity:
\begin{equation}\label{eq:energy}
    \mathcal{E}_R[L; F] = |\FS_L| - \cos\YP\relAreaP{\PS_L} - \cos\YC \relAreaC{\CS_L} + \int_{L\Delta \Cyl_{R,F}} g|z - F|.
\end{equation}
Here $\YP$ and $\YC$ are the Young contact angles for the plate and outer container respectively.  The quantity
\begin{equation*}\label{e.bond-number}
    g  = \frac{\rho \mathfrak{g} r^2}{\sigma_{\textup{LV}}}
\end{equation*}
is the non-dimensional Bond number, where $r$ is the characteristic length scale, $\mathfrak{g}$ is the gravitational constant, $\rho$ is the density different between the liquid and vapor phases, and $\sigma_{\textup{LV}}$ is the surface tension between liquid and vapor phases. Since we take $g$ as an $O(1)$ constant and the width of the plate is normalized to $1$, we are essentially assuming that the Bond number for plate radius $r$ is of unit order. This can be checked to be physically reasonable for, for example, water and air where the Bond number will be of unit size around $2.7$ millimeters, for example see \cite{RAPP2017445}*{Section 21.2}.

For minimization purposes, we impose the volume constraint inside $\Cyl_R$:
\begin{equation}\label{eq:volume_constraint}
    \mathcal{C}_{R,F}(L) := \int_{L\setminus \Cyl_{R,F}} 1 - \int_{\Cyl_{R,F}\setminus L} 1 = 0
\end{equation}
We remark that as part of this constraint, we require both integrals to be finite.
    
To consider the container limit as $R\to \infty$, we also define
\begin{equation}\label{e.limiting_energy}
    \mathcal{E}_\infty[L] = \int_{\FS_L} (1 - e_z\cdot \nu^f_L)\,d\mathcal{H}^{d-1} - \cos\YP\relAreaP{\PS_L} + \int_{L\Delta \Cyl_{\infty,F}} g|z - F|
\end{equation}
Here, the relativization of $|\FS_L|$ is consistent with subtracting $|B_R|$ from $\mathcal{E}_R$.

The dissipation distance between two liquid-plate contact sets is
\[ \Diss(L_0,L_1) = \Diss(\PS_{L_0}, \PS_{L_1}) = \mu_+ | \PS_{L_1} \setminus \PS_{L_0}| + \mu_-| \PS_{L_0}\setminus \PS_{L_1}|\]
Let us note that the dissipation is nonsymmetric, but still satisfies the following triangle inequality (see, for example, \cite{FeldmanKimPozarEnergy}*{Lemma A.2}):
\begin{equation}\label{e.diss_triangle_inequality}
    \Diss[L_0, L_2] \leq \Diss[L_0, L_1] + \Diss[L_1, L_2].
\end{equation}
%The pressure in the finite radius case is
%\begin{equation}\label{eq:pressure_r}
%    P_R[L] = \cos\YC \mathcal{H}^{d-2}(\partial B_R) + \lambda(L)|B_R\setminus B_1|
%\end{equation}
%In the limit case the pressure is
%\begin{equation}\label{eq:pressure_inf}
%    P_\infty[L] = -\int_{\PL_L} e_z\cdot \eta^{\mathrm{co}}_{\FS_{L}}\,d\mathcal{H}^{d-2} 
%\end{equation}
% The energy dissipation inequality: for all $t_0 < t_1$
% \begin{equation}\label{eq:EDI}
%     \mathcal{E}_R(L_{t_0}) - \mathcal{E}_R(L_{t_1}) + \int_{t_0}^{t_1}P(t)\dot{F}(t)dt \geq \Diss[L_{t_0}, L_{t_1}]
% \end{equation}
% where
% \begin{equation}\label{eq:pressure}
%     P(t) = -\int_{\FS_{L_t}\cap \partial C_1} e_z\cdot \eta^{\mathrm{co}}_{\FS_{L_t}}\,d\mathcal{H}^{d-2} = \cos\YC |\partial B_R| + \lambda(t)|B_R\setminus B_1|
% \end{equation}
% And the limiting EDI
% \begin{equation}\label{eq:limiting_EDI}
% \mathcal{E}(L_{t_0})- \mathcal{E}(L_{t_1}) + \int_{t_0}^{t_1}\dot{F}(t)P(t) \ dt\geq\Diss(L(t_0),L(t_1))
% \end{equation}

We define the space of admissible sets for the finite container evolution to be all sets which are locally Caccioppoli, satisfy the volume constraint, and are such that the relative surface areas of both contact sets are finite, so that the value of $\mathcal{E}_R\in (-\infty, \infty]$ is well-defined:
\[ X^R_F = \{ X\in \textup{Cacc}_{loc}(\overline{\Cyl_R}) : \mathcal{C}_{R,F}[L] = 0 \hbox{ and } \relAreaC{\CS_L},\relAreaP{\PS_L} \hbox{ are finite }\} \]

We note that $X^R_F \subset L^1_{\mathrm{loc}}(\Cyl_R)$, a Fr\'echet space with topology induced by the $L^1$ seminorms on a countable exhaustion of $\Cyl_R$ by compact sets. When we discuss measurability of evolutions, it will always be with respect to the $L^1_{\mathrm{loc}}$ topology.

\begin{definition}\label{def:constrained_energy_sln}
    A measurable mapping $t \mapsto L(t) \in X^R_{F(t)}$ on $t \in [0,T]$ is a solution if
    \begin{enumerate}
        \item[(i)] (Global stability with volume constraint) For all $t \in [0,T]$ and all $L' \in X^R_F$
        \[ \mathcal{E}_R[L(t);F(t)] \leq \mathcal{E}_R[L';F(t)] + \Diss(L(t),L').\]
        As a consequence of global stability, $L_t$ satisfies the Euler-Lagrange equation:
        \begin{equation}
            -H_{\FS_{L(t)}} + gz = \lambda(t) \hbox{ on } \FS_{L(t)}
        \end{equation}
        where $\lambda(t)$ is a constant.
        
        \item[(ii)] (Energy dissipation inequality) For all $0 \leq t_0 <t_1 \leq T$,
        \begin{equation}\label{e.constrained_edi}
            \mathcal{E}_R[L(t_0); F(t_0)]- \mathcal{E}_R[L(t_1); F(t_1)] + \int_{t_0}^{t_1}P_R[L(t)]\dot{F}(t)dt\geq\Diss(L(t_0),L(t_1)).
        \end{equation}
        where the pressure $P_R[L(t)]$ is defined by
        \begin{equation}\label{eq:pressure_r}
            P_R(t) = -\int_{\PL_{L_t}} e_z\cdot \eta^{\mathrm{co}}_{\FS_{L_t}}\,d\mathcal{H}^{d-2} -\int_{\CL_{L_t}} e_z\cdot \eta^{\mathrm{co}}_{\FS_{L_t}}\,d\mathcal{H}^{d-2}.
        \end{equation}
        In practice, we use a weaker formulation of the pressure stated in \pref{energy_taylor_exp_alt}, for which it is easier to show convergence.
    \end{enumerate}
\end{definition}
For the limiting evolution, we expect the volume constraint to become ineffective. Thus, we define admissible sets as:
\[X^\infty_F = X^\infty = \{ X\in \textup{Cacc}_{loc}(\Cyl_\infty) : \relAreaP{\PS_L}\in (-\infty, \infty) \}\]
\begin{definition}\label{def:limiting_energy_sln}
    A measurable mapping $t \mapsto L(t) \in X$ on $t \in [0,T]$ is a solution if
    \begin{enumerate}
        \item[(i)] (Global stability) For all $t \in [0,T]$ and all $L' \in X$
        \begin{equation}
            \mathcal{E}_\infty[L(t);F(t)] \leq \mathcal{E}_\infty[L';F(t)] + \Diss(L(t),L').
        \end{equation}
        \item[(ii)] (Energy dissipation inequality) For all $0 \leq t_0 <t_1 \leq T$
        \begin{equation}
            \mathcal{E}_\infty[L(t_0); F(t_0)]- \mathcal{E}_\infty[L(t_1); F(t_1)] + \int_{t_0}^{t_1}P_\infty[L(t)]\dot{F}(t)dt\geq\Diss(L(t_0),L(t_1))
        \end{equation}
        where
        \begin{equation}\label{eq:pressure}
            P_\infty[L(t)] = -\int_{\PL_{L_t}} e_z\cdot \eta^{\mathrm{co}}_{\FS_{L_t}}\,d\mathcal{H}^{d-2}.
        \end{equation}
    \end{enumerate}
\end{definition}

We now state a quantitative counterpart to \tref{main}. This result estimates the energy-error incurred in modifying a stable profile in $X^R_F$ to be admissible for the infinite radius setting. Notably, this estimate requires a shift in forcing which decays with $R$.

\begin{theorem}\label{t.main_quantitative}
    Let $L_R\in X^R_F$ be globally stable. Then there exists an explicit universal constant $a_0 := -\cos\YC \frac{\omega_{d-2}}{\omega_{d-1}}$, where $\omega_n$ is the area of the $n$-dimensional unit ball, such that:
    \[ X^\infty_{F + \frac{a_0}{gR}}\ni \widetilde{L}_R := (L_R\cap \Cyl_{R/2})\cup (\Cyl_{\infty, F + \frac{a_0}{gR}}\setminus \Cyl_{R/2})  \]
    satisfies
    \[ \mathcal{E}_\infty[\widetilde{L}_R; F + \frac{a_0}{gR}] = \mathcal{E}_R[L_R; F] + E(R) + O\left(\frac{\log R}{R}\right) \]
    where $E(R)$ is an explicit correction for the energy of $L_R$ near the container independent of $L_R$ and $F$.
\end{theorem}
\begin{proof}
    This is a consequence of \lref{energy_asymptotics}. We can first replace $L_R$ in $\Cyl_R\setminus \Cyl_{R/2}$ by the reference configuration $\mathsf{L}_R$ defined in \eqref{e.reference_configuration}, incurring an $O(\frac{\log R}{R})$ error in the energy. Subsequently, we can replace the reference configuration by the flat profile $\Cyl_{\infty, F + \frac{a_0}{gR}}\setminus \Cyl_{R/2}$ extended to infinity, incurring a finite shift in the energy which depends only on $R$.
\end{proof}

\subsection{Sign conventions}

We adopt the sign convention that $H$ is negative when $L$ is convex; this ensures that $H$ behaves like $+\Delta$ in terms of sign. With regards to geometric formulas, this implies that the first variation of $\FS_L$ is $-H$. This is the opposite of the convention which makes balls positively curved.

We also take the convention that $\YP$ refers to the contact angle measured inside of the liquid phase, see \fref{intro-fig}. This implies the following relations:
\begin{align*}
    &\eta^{\mathrm{co}}_{\FS}\cdot \eta^{\mathrm{co}}_{\PS} := \cos\theta
    \\&\nu^f\cdot \eta^{\mathrm{co}}_{\PS} = \eta^{\mathrm{co}}_{\FS}\cdot \nu^c = -\sin\theta
    \\&\nu^f\cdot \nu^c = -\cos\theta
\end{align*}
Here, the $\eta$ are inward conormal vectors (normal to the contact line, tangent to the subscripted surface), while the $\nu$ are outward normal vectors.

\subsection{Parameters and ellipticity hypotheses}\label{s.ellipticity-hypothesis}

We assume throughout that $|\cos\YC| < 1$, $[ \cos\YP - \mu_+,\cos\YP  +\mu_-] \subset (-1, 1)$, $g > 0$, and $R > 1$.

\subsection{Outline of the paper} In \sref{basic_properties_glbl_stbl}, we show coercivity of the energy and develop basic estimates for how the energy behaves under perturbations of the profile. In \sref{viscosity}, we build on this by developing a viscosity theory for the Euler-Lagrange equation satisfied by globally stable profiles, which we apply to show an oscillation bound for the free surface of the liquid. In \sref{higher_regularity_glbl_stbl}, we use this oscillation bound to develop higher regularity of globally stable bounds, in the framework of perimeter almost-minimizers. In \sref{rate_independent_energy_solutions}, we develop some abstract lemmas for convergence to energy solutions. In \sref{discrete_scheme_and_convergence}, we show existence of energy solutions to both the finite container radius and infinite container radius evolutions. Finally, in \sref{container_limit}, we prove the energy asymptotics which let us show compactness of energy solutions in the $R\to \infty$ limit.

\subsection{Acknowledgments}

Both authors received support from the NSF Grant DMS-2407235. The first author also received support from the NSF Grant DMS-2153254.

\section{Basic properties of globally stable profiles}\label{s.basic_properties_glbl_stbl}

In this section we establish several basic regularity properties of globally stable states, along with some important calculus formulas which are useful in computing the time derivative of the energy along the evolution. We accomplish:
\begin{enumerate}
    \item A local coercivity property controlling the contact surface area by the free surface area up to a lower order term. As corollaries, we obtain a lower bound on the energy, lower semicontinuity of the energy, and existence of globally stable profiles.
    \item A derivation of the Euler-Lagrange equation for globally stable profiles.
    \item Taylor expansion of the energy of a profile under perturbations which modify the forcing $F$, or the volume. The introduction of the pressure.
\end{enumerate}

\subsection{Energy coercivity and existence of globally stable states}

We first recall the following estimate, which can be viewed as a trace inequality for functions of bounded variation. Our proof is similar in spirit to standard averaging techniques as in \cite{Giusti1984}*{Equation (2.11)}.

\begin{lemma}\label{l.bv_trace_ineq}
    Let $\Gamma$ be a bounded $C^2$ surface with unit normal $\nu_\Gamma$. Write $\Gamma_r$ for the one-sided neighborhood $\Gamma_r = \{ y + s\nu_\Gamma(y) : y\in \Gamma, s\in (0,r)\}$, with $r$ sufficiently small that the tubular coordinates $(y,s)$ are a diffeomorphism of $\Gamma\times (0,r)\to \Gamma_r$.

    Let $L\in \mathrm{Cacc}(\Gamma_r)$. Then there exists $C > 0$ such that for any $\beta\in (0,r)$
    \begin{equation}\label{eq:bv_trace_ineq}
        |\mathrm{tr}_\Gamma(L)| \leq |\partial^* L\cap \Gamma_\beta| + \frac{C}{\beta}\vol(L\cap \Gamma_\beta)
    \end{equation}
    In particular, both $C, r$ depend only on the $C^2$ character of $\Gamma$.
\end{lemma}
\begin{proof}
    We prove the corresponding estimate for smooth $BV$ functions, and then derive the statement by mollifying $\one_L$. Let $f\in C^\infty(\Gamma_r)\cap BV(\Gamma_r)$. Write $d(x)$ for the signed distance from $x$ to $\Gamma$; i.e. $d(y + s\nu_\Gamma(y)) = s$, using the uniqueness of the tubular projection. Note that $d$ is $C^2$ in $\Gamma_r$, with $\nabla d(\xi(y,s)) = \nu_\Gamma(y)$.

    Using divergence theorem on $\Gamma_t$ with $0 < t < r$, we have
    \[ \int_{\partial \Gamma_t} f(\nabla d\cdot \nu_{\Gamma_t}) = -\int_{\Gamma_t} \nabla\cdot(f\nabla d) \]
    with $\nu_{\Gamma_t}$ the inward unit normal, coinciding with $\nu_{\Gamma}$ on $\Gamma$. Now, observe that $\partial \Gamma_t = \{ y + s\nu_\Gamma(y) : y\in \partial \Gamma \hbox{ or } s\in \{0, t\}\}$. The $\partial \Gamma$ terms vanish from the above since $\nabla d$ is parallel to $\nu_\Gamma$ (hence tangential in that portion). Thus, we can rearrange as
    \[ \int_{\Gamma} f = \int_{\{y + t\nu_\Gamma(y) : y\in \Gamma\}} f -\int_{\Gamma_t} \nabla\cdot(f\nabla d) \]
    Then, we average in $t$ to get
    \[ \int_{\Gamma} f d\mathcal{H}^{d-1}(y)= \frac{1}{\beta}\int_0^\beta\int_{\{y + t\nu_\Gamma(y) : y\in \Gamma\}} fd\mathcal{H}^{d-1}(y)\,dt - \frac{1}{\beta}\int_0^\beta\int_{\Gamma_t} \nabla\cdot(f\nabla d)\,dxdt \]
    To bound the right-hand-side, we note that 
    \begin{align*}
        \frac{1}{\beta}\int_0^\beta\int_{\{y + t\nu_\Gamma(y) : y\in \Gamma\}} fd\mathcal{H}^{d-1}(y)\,dt &\leq \frac{J}{\beta}\int_{\Gamma_\beta} |f| dx
    \end{align*}
    where $J$ bounds the weight for the change of measure. For the next term, we have
    \[ \frac{1}{\beta}\int_0^\beta\int_{\Gamma_t} \nabla d\cdot \nabla f \leq \int_{\Gamma_\beta} |\nabla f| \]
    since $|\nabla d| = 1$ and $\Gamma_t \subset \Gamma_\beta$ for each $t < \beta$. For the last term, we use
    \[ \frac{1}{\beta}\int_0^\beta\int_{\Gamma_t} f\Delta d \leq \frac{C}{\beta}\int_{\Gamma_\beta} |f| \]
    where $C$ is a bound for $|\Delta d|$. Altogether,
    \[ \int_{\Gamma} f\,d\mathcal{H}^{d-1}(y)\leq \int_{\Gamma_\beta} |\nabla f|\,dx + \frac{C + J}{\beta}\int_{\Gamma_\beta} |f|\,dx \]
    Then, it is straightforward to replace $f$ by $|f|$ on the left hand side and extend to general $BV$ functions by regularization.
\end{proof}

A variant of \lref{bv_trace_ineq} which we will need concludes that given $L^1$-convergence of profiles, we have $L^1$-convergence of contact sets, provided that the perimeter of the free surface does not accumulate near the fixed boundary. We state this below:
\begin{lemma}\label{l.conv_of_contact_sets}
Let $L_n, L_\infty\in \mathrm{Cacc}(\Gamma_{r})$ such that
\[ \lim_{n\to\infty} \vol(L_n\Delta L_\infty) = 0 \]
and
\[ \lim_{t\to 0^+} \limsup_{n\to\infty} \mathrm{Per}(L_n; \Gamma_t) = 0 \]
Then
\[ \lim_{n\to\infty}|\mathrm{tr}_\Gamma(L_n)\Delta \mathrm{tr}_\Gamma(L_\infty)| = 0 \]
\end{lemma}
\begin{proof}
    It is straightforward that
    \[ |\mathrm{tr}_\Gamma(L_n)\Delta \mathrm{tr}_\Gamma(L_\infty)| = |\mathrm{tr}_\Gamma(L_n\Delta L_\infty)| \]
    By \lref{bv_trace_ineq}, we have for all $t \in (0,r)$ that:
    \begin{align*}
        |\mathrm{tr}_\Gamma(L_n\Delta L_\infty)| &\leq \mathrm{Per}(L_n\Delta L_\infty; \Gamma_t)+ \frac{C}{t}\vol((L_n\Delta L_\infty)\cap \Gamma_t)
        \\&\leq \mathrm{Per}(L_n; \Gamma_t) + \mathrm{Per}(L_\infty; \Gamma_t) + \frac{C}{t}\vol(L_n\Delta L_\infty)
    \end{align*}
    Sending $n\to \infty$, the volume term disappears and we are left with
    \[ \limsup_{n\to\infty}|\mathrm{tr}_\Gamma(L_n\Delta L_\infty)| \leq \mathrm{Per}(L_\infty; \Gamma_t) +\limsup_{n\to\infty} \mathrm{Per}(L_n; \Gamma_t) \]
    Sending $t\to 0^+$, we have that $\Gamma_t$ shrinks to the empty set, so the first term on the right side goes to $0$. The second term on the right goes to $0$ by the hypothesis, and we conclude.
\end{proof}

Now we return to the Wilhelmy plate setting. We will apply \lref{bv_trace_ineq} to conclude that the global energy $\mathcal{E}_R$ is bounded below.

\begin{lemma}\label{l.energy_bdd_below}
    For $L\in X^R_F$ such that $|\PS_L\Delta \Cyl_{R,0}|, |\CS_L\Delta \Cyl_{R,F}| < \infty$, the contact surface terms in the energy are controlled by the free surface and gravity terms, up to an additive constant:
    \begin{align*}
        &|\cos \YP| |\relAreaP{\PS_L}| + |\cos\YC||\relAreaC{\CS_L}|
        \\&\leq \left(|\cos\YP|\vee |\cos\YC|\right)\left(|\FS_L| + \int_{L \Delta \Cyl_{R,F}} g|z-F|\right)
        \\&\quad\dots+ C|\cos\YP|(1 + |F|) + C|\cos\YC|R^{d-2}.
    \end{align*}
    It follows that the energy is bounded from below
    \[\mathcal{E}_R[L;F] \geq -C|\cos\YP|(1 + |F|) -C|\cos\YC|R^{d-2}. \]
    Moreover, since $|\cos\YP|, |\cos\YC| < 1$ with strict inequality, the energy controls the area of the free surface and the gravitational energy up to additive and multiplicative constants:
    \begin{align*}
        &|\FS_L| + \int_{L\Delta \Cyl_{R,F}} g|z-F|
        \\&\leq \frac{1}{1 - (|\cos\YP|\vee |\cos\YC|)}\left(\mathcal{E}_R[L;F] + C|\cos\YP|(1 + |F|) + C|\cos\YC|R^{d-2}\right).
    \end{align*}
\end{lemma}
\begin{proof}

Using \lref{bv_trace_ineq}, we can locally estimate the area of the surfaces $\PS_L\setminus \Cyl_{R,0}, \Cyl_{R,0}\setminus \PS_L, \CS_L\setminus \Cyl_{R,F}$ and $\Cyl_{R,F}\setminus \CS_L$ in the same way. We cover these surfaces by nonoverlapping open sets $B_i$, chosen as uniform open cubes in tubular coordinates with the property that $B_i\cap (\partial B_1\times \{z = 0\}) = \emptyset$ and $B_i\cap \partial B_R\times \{z = F\}) = \emptyset$ for each $i$. For future use, we write $r$ for the coordinate side length of the $B_i$ in the normal direction. 

Now, each $B_i$ only hits one of the four surfaces we need to estimate, and we can apply the lemma for each $B_i$ with a fixed constant $C$. The resulting inequalities we obtain are:
\begin{equation}\label{eq:example_coercivity_inequality}
\begin{cases}
    |(\PS_L\setminus \thinCyl_{1,0})\cap B_i| \leq |\FS_L\cap B_i| + C\vol(L\cap B_i) \\
    |(\CS_L\setminus \thinCyl_{R,F})\cap B_i| \leq |\FS_L\cap B_i| + C\vol(L\cap B_i) \\
    |(\thinCyl_{1,0}\setminus \PS_L)\cap B_i| \leq |\FS_L\cap B_i| + C\vol(L^c\cap B_i) \\
    |(\thinCyl_{R,F}\setminus \PS_L)\cap B_i| \leq |\FS_L\cap B_i| + C\vol(L^c\cap B_i)
\end{cases}
\end{equation}
Now, we can sum these over the covering. For the free surface terms, we have
\[ \sum |\FS_L\cap B_i| \leq |\FS_L|  \]
since the covering is nonoverlapping. For the volume terms, let us consider for a moment just the $B_i$ which cover $\PS_L\setminus \thinCyl_{1,0}$. For any $h_0 > 0$, we can choose the covering so that it splits exactly into cubes contained in $\{ z - F \geq h_0 \}$ and cubes contained in $\{ z - F < h_0\}$. In the former region, we have
\[ \int_{(L\Delta \Cyl_{R,0})\cap \{ |z-F| \geq h_0\}} g|z-F| \geq gh_0\vol((L\Delta \Cyl_{R,0})\cap \{ |z-F| \geq h_0\}) \]
Therefore,
\begin{align*}
    \sum_{B_i\cap (\PS_L\setminus \thinCyl_{1,0})\cap \{ |z-F|\geq h_0\}\neq \emptyset} C\vol(L\cap B_i) &\leq C\vol((L\Delta \Cyl_{R,0})\cap \{ |z-F| \geq h_0\})
    \\&\leq \frac{C}{gh_0}\int_{(L\Delta \Cyl_{R,F})\cap \{ z-F \geq h_0\}} g|z-F|
\end{align*}
If we choose $h_0 > \frac{C}{g}$, then we obtain that the left hand side is controlled by the gravitational energy. The remaining region is $\{ 0 < z < F + h_0 \}$. Assuming we have chosen uniform cubes of length $r$ in the normal direction, the volume of this region is bounded by $Cr\mathcal{H}^{d-2}(\partial B_1)(F + h_0)_+$, for a constant $C$ that can be estimated in terms of the tubular parametrization.

The volume terms for the other three contact surfaces are estimated similarly, using gravity for balls far from $z = F$ and using the na\"ive estimate for the volume near $z = F$. The latter scales like $Cr\mathcal{H}^{d-2}(\partial B_R)h_0$ for the two outer contact surfaces with the same justification as before (the $F$ disappears since we work with respect to $\thinCyl_{R,F}$ rather than $\thinCyl_{1,0}$). We note for emphasis that $r, h_0$ are uniformly bounded with respect to $R$. 

Finally, we sum \eqref{eq:example_coercivity_inequality} with weights $|\cos\YP|$ and $|\cos\YC|$ for the inner and outer contact terms, to obtain:
\begin{align*}
    &|\cos\YP||\relAreaP{\PS_L}| + |\cos\YC|||\relAreaC{\CS_{L}}|
    \\&\leq|\cos\YP||\PS_{L}\Delta \thinCyl_{1,0}| + |\cos\YC||\CS_L\Delta \thinCyl_{R,F}| 
    \\&\leq \left(|\cos\YP|\vee |\cos\YC|\right)\left(|\FS_L| + \int_{L\Delta \Cyl_{R,F}} g|z-F|\right)
    \\&\quad\dots+C|\cos\YP|(|F| + 1) + C|\cos\YC|R^{d-2}
\end{align*}
With this, the other conclusions in the statement follow directly.
\end{proof}

\begin{lemma}[Local lower semicontinuity of energy]\label{l.local_lower_semicontinuity}
Fix $R\in (1, \infty]$. If $L_n\to L_\infty$ in $L^1_{\mathrm{loc}}$ with $L_n, L_\infty\in X^R_F$, then in any bounded set $U$ we have
\[ \mathcal{E}[L_\infty; U] \leq \liminf_{n\to\infty} \mathcal{E}[L_n; U] \]
and for any competitor $L_0\in X^R_F$, it holds that
\[ \mathcal{E}[L_\infty; U] + \Diss[L_0, L_\infty; U] \leq \liminf_{n\to\infty} \mathcal{E}[L_n; U] + \Diss[L_0, L_n; U] \]
\end{lemma}
\begin{proof}
We split into cases. Away from the container and plate, the convergence of the $L_n$ implies that the integrand $|z-F|\one_{L_n\Delta \Cyl_{R,F}}$ converges locally in $L^1$ to $|z-F|\one_{L\Delta \Cyl_{R,F}}$. Thus, we get lower semicontinuity of gravity from Fatou (which only requires local convergence in measure). The lower semicontinuity of the free surface energy follows from lower semicontinuity of relative perimeter. 

The situations near the container and near the plate are similar, so we focus on the latter (in the case with dissipation), following the approach in \cite{caffarelliMellet}*{Lemma 2}. We will apply \lref{bv_trace_ineq} with $\Gamma = \mathrm{tr}_{\thinCyl_1}(U)$. Up to subdivision and covering, it then suffices to prove the lower semicontinuity with $U$ replaced by the one-sided tubular neighborhood $\Gamma_r$ of $\Gamma$.

Suppose that lower semicontinuity fails in $\Gamma_r$. Then we may pass to a subsequence along which 
\[ \mathcal{E}[L_\infty; \Gamma_r] > \mathcal{E}[L_n; \Gamma_r] + \delta \]
for all $n$ for some $\delta > 0$. Ignoring gravity, which converges just by $|L_n\Delta L_\infty|\to 0$, we have
\begin{align*}
    |\FS_{L_n}\cap \Gamma_r| + \delta &\leq |\FS_{L_\infty}\cap \Gamma_r| - \cos\YP |\PS_{L_\infty}\cap \Gamma_r| + \cos\YP |\PS_{L_n}\cap \Gamma_r|
    \\&\dots + \Diss(L_0, L_\infty; \Gamma_r) - \Diss(L_0, L_n; \Gamma_r)
    \\&\leq |\FS_{L_\infty}\cap \Gamma_r| - \cos\YP |\PS_{L_\infty}\cap \Gamma_r| + \cos\YP |\PS_{L_n}\cap \Gamma_r|
    \\&\dots + \Diss(L_\infty, L_n; \Gamma_r)
    \\&\leq |\FS_{L_\infty}\cap \Gamma_r| + |\PS_{L_\infty}\Delta \PS_{L_n}|
\end{align*}
Here, in the second line we use the dissipation distance triangle inequality \eref{diss_triangle_inequality}, and in the third line we use the ellipticity hypotheses in \sref{ellipticity-hypothesis} to bound all contact surface terms. Now, from \eqref{eq:bv_trace_ineq}, we have for all sufficiently small $t$ that
\[ |\PS_{L_\infty}\Delta \PS_{L_n}| \leq |\FS_{L_\infty}\cap \Gamma_t| + |\FS_{L_n}\cap \Gamma_t| + C(t)\vol((L_\infty\Delta L_n)\cap \Gamma_t) \]
Putting this with $t < r$ into the previous inequality, we use lower semicontinuity of perimeter to obtain:
\begin{align*}
    |\FS_{L_\infty} \cap (\Gamma_r\setminus \Gamma_t)| + \delta &\leq \liminf_{n\to\infty}|\FS_{L_n}\cap (\Gamma_r\setminus \Gamma_t)| + \delta \\&\leq |\FS_{L_\infty}\cap \Gamma_r| + |\FS_{L_\infty}\cap \Gamma_t| 
    \\&\dots + \liminf_{n\to\infty} C(t)\vol((L_\infty\Delta L_n)\cap \Gamma_t)
\end{align*}
Since $L_n\to L_\infty$, we may drop the volume term. Thus, we reduce to
\[ \delta \leq 2|\FS_{L_\infty}\cap \Gamma_t| \]
As $t\to 0$, the right side decreases to 0, being a finite measure evaluated on a family of sets shrinking to empty. Then $\delta \leq 0$ is a contradiction, so we conclude.

\end{proof}

\begin{lemma}[Lower semicontinuity of energy]\label{l.lower_semicontinuity}
If $L_n\to L_\infty$ in $L^1_{\mathrm{loc}}$ with $L_n, L_\infty\in X^R_F$, then we have
\[ -\infty < \mathcal{E}_R[L_\infty; F] \leq \liminf_n \mathcal{E}_R[L_n; F] \leq +\infty \]
For any $L_0\in \mathrm{Cacc}_{\mathrm{loc}}$, we have
\[ -\infty < \mathcal{E}_R[L_\infty; F] + \Diss[L_0, L_\infty] \leq \liminf_n \left(\mathcal{E}_R[L_n; F] + \Diss[L_0, L_n] \right) \leq +\infty \]
\end{lemma}
\begin{proof}
    We cover $\Cyl_R$ by countably many disjoint bounded sets $U_n$. Near the plate and (if $R < \infty$) the container, we choose these as in the proof of \lref{energy_bdd_below} and write $C(U_n)$ for the additive constant such that $\mathcal{E}[L; U_n]+ C(U_n)$ is the contribution of $U_n$ to the relativized energy $\mathcal{E}_R[L; F]$. From the proof of \lref{energy_bdd_below}, we have:
    \begin{itemize}
        \item For all but finitely many $n$, for all $L\in X^R_F$, $\mathcal{E}[L; U_n]+ C(U_n) \geq 0$.
        \item On the finitely many remaining $n$, $\mathcal{E}[L; U_n]+ C(U_n)$ is bounded below by a negative constant over all $L\in X^R_F$.
    \end{itemize}
    Thus, we apply the local lower semicontinuity of the energy from \lref{local_lower_semicontinuity} in each set. We have
    \[ \mathcal{E}[L_\infty; U_k] \leq \liminf_{n\to\infty} \mathcal{E}[L_n; U_k] \]
    and
    \[ \mathcal{E}[L_\infty; U_k] + \Diss[L_0, L_\infty; U_k] \leq \liminf_{n\to\infty} \mathcal{E}[L_n; U_k] + \Diss(L_0, L_n; U_k) \]
    In both cases, we relativize the energies and sum in $k$, using the aforementioned bounds from \lref{energy_bdd_below} to ensure that this sum is well-defined in $(-\infty, \infty]$. Subsequently, we bring the summation inside the liminf to conclude.
\end{proof}

With the lower semicontinuity of the energy augmented by dissipation, it is straightforward to show existence of globally stable profiles by the direct method:

\begin{lemma}[Existence of globally stable profiles]\label{l.existence_of_minimizer}
    Let $R\in (1,\infty]$. For any $L_0\in \mathrm{Cacc}_{\mathrm{loc}}(C_R)$, there exists a minimizer of
    \begin{equation}\label{eq:global_stability_functional}
        X^R_F\ni L\mapsto \mathcal{E}_R[L; F] + \Diss[L_0, L]
    \end{equation}
    Any minimizing $L$ is a globally stable profile.
\end{lemma}
\begin{proof}
    We restrict to the case $R < \infty$. The case $R = \infty$ is similar, omitting the subproof to verify the volume constraint.

    Let $(L_n)$ be a minimizing sequence for \eqref{eq:global_stability_functional}. Then $\mathcal{E}_R[L_n; F]$ is uniformly bounded, so by \lref{energy_bdd_below}, we have $|\FS_{L_n}| = \mathrm{Per}(L_n; \Cyl_R)$ is uniformly bounded, and $\int_{L_n\Delta \Cyl_{R,F}} g|z-F|$ is uniformly bounded. From compactness of sets of finite perimeter, we can restrict to a subsequence (without relabeling) which converges in $L^1_{\mathrm{loc}}(\Cyl_R)$ and pointwise a.e. to some $L\in \mathrm{Cacc}_{\mathrm{loc}}(\Cyl_R)$. Using the pointwise convergence with Fatou, we have
    \[ \int_{L\Delta \Cyl_{R,F}} g|z-F| \leq \liminf_n \int_{L_n\Delta \Cyl_{R,F}} g|z-F| \]
    As a corollary, $|L\setminus \Cyl_{R,F}|$ and $|\Cyl_{R,F}\setminus L^c|$ are both finite. Furthermore, we observe that for any $h > 0$,
    \begin{align*}
        \limsup_n|L_n\Delta L| &= \limsup_n \bigg[|(L_n\Delta L)\cap \{ |z-F| \leq h\}| + |(L_n\Delta L)\cap \{ |z-F| > h \}|\bigg]
        \\&\leq \frac{C}{h}
    \end{align*}
    where we use the $L^1_{\mathrm{loc}}$ convergence on the portion in $|z-F| \leq h$, and the uniform bound on gravitational energy of $L_n$ and $L$ to control the portion in $|z-F| \geq h$. Sending $h\to\infty$, we get that $|L_n\Delta L|\to 0$, and thus
    \begin{align*}
        |\mathcal{C}_F[L]| &= ||L\setminus \Cyl_{R,F}| - |\Cyl_{R,F}\setminus L||
        \\&\leq \liminf_n||L_n\setminus \Cyl_{R,F}| - |\Cyl_{R,F}\setminus L_n|| + |L\Delta L_n|
        \\&= \liminf_n |\mathcal{C}_F[L_n]| + |L\Delta L_n|
        \\&=0
    \end{align*}
    so $L$ satisfies the volume constraint, and $L\in X^R_F$. Using \lref{lower_semicontinuity}, we conclude that $L$ minimizes \eqref{eq:global_stability_functional}. Global stability of $L$ follows directly from the dissipation triangle inequality \eref{diss_triangle_inequality}.
\end{proof}

\subsection{Euler-Lagrange equation}

\begin{lemma}\label{l.initial_local_regularity}
    Let $L$ be globally stable for forcing $F$. For all test vector fields $\xi\in C^\infty_c(\Cyl_R; \R^d)$, we have
    \begin{equation}
        \int_{\FS_L} (\nabla\cdot \xi - \nu^f\cdot D\xi\nu^f) + g(z-F)\xi\cdot\nu^f = \lambda_L\int_{\FS_L} \xi\cdot \nu^f
    \end{equation}
    Thus, $\FS_L$ has distributional mean curvature which satisfies the equation 
    \[ -H_{\FS_L} + g(z-F) = \lambda_L\] 
    where $\lambda_L\in \R$ is a Lagrange multiplier for the volume constraint.
\end{lemma}
\begin{proof}
    The distributional mean curvature can be checked analogous to (\cite{Maggi_2012}*{Theorem 17.20}), which treats volume-constrained perimeter minimizers among sets of finite perimeter.
\end{proof}

\subsection{Perturbations of globally stable states}

In this section, we estimate the energy cost of perturbing a state to satisfy a perturbed constraint. This estimate will be key in obtaining the energy-dissipation balance for limits of discrete-time schemes.  The ideas in this section are more or less following \cite{AlbertiDeSimone}*{Proposition 5.3}, though due to our different setting and constraint we must use a different family of perturbations.

Fix $L_0, F_0$ such that $L_0$ is an admissible volume-constrained state for forcing $F_0$. Fix $R_0$ such that $1 < R_0 < R$. Let $V\in C^\infty(B_R\setminus \overline{B_1})$, such that
\begin{equation}
\begin{cases}
    V\equiv 1 & \hbox{ on }B_R\setminus B_{R_0} \\
    V\equiv 0 & \hbox{ near } \partial B_1 \\
    \int_{B_{R_0}\setminus B_1} V = |B_{R_0}\setminus B_1|
\end{cases}
\end{equation}
We form the vector field $\mathcal{V}(x,z) = V(x)e_z$ in the vertical direction, and introduce the flow $\Phi_s = I + s\mathcal{V}$. Write 
\[ L_s = \Phi_s(L_0),\; F_s = F_0 + s \]
Note that if we take $V \equiv 1$ on $\R^{d-1}\setminus B_{R_0}$, then $V$ is admissible for every $R > R_0$.

The main result of this section is a Taylor expansion at $s=0$ with second order remainder for the energy along the flow $\Phi_s$. The pressure appears as the linear term.

\begin{proposition}\label{p.energy_taylor_exp_alt}
Let $L_0$ be admissible for $\mathcal{E}_R[\cdot; F_0]$ with $R\in (R_0, \infty]$ with finite energy, and satisfy that $\FS_{L_0}$ is bounded above and below in the vertical direction: there exists $h_0 > 0$ such that $\FS_{L_0}\subset B_R\times [F_0 - h_0, F_0 + h_0]$. We have the pressure formula
\begin{equation}\label{e.pressure_with_r0}
    P^*_R[L_0, F_0] := \left.\frac{d}{ds}\right|_{s=0}\mathcal{E}_R[L_s, F_s] = \int_{\FS_{L_0}\cap \Cyl_{R_0}} \nu^f_{L_0}\cdot D\mathcal{V}\nu^f_{L_0} + g(z-F_0)(\mathcal{V} - e_z)\cdot \nu^f_{L_0}
\end{equation}
Moreover, we have the Taylor estimate
\[ \left|\mathcal{E}_R[L_s; F_s] - \mathcal{E}_R[L_0; F_0] - P^*[L_0,F_0]s\right| \leq C|\FS_{L_0}\cap \Cyl_{R_0}|s^2 \]
which holds for all $|s| \leq \frac{1}{\|\nabla_x V\|_{C_0}}$ and for a $C = C(\|V\|_{C^1})(1+h_0)$.
\end{proposition}
\begin{remark}
    We will later show in \lref{qualitative_height_bound} that globally stable profiles satisfy the assumption that $\FS$ is bounded in the vertical direction, and in \lref{width-bound} that $h_0$ can be controlled uniformly for globally stable profiles.

    We will also later show the local perimeter bound \lref{lambda_and_energy_bound} for globally stable profiles. With that, both the pressure and second derivative bound are uniformly bounded for fixed $R_0$, over all $R\in (R_0, \infty]$, all $F_0$, and all $L_0$ globally stable for $\mathcal{E}_R[\cdot; F_0]$.
\end{remark}

We will establish this Taylor expansion by explicit computations involving the constraint function and the energy. First we show that the constraint is preserved under the flow.

\begin{lemma}\label{l.volume-fix-flow-alt}
    We have $\mathcal{C}_{F_s}[L_s] = \mathcal{C}_{F_0}[L_0] = 0$.
\end{lemma}
\begin{proof}
First, with regards to the constraint, we have
\[\frac{d}{dF}\mathcal{C}_F(L) = - \int_{B_R\setminus B_1} 1dx\]
and
\[\frac{d}{ds}\mathcal{C}_{F}(L_s) = \int_{\FS_{L_s}} Ve_z\cdot \nu^f_{L_s} dS \]
An integration by parts which we will use several times in the following gives that
\begin{equation}\label{eq:ez_dot_nu_IBP_alt}
    \int_{\FS_{L_s}} Ve_z\cdot \nu^f_{L_s} dS = \int_{L_s\Delta \Cyl_{R,0}} \nabla\cdot(V(x)|z|e_z) +  \int_{B_R\setminus B_1} V(x) \ dx = \int_{B_R\setminus B_1} V(x)\,dx
\end{equation}
Note that this does not use any property of $\FS_{L_s}$ or $V$ except that $V$ only depends on the horizontal variable. Combining the above computations,
\begin{align*}
    \frac{d}{ds}\mathcal{C}_{F_0 + s}[L_s] = \int_{B_R\setminus B_1} V(x) - 1\, dx = 0.
\end{align*}
\end{proof}

Next we write the energy, which involves surface integrals on $\FS_{L_s}$, using the change of variables theorem as surface integrals over $\FS_{L_0}$.
\begin{lemma}
For all $s\in \R$, we have
\[ \mathcal{E}[L_s; F_s; \Cyl_R\setminus \Cyl_{R_0}] = \mathcal{E}[L_0; F_0; \Cyl_R\setminus \Cyl_{R_0}] \]
and
\begin{align}\label{eq:energy_expansion_for_Ls_alt}
    \mathcal{E}[L_s; F_s; \Cyl_{R_0}] &= \int_{\Sigma_{L_0}\cap \Cyl_{R_0}} |\nu^f_{L_0} - s(\nu^f_{L_0}\cdot e_z)(\nabla_x V,0)| 
    \\&\dots+\int_{\Sigma_{L_0}\cap \Cyl_{R_0}} g\frac{|z + sV-F_s|^2}{2}(e_z\cdot \nu^f_{L_0})d\mathcal{H}^{d-1} \notag
    \\&\dots- \cos\YP\relAreaP{\PS_{L_0}} \notag
\end{align}
\end{lemma}
\begin{proof}
The first claim is trivial, since in $\Cyl_R\setminus \Cyl_{R_0}$, $\Phi_s(x) = x + se_z$. Thus, the area of the free surface does not change, and the changes in gravity and contact surface area cancel exactly with the effect of forcing.

Restricting to inside $\Cyl_{R_0}$, we first compute
    \[ D\Phi_s = I + sD\mathcal{V} = \begin{pmatrix}
    1 & 0 & \dots & 0 \\
    0 & 1 & \ddots & \vdots \\
    \vdots & \ddots & \ddots & \vdots \\
    s\partial_{x_1}V & \dots & s\partial_{x_n} V & 1
    \end{pmatrix}  = I+se_z \otimes (\grad_xV,0)\]
We observe that $\Phi_s^{-1} = \Phi_{-s}$, and also that $\mathcal{V}\circ \Phi_s = \mathcal{V}$ for all $s \in \R$ and the same for $D\mathcal{V}$. Hence, using the computations in \cite{Maggi_2012}*{Proposition 17.1}, we have (using $M^T$ to denote the transpose of $M$):
\begin{equation}\label{eq:perturbed_normal_alt}
    \nu^f_{L_s}(\Phi_s(x)) = \frac{(D\Phi_{-s})(\Phi_s(x))^T\nu^f_{L_0}(x)}{|(D\Phi_{-s})(\Phi_s(x))^T\nu^f_{L_0}(x)|} = \frac{\nu^f_{L_0}(x) - s(\nu^f_{L_0}\cdot e_z)(\nabla_x V,0)}{|\nu^f_{L_0}(x) - s(\nu^f_{L_0}\cdot e_z)(\nabla_x V,0)|}
\end{equation}
The Jacobian factor for $\Phi_{-s}$ restricted to $\FS_{L_s}$ is then
\begin{equation}\label{eq:perturbation_jacobian_alt}
    |(J_{\FS_{L_s}}\Phi_{-s})(\Phi_s(x))| = |(D\Phi_{-s})(\Phi_s(x))^T\nu^f_{L_0}(x)| = |\nu^f_{L_0}(x) - s(\nu^f_{L_0}\cdot e_z)(\nabla_x V,0)|
\end{equation}
With this Jacobian, we can use the area formula to compute
\begin{equation}\label{eq:sigmaLs-formula_alt}
|\FS_{L_s}| = \int_{\FS_0} |\nu^f_{L_0}(x) - s(\nu^f_{L_0}\cdot e_z)(\nabla_x V,0)|
\end{equation}
For the contribution of gravity, we compute
\begin{align*}
    \int_{L_s\Delta \Cyl_{R,F_s}} |z - F_s| &= \int_{L_s\Delta \Cyl_{R,F_s}} \nabla\cdot\left(\frac{(z-F_s)|z-F_s|}{2}e_z\right)
    \\&= \int_{\FS_{L_s}} \frac{|z - F_s|^2}{2}(e_z\cdot  \nu^f_{L_s})
    \\&= \int_{\FS_{L_0}} \frac{|z + sV(x) - F_s|^2}{2}(e_z\cdot \nu^f_{L_0})
\end{align*}
where we use \eqref{eq:perturbed_normal_alt} and \eqref{eq:perturbation_jacobian_alt} in the last line.

Finally, we observe that $\relAreaP{\PS_{L_s}} = \relAreaP{\PS_{L_0}}$, since $V$ vanishes near $\partial B_1$. Thus, we obtain the formula for $\mathcal{E}[L_s; F_s]$ in each region as stated in the lemma.
\end{proof}

\begin{proof}[Proof of \pref{energy_taylor_exp_alt}]
It suffices to explicitly compute $\frac{d}{ds} \mathcal{E}[L_s;F_0+s]$ at $s = 0$, and to show that $\frac{d^2}{ds^2} \mathcal{E}[L_s;F_0+s]$ is bounded as in the statement at each $s$. Since the energy outside $\Cyl_{R_0}$ is constant in $s$, we restrict to the expansion \eqref{eq:energy_expansion_for_Ls_alt}. Differentiating \eqref{eq:energy_expansion_for_Ls_alt} twice in $s$, we get
\begin{align*}
    \frac{d^2}{ds^2}\mathcal{E}[L_s; F_s] &= \int_{\Sigma_{L_0}\cap \Cyl_{R_0}} \frac{|\nabla_x V|^2 - \left(\frac{\nu^f_{L_0} - s(\nu^f_{L_0}\cdot e_z)(\nabla_x V,0)}{|\nu^f_{L_0} - s(\nu^f_{L_0}\cdot e_z)(\nabla_x V,0)|}\cdot (\nabla_x V,0)\right)^2}{|\nu^f_{L_0} - s(\nu^f_{L_0}\cdot e_z)(\nabla_x V,0)|} (\nu^f_{L_0}\cdot e_z)^2
    \\&+ \int_{B_{R_0}\setminus B_1} g(V-1)^2
\end{align*}
Here, we have again used the trick in \eqref{eq:ez_dot_nu_IBP_alt} to replace the integral over $\FS_{L_0}\cap \Cyl_{R_0}$ with an integral over $B_{R_0}\setminus B_1$ in the second term. If we take $|s| \leq \frac{1}{\|\nabla_x V\|_{C^0}}$, we have 
\[|\nu^f_{L_0} - s(\nu^f_{L_0}\cdot e_z)(\nabla_x V,0)|\geq |\nu^f_{L_0}\cdot e_z|\vee (1 - |\nu^f_{L_0}\cdot e_z|) \geq \frac{1}{2}\]
using that $(\nabla_x V,0)$ is horizontal for the first option of maximum in the lower bound, and using reverse triangle inequality for the second option. Consequently, we get the second derivative estimate
\[ \left|\frac{d^2}{ds^2}\mathcal{E}[L_s; F_s]\right| \leq 2|\FS_{LV}\cap \Cyl_{R_0}|\|\nabla_x V\|_{C^0}^2 + g|B_{R_0}\setminus B_1|(1 + \|V\|_{C^0}^2) \]

On the other hand, we can compute directly (cf. Maggi eqns 17.23 and 17.27) that
\begin{align*}
        \left.\frac{d}{ds}\right|_{s=0}|\FS_{L_s}| &= \int_{\FS_{L_0}\cap \Cyl_{R_0}} (\nu^f_{L_0}\cdot e_z)(\nabla_x V,0) \cdot \nu^f_{L_0} 
        \\&= \int_{\FS_{L_0}\cap \Cyl_{R_0}} \nu^f_{L_0}\cdot D\mathcal{V} \nu^f_{L_0} \tag{$D\mathcal{V} = e_z\otimes (\nabla_x V,0)$}
        \\&= \int_{\FS_{L_0}\cap \Cyl_{R_0}} -H_{\FS_{L_0}}(\nu^f_{L_0}\cdot \mathcal{V}).
\end{align*}
As for the gravity term we compute
\begin{align*}
    \left.\frac{d}{ds}\right|_{s=0}\int_{\FS_{L_0}\cap \Cyl_{R_0}} \frac{|z + sV - F_s|^2}{2}(e_z\cdot \nu^f_{L_0}) &= \int_{\FS_{L_0}\cap \Cyl_{R_0}} (z-F_0)(V-1)(e_z\cdot\nu^f_{L_0})
\end{align*}
\end{proof}

To analyze local regularity of globally stable states, we will need an estimate for the energy cost of fixing volume in a specified cylindrical region, holding forcing constant. We restrict to the case $R < \infty$, and consider $V_0\in C^\infty_c(B_R\setminus \overline{B_1})$ with $\int_{B_R\setminus B_1} V_0 = 1$, and set $\mathcal{V}_0(x,z) = V_0(x)e_z$. Note the distinction that the previous $\mathcal{V}$ was identically 1 near $\partial B_R\times \R$, whereas $\mathcal{V}_0$ identically vanishes; note also that $V_0$ integrates to 1 rather than averages to 1 as did $V$, to be more convenient for computing change in volume. We will repurpose the previous notation to this new setting, with $\Phi_s = I + s\mathcal{V}_0$ and $L_s = \Phi_s(L_0)$. Let us stress that we will not require $L_s$ (including $L_0$) to solve the volume constraint.

\begin{corollary}[Volume-correcting perturbations]\label{c.volume_only_energy_taylor_exp}
    Suppose that $\mathcal{E}_R[L_0; F_0] < \infty$, and that $\FS_{L_0}$ is bounded above and below in the vertical direction: there exists $h_0 > 0$ such that $\FS_{L_0}\subset B_R\times [F_0 - h_0, F_0 + h_0]$. Then for $|s| < \frac{1}{\|\nabla_xV_0\|_{C^0}}$, $\mathcal{E}_R[L_s; F_0]$ is smooth as a function of $s$, with the constraint moving by:
    \begin{equation}\label{eq:volume_flow_derivative}
        \frac{d}{ds}\mathcal{C}_{F_0}[L_s] \equiv 1
    \end{equation}
    and the energy moving at $s = 0$ by:
    \begin{equation}\label{eq:volume_only_weak_pressure}
        \left.\frac{d}{ds}\right|_{s=0}\mathcal{E}_R[L_s; F_0] = \int_{\FS_{L_0}} \nu^f_{L_0}\cdot D\mathcal{V}_0 \nu^f_{L_0} + g(z-F_0)\mathcal{V}_0\cdot \nu^f_{L_0}
    \end{equation}
    If the distributional mean curvature of $\FS_{L_0}$ exists as a Radon measure, then we can instead write this as:
    \begin{equation}\label{eq:volume_only_strong_pressure}
        \left.\frac{d}{ds}\right|_{s=0}\mathcal{E}_R[L_s; F_0] = \int_{\FS_{L_0}} (-H_{\FS_{L_0}} + g(z-F_0))\mathcal{V}_0\cdot\nu^f_{L_0}
    \end{equation}
    In particular, in the case that $L_0$ is globally stable for $\mathcal{E}_R[\cdot, F_0]$ (or at least agrees with a stable profile on the support of $\mathcal{V}_0$), this simplifies to:
    \begin{equation}\label{eq:volume_only_lagrange_multiplier}
        \left.\frac{d}{ds}\right|_{s=0}\mathcal{E}_R[L_s; F_0] = \lambda_{L_0}
    \end{equation}
    where $\lambda_{L_0}$ is the Lagrange multiplier in \lref{initial_local_regularity}.

    In any of the above cases, we can Taylor-expand around $s = 0$, with the remainder estimate:
    \[ \left|\mathcal{E}_R[L_s; F_0] - \mathcal{E}_R[L_0; F_0] - s\left.\frac{d}{ds}\right|_{s=0}\mathcal{E}_R[L_s; F_0]\right| \leq C|\FS_{L_0}\cap \mathrm{spt}(\mathcal{V}_0)|s^2 \]
    which holds for all $|s| \leq \frac{1}{\|\nabla_x V_0\|_{C_0}}$ and for a $C = C(\|V_0\|_{C^1})(1+h_0)$.
\end{corollary}
\begin{proof}
    We omit the derivations of \eqref{eq:volume_flow_derivative} and \eqref{eq:volume_only_weak_pressure}, which are essentially the same as those presented in \lref{volume-fix-flow-alt} and \pref{energy_taylor_exp_alt} respectively. To obtain \eqref{eq:volume_only_strong_pressure}, note that the surface divergence theorem gives
    \begin{equation}
        \int_{\FS_{L_0}} \nu^f_{L_0}\cdot D\mathcal{V}_0 \nu^f_{L_0} = -\int_{\FS_{L_0}}  H_{\FS_{L_0}}(\mathcal{V}_0\cdot \nu^f_{L_0})
    \end{equation}
    where there is no boundary term since $\mathcal{V}_0$ is compactly supported in the horizontal variable. The final equation then follows from substituting the definition of $\lambda_{L_0}$, and using the same integration by parts as in \eqref{eq:ez_dot_nu_IBP_alt} with the fact that $\lambda_{L_0}$ is constant.
\end{proof}

As an immediate application of \cref{volume_only_energy_taylor_exp}, we establish a non-sharp bound on the vertical oscillations of the surface $\FS_L$ of a globally stable state. We will use this to justify integration by parts and comparison arguments for the free surface.

\begin{lemma}\label{l.qualitative_height_bound}
    Let $L$ be globally stable for $\mathcal{E}_R[\cdot, F]$, with $R\in (1, \infty]$. Then $\FS_L$ is bounded in the vertical direction: there exists $h$ such that $\FS_L\subset \Cyl_{R,F+h}\setminus \Cyl_{R,F-h}$.
\end{lemma}
We postpone the proof, which follows \cite{AlbertiDeSimone}*{Proposition 5.4ii}, to Appendix \ref{app.initial_height_bound}. We will state and prove sharper quantitative estimates in \lref{width-bound} and \lref{oscillation_bound}. Both of these later results depend qualitatively on \lref{qualitative_height_bound}, through requiring boundedness in the vertical direction \emph{a priori}.

\section{Viscosity solution properties and height bound}\label{s.viscosity}

In this section we introduce a notion of viscosity solution for the Euler-Lagrange equation satisfied by globally stable states
\begin{equation}\label{e.PDE-inc}
    \begin{cases}
        -H_{\FS_L}(x,z) + g(z-F) = \lambda_L & x\in \FS_L\\
        \cos \YP-\mu_+\leq \cos \THP{L}(x) \leq \cos \YP+\mu_- & x \in \PL_L \\
        \cos \THC{L}(x) = \cos \YC & x \in \CL_L.
    \end{cases}
\end{equation}
We show that globally stable states indeed satisfy \eref{PDE-inc} in the viscosity sense.  No higher regularity beyond the $BV$ bounds and qualitative height bound is needed to justify the viscosity solution properties. In fact the pinned contact angle condition appearing in \eref{PDE-inc}, by itself, likely does not admit $C^1$-estimates for the contact line.

With the viscosity theory, we gain access to comparison and barrier arguments to characterize globally stable profiles. As an application, we use barrier arguments to obtain a quantitative height bound for globally stable profiles. This height bound, in conjunction with the local length bound of the contact line, implies a quantitative global length bound for the contact line. The length bound will be a key estimate for our analysis of the Lagrange multiplier $\lambda$ in \eref{PDE-inc} in \sref{lagrange-decomp}.

\subsection{Viscosity solutions of the capillary problem}Let $U\subset \Cyl_R$ be a smooth bounded open set. Say that $\Phi\subset U$ is a test region if $\FS_{\Phi} := \partial \Phi\cap U$, $\PS_{\Phi} := \partial \Phi\cap \thinCyl_1$, and $\Sigma^U_\Phi := \partial \Phi \cap (\partial U\setminus \partial C_R)$ are (possibly empty) smooth manifolds with boundary. We write $\gamma_\Phi := \FS_\Phi\cap \PS_\Phi \cap (\partial C_1\times \R), \CL_\Phi := \FS_\Phi \cap \PS_\Phi \cap (\partial C_R\times \R)$. Write $\cos\theta^c_{\Phi} = \nu^f_{\Phi}\cdot e_z$ along $\PL_\Phi$ for the contact angle between the free surface and $\partial B_1\times \R$ (resp. for the outer boundary).

\begin{definition}
    Say that a set $\Phi$ is a test region in $B_r(0)$ if, after a rotation, $\Phi$ can be written as a smooth subgraph
    \[\Phi \cap B_r(0) = \{(x',x_d) : x_d < \varphi(x')\}.\]
    In these coordinates the mean curvature is defined on $\partial \Phi$
    \[H_{\Phi}(x) = \grad \cdot ( \frac{\grad \varphi}{\sqrt{1+|\grad \varphi|^2}}).\]
    If $S$ is another smooth region and $x \in \partial \Phi \cap \partial S$ define $\theta^c_\Phi(x)$ to be the angle of contact between $\partial \Phi$ and $\partial S$.
\end{definition}
System
\begin{equation}\label{e.PDE}
    \begin{cases}
        -H_{\FS_L}(x,z) + g(z-F) = \lambda_L & x\in \FS_L\\
        \THG{L}(x) = \alpha & x \in \gamma_L
    \end{cases}
\end{equation}
well defined for smooth test regions. Say that $\Phi$ is a strict supersolution if each equality is a $<$, and vice versa for strict subsolution.

Given a region $L \subset \R^d$ define the upper and lower semicontinuous envelopes
\[L^* = \{ x \in \R^d: \exists y_n \in L, y_n \to x\}\]
and
\[L_* = \R^d \setminus (\R^d \setminus L)^*.\]

$\Phi$ touches $L$ from outside at $x$ if $L^*\subset \Phi$ and $x\in \partial \Phi\cap \partial L^*$. The touching is strict if $\partial \Phi \cap L^* = \{ x \}$.

$\Phi$ touches $L$ from inside at $x$ if $\Phi \subset L_*$ and $x\in \partial \Phi \cap \partial L_*$. The touching is strict if $\partial \Phi \cap L_* = \{ x \}$.

Let us remark that since we will be dealing with globally stable $L$, the local regularity and characterization of the singular set from \lref{initial_local_regularity} ensures that, with appropriate normalization of measure 0 sets, $L^* = \overline{L}, L_* = \mathrm{int}(L)$. As such, we will ignore this detail and refer to the set simply as $L$ in the arguments below.

\begin{definition}\label{def:visc_subsln_with_test_region}
    Say that $L$ is a viscosity subsolution of \eref{PDE} if, whenever $\Phi$ is a smooth test region touching $L^*$ from the outside in $\overline{U \setminus S}$ at some $(x_0,z_0) \in \partial L^*$ either
    \[-H_{\Phi}(x_0,z_0) + g(z_0-F) \leq \lambda\]
    or $(x_0,z_0) \in \gamma_\Phi$ and
    \[\THG{\Phi}(x) \geq \alpha.\]
\end{definition}
\begin{definition}\label{def:visc_supsln_with_test_region}
    Say that $L$ is a viscosity supersolution of \eref{PDE} if, whenever $\Phi$ is a smooth test region touching $L_*$ from the inside in $\overline{U \setminus S}$ at some $(x_0,z_0) \in \partial L_*$ either
    \[-H_{\Phi}(x_0,z_0) + g(z_0-F) \geq \lambda\]
    or $(x_0,z_0) \in \gamma_\Phi$ and
    \[\THG{\Phi}(x) \leq \alpha.\]
\end{definition}
\begin{definition}
    Say $L$ is a viscosity solution of \eref{PDE} if $L^*$ is a subsolution of \eref{PDE} and $L_*$ is a supersolution of \eref{PDE}.
\end{definition}

We will prove that globally stable states are viscosity solutions of \eref{PDE-inc}.

\begin{theorem}\label{t.globally-stable-viscosity}
    If $L$ is globally stable then $L$ solves \eref{PDE-inc} in the viscosity sense.
\end{theorem}

Viscosity solution theory is commonly used to study the mean curvature flow \cites{EvansSonerSouganidis,BarlesSonerSouganidis,EvansSpruck}. It is less frequently used in the context of regularity theory of minimal surfaces, however comparison properties have appeared in a few places in the minimal surface literature. Our proofs will be based on an idea of Caffarelli and Cordoba \cite{CaffarelliCordoba}. Savin \cite{Savin} used this concept in an alternative proof of De Giorgi's $\ep$-regularity theorem for minimal surfaces. De Philippis, Fusco and Morini \cite{DePhilippisFuscoMorini} have proved similar properties in the context of capillary boundary conditions, however their approach is a bit different and they only prove a slightly weaker viscosity solution property which is, nonetheless, well suited to the proof of $\ep$-regularity.

\subsection{Viscosity solution properties of globally stable states}\label{s.viscosity-property-minimizers}
Now we proceed to prove \tref{globally-stable-viscosity} via two Lemmas, proving first the interior PDE and then the angle condition at the contact line.

Set $d_\Phi(x) = \mathrm{dist}(x, \partial\Phi)$. Note that $\Delta d_\Phi(x) = \sum_i\frac{\kappa_i}{1 - \kappa_i d_\Phi}$, where $\kappa_i$ is the $i$th curvature of the level surface $d_\Phi^{-1}(d_\Phi(x))$ at $x$.

\begin{lemma}\label{l.viscosity-solution-interior}
    If $L$ is globally stable for $F$, and $\Phi$ touches $L$ from above at $(x_0, z_0)\in \FS_\Phi$ (resp. from below), then $-H_\Phi + g(z - F) \leq \lambda_L$ (resp. $\geq$).
\end{lemma}
\begin{proof}
The two cases are symmetric, so we will only treat the first one in the case of strict touching. Suppose that in some region $U$ avoiding the fixed boundary, we have a test region $\Psi$ which touches $L$ strictly from above at $(x_0, z_0)$. For $\alpha > 0$, set $L_\alpha = L\cap \{ d_\Phi \geq \alpha \}$. Due to the strict touching, we may set $\alpha$ sufficiently small and assume that $d_\Phi > \alpha$ outside $U$.

Subsequently, set $\widetilde{L}_\alpha = \Phi_{s}(L_\alpha)$, with deformation as in \cref{volume_only_energy_taylor_exp}, chosen such that $\mathrm{spt}(V)\cap U = \emptyset$ and with $s(\alpha) = |L\cap \{ d_\Phi < \alpha\}|$ chosen such that $\mathcal{C}_F[\widetilde{L}_\alpha] = \mathcal{C}_F[L] = 0$. Note that the density bound from \lref{initial_local_regularity} ensures that $s(\alpha) > 0$ for $\alpha$ sufficiently small.

Then we have
\begin{align*}
    0 &\leq \mathcal{E}_R[\widetilde{L}_\alpha; F] - \mathcal{E}_R[L; F] 
    \\&= \mathcal{E}_R[L_\alpha; F] - \mathcal{E}_R[L; F] +\lambda_L s + O(s^2)
    \\&= \mathcal{H}^{d-1}(\{ d = \alpha\} \cap L) - \mathcal{H}^{d-1}(\partial L \cap \{ d < \alpha\})
    \\&\quad \cdots-\int_{L\cap \{ d_\Phi < \alpha\}} g(z - F) +\lambda_L s + O(s^2)
\end{align*}
Here, the first line is the stability property of $L$, the second line comes from the Taylor expansion in \cref{volume_only_energy_taylor_exp}, and the last line follows from the particular truncation performed in the definition of $L_\alpha$. We note that the $g|z-F|$ term that usually appears in the relative gravitational energy is replaced by $g(z - F)$, assuming that for $\alpha$ sufficiently small $\{ d_\Phi < \alpha \}$ stays on one side of $\{ z = F\}$; in the case where $z_0 = F$, we can instead have that this term goes to 0.

To control the terms from the change of the free surface, we use:
\begin{align*}
    \mathcal{H}^{d-1}(\{ d_\Phi = \alpha\} \cap L) - \mathcal{H}^{d-1}(\partial L \cap \{ d_\Phi < \alpha\}) &\leq \int_{\partial(L\cap \{ d_\Phi < \alpha \})} \nabla d_\Phi\cdot \nu
    \\&= \int_{L\cap\{ d_\Phi < \alpha\}} \Delta d_\Phi
    \\&= \int_{L\cap \{ d_\Phi < \alpha\}} \sum_i \frac{\kappa_i}{1 - \kappa_i d_\Phi}
\end{align*}
Now, we divide the inequality by $s(\alpha) = |L\cap \{ d_\Phi < \alpha \}|$ and send $\alpha$ to 0. By strict touching, the sets $L\cap \{ d_\Phi < \alpha \}$ shrink to $(x_0, z_0)$, and the smoothness of $\Phi$ ensures that in the limit we obtain $0 \leq \lambda_L + H_\Phi - g(z_0 - F)$.
\end{proof}

\begin{lemma}\label{l.viscosity-solution-angle}
If $L$ is globally stable for $F$, and $\Phi$ touches $L$ from above (resp. from below) at $(x_0, z_0) \in \partial B_1\times \R$, then $\cos\THG{\Phi} \leq \cos\YP + \mu_-$ (resp. $\geq \cos\YP - \mu_+$). The same holds at $\thinCyl_R$ without friction terms and with $\cos\YC$ replacing $\cos\YP$.
\end{lemma}
\begin{proof}
Again the four cases are similar, so we only treat the first. Let $U$ be a region such that $\overline{U}$ has relatively open intersection with $\thinCyl_1$, and suppose that $\Phi$ touches $L$ from above at $(x_0, z_0)\in \thinCyl_1$. Once again, we assume the touching is strict, and we additionally assume that $\Phi$ is a strict supersolution away from the fixed boundary: $H_\Phi + g(z - F) < \lambda_L$ on $U$.

Then, as in the previous lemma, we create $L_\alpha$ by first intersecting with $\{ d_\Phi \leq \alpha \}$ inside $U$, and subsequently perturbing $L_\alpha$ to $\widetilde{L}_\alpha$, in some region disjoint from $U$, to fix the volume constraint. As before, applying \cref{volume_only_energy_taylor_exp}:
\begin{align*}
    0 &\leq \mathcal{E}_R[\widetilde{L}_\alpha; F] + \Diss[L ,\widetilde{L}_\alpha] - \mathcal{E}_R[L;F]\\
    &=\mathcal{E}_R[L_\alpha; F] + \Diss[L ,L_\alpha] - \mathcal{E}_R[L;F]+\lambda_L s + O(s^2)
    \\&= \mathcal{H}^{d-1}(\{ d = \alpha\} \cap L) - \mathcal{H}^{d-1}(\partial L \cap \{ d < \alpha\})
    \\& \quad \cdots-\int_{L\cap \{ d_\Phi < \alpha\}} g|z - F|
    \\& \quad \cdots+(\cos\YP + \mu_-)|\PS_L\cap \{ d < \alpha\}|-\lambda_L s + O(s^2)
\end{align*}
Now, when we use $d_\Phi$, we instead get
\begin{align*}
    \mathcal{H}^{d-1}(L\cap \{ d_\Phi = \alpha \}) &- \mathcal{H}^{d-1}(\FS_L\cap \{ d_\Phi < \alpha\}) 
    \\&\leq \int_{\partial (L\cap \{ d_\Phi < \alpha\})\setminus (\PS_L\cap \{d_\Phi < \alpha\})} \nabla d_\Phi \cdot \nu
    \\&=\int_{L\cap \{ d_\Phi < \alpha\}} \Delta d_{\Phi} -\int_{\PS_L\cap \{d_\Phi < \alpha\}} \nabla d_\Phi\cdot \nu
\end{align*}
The rest of the computation proceeds similarly to before. Since we assume that $-H_\Phi + g(z - F) > \lambda_L$, we can group those terms together even before sending $\alpha$ all the way to 0 and have that their contribution to the right-hand side of the inequality is negative (absorbing the $O(s^2)$ into the other terms). Then, we can divide by $|\PS_L\cap \{ d_\Phi < \alpha \}|$ and send $\alpha$ to 0. As $\alpha\to 0$, we have $\nabla d_\Phi\cdot \nu\to \cos\THG{\Phi}$ on $\PS_L\cap \{ d_\Phi < \alpha \}$. The limiting inequality we obtain is then
\[ 0 \leq -\cos\THG{\Phi} + \cos\YP + \mu_- \]

\end{proof}

\subsection{Vertical bound of the the free surface}
Now we prove a much more precise bound on the width of the free surface, and the flatness caused by gravity, by making use of the barrier properties established in \sref{viscosity-property-minimizers}.
\begin{lemma}\label{l.width-bound}
If $L$ is globally stable with height $F$ and multiplier $\lambda$ then for $R \geq R_0 \geq 1$
\[ |z - (F+\tfrac{\lambda}{g})| \leq Ce^{-cgr}+Ce^{-cg(R-r)} \ \hbox{ for all } \ (x,z) \in \partial L\]
for $C,R_0 \geq 1 \geq c>0$ depending only on the dimension, $1-|\cos \YP \mp\mu_\pm|$, and $1-|\cos \YC|$.
\end{lemma}

Note that the only dependence on the Lagrange multiplier $\lambda$ here is additive and so we will actually be able to use this estimate to obtain much sharper bounds on $\lambda$ and on the energy.

We will argue for the radially symmetric case in $\Cyl_R = (B_R \setminus B_1) \times \R$. This is not essential, but some geometric assumptions are indeed needed here. The same argument below will work, with minor modifications, in the case that $C_R = (RO \setminus K) \times \R$ as long as both $O$ and $K$ are strictly star-shaped with respect to the origin.
\begin{proof}
The argument is by a sliding comparison with explicit sub / supersolution barriers. We will just argue for the upper bound, the lower bound argument is symmetric.

For a subgraph
\[\cos \theta(x) = \frac{\partial_\nu u}{\sqrt{1+|\grad u|^2}} \ \hbox{ where $\nu$ is outer normal from $U$ into $S$}\]

We will choose a sliding family of axisymmetric subgraphs as our barriers, for $t \in \R$ call
\[\Phi_t = \{(x,z): \ z \leq \varphi(x)\} \ \hbox{ with }  \ \varphi(x) = F+t+\psi(|x|).\]
So, for $t \geq \lambda/g$
\[-H_{\Phi}+g(z-F) = -H(\psi)+g\psi+gt \geq [-H(\psi)+g\psi]+\lambda \]
If $\psi$ is chosen smooth and so that 
\[-H(\psi) + g \psi > 0\]
and
\[\frac{\partial_r\psi}{\sqrt{1+(\partial_r\psi)^2}}(1) \leq  - \cos \YP -\mu_-\ \hbox{ and } \ \frac{\partial_r\psi}{\sqrt{1+(\partial_r\psi)^2}}(R) \geq \cos \YC\]
then $\Phi_t$ will be a strict supersolution of \eref{PDE-inc} for all $t \geq \lambda$.  Since $z$ is bounded from above on $\partial L$ then $\Phi_t \supset L$ for sufficiently large $t>0$. Sliding $t$ downwards, and using \lref{viscosity-solution-interior} and \lref{viscosity-solution-angle}, we see that $L \subset \Phi_t$ for all $t \geq \lambda$.

We will choose $\psi$ of the form
\[\psi(r) = Ae^{-ar/A}+Be^{-b(R-r)/B}.\]
The goal is to choose $a$ large depending only on $1-\cos \YP - \mu_->0$ to achieve a supersolution condition at the contact line, and then $A$ larger depending on $a$ and on $g$ to achieve a supersolution condition on the interior.  We will do similarly for $b$ and $B$ depending on $1-\cos \YC >0$ and on $g$. Note that the (almost symmetric) subsolution bound would depend on $1+\cos \YP - \mu_+>0$ and $1+\cos \YC>0$.

Notice that
\[\psi(1) = -a +be^{-bR/B} \ \hbox{ and } \ \psi(R) = b - ae^{-aR/A}.\]
For $R$ sufficiently large, depending on $a$, $A$, $b$, and $B$ to be specified, then $\psi(1) \leq a/2$ and $\psi(R) \geq b/2$. Then call
\[  c(s) = \frac{s}{\sqrt{1+s^2}} \]
note that $c: \R \to \R$ is increasing with
\[\textup{range}(c) = (-1,1)\]
and
\[c(\partial_r\psi(1)) \leq c(-a/2) \ \hbox{ and } \ c(\partial_r\psi(R)) \geq c(b/2).\]
So we can choose $a$ and $b$ sufficiently large and positive, depending only on $1-\cos \YP>0$ and on $1-\cos \YC>0$ respectively, so that
\[-c(-a/2) >  \cos \YP + \mu_- \ \hbox{ and } \ c(b/2) > \cos \YC.\]
The graphical mean curvature operator in radial coordinates has the form
\begin{align*}
    \grad \cdot (\frac{\grad \varphi}{\sqrt{1+|\grad \varphi|^2}}) &= \partial_r(\frac{\partial_r \psi}{\sqrt{1+(\partial_r \psi)^2}})+\frac{n-1}{r}\frac{\partial_r \psi}{\sqrt{1+(\partial_r \psi)^2}}\\
    &=\frac{1}{\sqrt{1+(\partial_r\psi)^2}}\left[\left(1-\frac{\partial_r \psi^2}{1+\partial_r\psi^2}\right)\partial^2_{r}\psi+\frac{n-1}{r}\partial_r \psi\right].
\end{align*}
Since $\partial_r^2 \psi >0$ we can bound from above, on $r \geq 1$,
\begin{align*}
    H_{\Phi} &\leq \partial^2_r\psi + \frac{n-1}{r}|\frac{\partial_r\psi}{\sqrt{1+(\partial_r \psi)^2}}| \\
    &\leq (A^{-1}a^2+(n-1)a) e^{-ar/A}+(B^{-1}b^2 + (n-1)b ) e^{-b(R-r)/B} \\
    &\leq g\psi 
\end{align*}
as long as we choose $A$ and $B$ sufficiently large so that
\[(A^{-1}a^2 + (n-1)a ) \leq gA \ \hbox{ and } \ (B^{-1}b^2 + (n-1)b ) \leq gB .\]
\end{proof}

\subsection{Lagrange multiplier and energy bound}
As a corollary of the height bound \lref{width-bound}, we obtain bounds on the Lagrange multiplier and the energy of globally stable states. Note the bound on $\lambda_L$ will be used to obtain improved regularity of $L$, and is itself subsequently improved to $O(R^{-1})$ later in \lref{lambda_decomposition}.

\begin{lemma}\label{l.lambda_and_energy_bound}
    If $L$ is globally stable with height $F$ then
    \[|\lambda_L| \leq C_0g \ \hbox{ and } \ \mathcal{E}_R[L;F] \leq |B_R| + C_1(\mu_+\vee \mu_-)\]
    Moreover, for any open $U\Subset B_R\setminus \overline{B_1}$, we have the local perimeter bound
    \[ |\FS_L\cap (U\times \R)| \leq C_2(|U| + \mathcal{H}^{d-2}(\partial U)) \]
    Here, $C_0,C_1, C_2$ both depend only on $d, 1 - |\cos\YP \mp \mu_\pm|, 1 - |\cos\YC|$.
\end{lemma}
\begin{proof}
First we bound the Lagrange multiplier. The idea is that the width bound \lref{width-bound} shows that $L$ stays constant order distance from height $F + \frac{\lambda}{g}$, this will contradict the volume constraint, which centers $L$ around height $F$, if $\lambda$ is too large. Since $L\in X^R_F$, we have $\mathcal{C}_{R,F}(L) = 0$, which is to say:
\[|L \setminus \Cyl_{R,F}| = |\Cyl_{R,F} \setminus L|\]
in particular there exists $z_+ \in \partial L$ with $z_+ \geq F$ and there exists $z_- \in \partial L$ with $z_- \leq F$. Then the width bound \lref{width-bound} implies that
\[ F \geq z_- \geq F+\frac{\lambda}{g}-C \ \hbox{ and } \ F \leq z_+ \leq F + \frac{\lambda}{g} + C\]
and rearranging these we get $-gC \leq \lambda \leq gC$. Here $C \geq 1$ depends only on $d$, $1 - |\cos\YP \mp \mu_\pm|$, and $1 - |\cos\YC|$.

To prove the global energy bound, take $L':= \{z \leq F\}$ as a competitor. Due to \lref{width-bound} and using the Lagrange multiplier bound from above
\[\textup{Diss}(L,L') \leq (C+\frac{|\lambda|}{g})(\mu_+\vee \mu_-)\mathcal{H}^{d-2}(\partial B_1) \leq C(\mu_+\vee \mu_-)\]
for some $C \geq 1$ with the same dependence as before as long as $R\geq 2$.  Therefore
 \[\mathcal{E}_R[L; F] \leq \mathcal{E}_R[L'; F] + \Diss(L,L') \leq |B_R \setminus B_1| + C(\mu_+\vee \mu_-).\]
To prove the local energy bound, take
\[ L' := (L\setminus (U\times \R)) \cup (U\times (-\infty,a]) \]
where $a\in \R$ is chosen so that $\mathcal{C}_F[L'] = 0$. Due to \lref{width-bound} and the Lagrange multiplier bound, we have $|z - F| \leq C + \frac{|\lambda|}{g} \leq C$ for $(x,z)\in \FS_L$, so we have $|a - F| \leq C$. Then global stability gives
\begin{align*}
    0 &\leq \mathcal{E}_R[L'; F] - \mathcal{E}_R[L; F]
    \\&\leq |U| + |\mathrm{tr}_{\partial U \times (a,\infty)}(L)| + |\mathrm{tr}_{\partial U\times (-\infty,a]}(L^c)| - |\FS_L\cap (U\times \R)| 
    \\&\dots+ g|U|\frac{(a-F)^2}{2} - \int_{(L \cap (U\times\R))\Delta(U\times (-\infty, F])} g|z-F|
    \\&\leq |U| + 2C\mathcal{H}^{d-2}(\partial U) - |\FS_L\cap (U\times \R)|
    \\&\dots + g|U|\frac{(a-F)^2}{2}
\end{align*}
We remark that the second line is a strict inequality when $|\FS_L\cap (\partial K \times \R)| > 0$.

We bound the two trace terms using \lref{width-bound} and rearrange to obtain:
\begin{align*}
    &|\FS_L\cap (U\times \R)| + \int_{(L \cap (U\times\R))\Delta(U\times (-\infty, F])} g|z-F|
    \\&\leq \left(1 + g\frac{(a-F)^2}{2}\right)|U| + 2C\mathcal{H}^{d-2}(\partial U)
\end{align*}

\end{proof}
\subsection{Comparison principle for nonstrict super/subsolutions}

We finish the section with an application of the width bound to compare solutions of \eqref{e.PDE-inc} away from the pinned contact line.  Since we will now compare two viscosity solutions instead of a viscosity solution with a smooth test function, this requires justification. For this, it is standard that this type of comparison principle can be obtained by sliding arguments. We recall the viscosity theory for time-evolving level set functions in full detail in Appendix \ref{app.sliding_comparison}, based on results by Barles \cite{Barles} and Ishii and Sato \cite{IshiiSato}. The application of this theory will be in the following result:

\begin{proposition}[Sliding comparison]\label{p.sliding-comparison}
Fix $R_0\in (1, R)$. Let $L_0, L_1$ be viscosity solutions of
\begin{equation}
    \begin{cases}
        -H_{L} + g(z - F - \tfrac{\lambda_L}{g}) = 0 &\hbox{on }\FS_L\cap (\Cyl_R\setminus \Cyl_{R_0}) \\
        \cos\theta_{\CL_L} = \cos\YC &\hbox{on } \CL_L \\
        \thinCyl_{R_0,F + \lambda_L/g - h_0} \subset  \mathrm{tr}_{\thinCyl_{R_0}}(L) \subset \thinCyl_{R_0,F + \lambda_L/g + h_0}
    \end{cases}
\end{equation}
for a constant $h_0 > 0$ (chosen, for example, using \lref{width-bound}). Abbreviating $\tilde{L}_i = L_i\cap (\Cyl_R\setminus \Cyl_{R_0})$, we also assume that $\tilde{L}_i\Delta (\Cyl_{R,F}\setminus \Cyl_{R_0,F})$ is a bounded set for each $i=0,1$.

We then have
\[ (\tilde{L}_1 - te_z)^* \subset (\tilde{L}_0)_* \subset (\tilde{L}_0)^* \subset (\tilde{L}_1 + te_z)_*\]
for any $t > \max(\frac{|\lambda_{L_0} - \lambda_{L_1}|}{g}, 2h_0)$.
\end{proposition}
\begin{proof}
    From \lref{sliding-viscosity-soln}, the stationary indicator $u_0(t,x,z) = \one_{L_0^*}(x,z)$ is a viscosity solution of
    \[\begin{cases}
    \partial_t u_0 -\textup{tr}\left((I-\frac{\grad u_0 \otimes \grad u_0}{|\grad u_0|^2})D^2u_0\right) +g(z-F-\tfrac{\lambda_{L_0}}{g})|\grad u_0| \geq 0 & \hbox{on }\Cyl_R \\
    \partial_r u_0 = \cos \YC |\grad  u_0| & \hbox{on }\thinCyl_R \\
    \one_{\thinCyl_{R_0,F + \lambda_L/g - h_0}} \leq u_0 &\hbox{on }\thinCyl_{R_0}
    \end{cases}\]
    and the sliding indicator $u_1(t,x,z) = \one_{L_1^*}(x,z - f(t))$ with $f$ a smooth increasing function is a viscosity solution of
    \[\begin{cases}
    \partial_t u_1 -\textup{tr}\left((I-\frac{\grad u_1 \otimes \grad u_1}{|\grad u_1|^2})D^2u_1\right) +g(z-F-\tfrac{\lambda_{L_1} + gf(t) + f'(t)}{g})|\grad u_1| \leq 0 & \hbox{on }\Cyl_R \\
    \partial_r u_1 = \cos \YC |\grad  u_1| & \hbox{on }\thinCyl_R \\
    u_1 \leq \one_{\thinCyl_{R_0,F + \lambda_L/g + h_0 + f(t)}} &\hbox{on }\thinCyl_{R_0}.
    \end{cases}\]

    We apply \lref{width-bound}, and choose $f(0) \ll 0$ so that $L_1 - f(0)e_z \subset L_0$. We can achieve this, and reduce to a comparison in a bounded domain, by the hypothesis that $\tilde{L}_i\Delta (\Cyl_{R,F}\setminus \Cyl_{R_0,F})$ are bounded.  Subsequently, we would like to increase $f$ as much as possible while maintaining the comparison principle given by \lref{singular_degenerate_parabolic_comparison}, which preserves $u_1(t) \leq u_0(t)$ from the ordering of initial data. The capillary boundary condition on $\thinCyl_{R}$ is invariant under sliding, so this is a question of the interior equation and the Dirichlet ordering on $\thinCyl_{R_0}$. For the interior equation, we have
    \[ -\textup{tr}\left((I-\frac{\grad u_1 \otimes \grad u_1}{|\grad u_1|^2})D^2u_1\right) +g(z -F-\tfrac{\lambda_{L_0}}{g})|\grad u_1| \leq (\lambda_{L_1} - \lambda_{L_0} + gf(t) + f'(t))|\nabla u_1| \]
    so the right side is nonpositive as long as
    \[ gf(t) < \lambda_{L_0} - \lambda_{L_1} \]
    and $f'$ is taken sufficiently small.
    
    On $\thinCyl_{R_0}$, we have inclusion of boundary data as long as
    \[ f(t) + h_0 \leq - h_0 \]
    Combining the two estimates, we conclude the comparison can be extended to any
    \[ f(t) < \min(-2h_0, \tfrac{\lambda_{L_0} - \lambda_{L_1}}{g})  \]
    The reverse bound is similar, so we conclude.
\end{proof}

\section{Regularity of globally stable profiles}\label{s.higher_regularity_glbl_stbl}

In this section we establish basic regularity of globally stable profiles. As the main tool in doing so, we show that these profiles satisfy an \emph{almost-minimality property} involving the dissipation. We use this property to obtain nondegeneracy of globally stable profiles near their contact sets, and consequently we get length bounds on the contact lines $\PL, \CL$.  Together these regularity estimates imply a pre-compactness property of globally stable states under which many important quantities, including the energy and the pressure, converge.

Throughout this section, we focus on the local regularity of globally stable states, and write for a profile $L$ and a bounded set $U$ relatively open in $\overline{\Cyl_R}$,
\[\mathcal{E}[L; U] := |\FS_L\cap U| - \cos\YP|\PS_L\cap U| - \cos\YC|\CS_{L}\cap U| + \int_{(L\Delta \Cyl_{R,F})\cap U} g|z-F| \]
Note that we are suppressing the dependence on $R$ and $F$ from our notation, but both are still present in the energy. Throughout, we allow $R\in (1 + \varepsilon_0, \infty]$, and our universal results depend on $\varepsilon_0$.

\subsection{Almost-minimality and length bound}

\begin{definition}\label{def:almost_minimizer}
    We say that $L$ is an almost-minimizer if there exist parameters $\Lambda, r_0$ satisfying $\Lambda r_0 \leq 1$ and such that for any competitor $L'$ such that $L' \Delta L \Subset B_r(x_0)\cap \overline{\Cyl_{R_0}}$, for some $x_0\in \overline{C_{R}}, r < r_0$, we have
    \begin{equation}
        \mathcal{E}[L; B_r(x_0)] \leq \mathcal{E}[L'; B_r(x_0)] + \Diss[L, L'] + \Lambda\vol(L\Delta L')
    \end{equation}
\end{definition}
Note that this is equivalent to inwards and outwards almost-minimality for the functional $\mathcal{E}$ with the replacements $\cos\YP{} \mapsto \cos\YP{}+\mu_-$ (inwards) and $\cos\YP{} \mapsto \cos\YP{}-\mu_+$ (outwards); see \cite{DePhilippisMaggi}*{Definition 2.3}.

\begin{lemma}\label{l.almost_minimality}
    Let $L$ be globally stable for $\mathcal{E}_R[\cdot, F]$. Then $L$ is an almost-minimizer for constants $\Lambda, r_0$ which depend only on $d$, $1 - |\cos\YP \mp \mu_\pm|, 1 - |\cos\YC|$.
\end{lemma}
\begin{proof}
    We take $r_0 < 1$, to be determined later. Then given any $L'$ such that $L\Delta L' \Subset B_r(x_0)$ for some $x_0$ and $r < r_0$, we use \cref{volume_only_energy_taylor_exp} to correct the volume of $L'$, choosing $V_0$ for the perturbation specified in that result such that $(\mathrm{spt}(V_0)\times \R)\cap B_r(x_0) = \emptyset$ (which is possible since $r_0 < 1 < R$).

    Then global stability gives
    \begin{align*}
        \mathcal{E}_R[L; F] &\leq \mathcal{E}_R[\widetilde{L}; F] + \Diss[L, \widetilde{L}]
        \\&= \mathcal{E}_R[\widetilde{L}; F] + \Diss[L, L']
        \\&\leq \mathcal{E}_R[L', F]+ \Diss[L, L'] + \lambda_L(\vol(L\setminus L') - \vol(L'\setminus L))
        \\&\cdots + C|\FS_{L_0}\cap \mathrm{spt}(\mathcal{V}_0)|)(|\vol(L\setminus L') - \vol(L'\setminus L)|^2)
    \end{align*}
    We will choose $V_0$ so that $|\mathrm{spt}(V_0)| \leq C$, $\mathcal{H}^{d-2}(\mathrm{spt}(V_0)) \leq C$, and $\|V_0\|_{C^1} \leq C$, with $C$ independent of all parameters. Then by \lref{lambda_and_energy_bound}, we have $|\lambda_L| \leq C_0$ and $|\FS_{L_0}\cap \mathrm{spt}(\mathcal{V}_0)| \leq C_1$, for some constants $C_0, C_1$ which only depend on the parameters in the current lemma's statement.

    We now take $\Lambda = C_0 + 1$ and
    \[ r_0 = \min\left(\frac{1}{C_0 + 1}, \frac{1}{(CC_1 \omega_d)^{1/d}} \right) \]
    where $\omega_d$ is the volume of the $d$-dimensional unit ball. This ensures that $\Lambda r_0 \leq 1$, and that the quadratic remainder of the above Taylor expansion is bounded by $\vol(L\Delta L').$
\end{proof}

Next we combine local Hausdorff dimension estimates of the contact line for capillary almost minimizers, discussed in Appendix \ref{a.local-contact-line-reg}, with the width bound to obtain a global length bound on the contact line. This length bound plays an essential role in establishing compactness properties of our evolutionary solutions. 

\begin{theorem}[Length bound of the contact line]\label{t.length_bound}
    Let $L$ be globally stable for $\mathcal{E}_R[\cdot; F]$, for $R\in [2, \infty]$. Then 
    \[\mathcal{H}^{d-2}(\PL_L) \leq C\]
    for a constant $C$ depending only on $d$, $g$, $1-|\cos\YP\mp \mu_\pm|$, and $1-|\cos\YC|$. 
\end{theorem}
\begin{proof}
    By \lref{local_length_bound}, we can estimate $\mathcal{H}^{d-2}(\PL_L)$ by covering and using the local bound. The size of the balls in the cover depends only on the parameters listed, and consequently the number of balls in the cover is controlled by using the width bound \lref{width-bound} to trap $\PL_L$ in a region $\partial B_1\times [F - C_0, F + C_0]$, where $C_0$ depends on the same parameters.
\end{proof}

\subsection{Compactness and closure of almost-minimizers}

One of the benefits of working with almost-minimizers is that they enjoy better compactness properties than mere Caccioppoli sets. Most importantly, we have convergence of contact sets and convergence of energy for convergent sequences of uniform almost-minimizers. This is the content of the next result, which closely follows the proof of \cite{Maggi_2012}*{Theorem 21.14} for perimeter almost-minimizers. See also \cite{DePhilippisMaggi} which treats almost-minimizers of general capillary energies, but of a different form than our energy augmented by dissipation.

\begin{lemma}\label{l.almost-minimizer-cpt-closure}
If $(L_n)$ are almost-minimizers in a bounded, relatively open set $U\subset \overline{\Cyl_R}$, uniformly in $n$, then there exists a subsequence $L_{n_k}$ and a profile $L_\infty$ such that $L_\infty$ is an almost-minimizer in $U$ for the same uniform constants as the $L_n$. Moreover, inside any $V\subset U$ with $\mathrm{dist}(\partial V\cap C_R, \partial U\cap C_R) > 0$ and $\mathcal{H}^{d-1}(\FS_{L_\infty}\cap \partial V) = 0$, we have the following convergence results as $k\to\infty$:
\begin{enumerate}
    \item Local convergence of sets: $\vol((L_{n_k}\Delta L_\infty)\cap V)\to 0$
    \item Local convergence of contact sets: $|(\CS_{L_{n_k}}\Delta \CS_{L_\infty})\cap V|\to 0, |(\PS_{L_{n_k}}\Delta \PS_{L_\infty})\cap V|\to 0$
    \item Local convergence of free surface area: $|\FS_{L_{n_k}}\cap G|\to |\FS_{L_\infty}\cap G|$.
    \item Convergence of local energy: $\mathcal{E}_R[L_{n_k}; V] \to \mathcal{E}_R[L_\infty; V]$
\end{enumerate}
\end{lemma}
\begin{proof}
    From uniform almost-minimality and the boundedness of $U$, we have a bound on perimeter $\mathrm{Per}(L_n; U)$. It follows that the $L_n$ are compact in $L^1_{\mathrm{loc}}(U)$ by compactness of sets of finite perimeter. We pass to a locally convergent subsequence $L_{n_k}\to L_\infty$, using the notation from the statement. Using \lref{conv_of_contact_sets} along with the near-boundary perimeter estimate \lref{perimeter_bound_near_contact_line}, we also have local convergence of contact sets $\PS_{L_{n_k}}\to \PS_{L_\infty}$ and $\CS_{L_{n_k}}\to \CS_{L_\infty}$.

    To check that $L_\infty$ is an almost-minimizer, we follow the approach in \cite{Maggi_2012}. Consider a competitor $L'$ with $L_\infty\Delta L'\Subset B_r(x)\cap U$ for some $r < r_0$. Our goal is to show
    \[ \mathcal{E}[L_\infty; B_r(x)\cap U] \leq \mathcal{E}[L'; B_r(x)\cap U] + \Diss[L_\infty, L'] \]
    It suffices to show \begin{equation}\label{e.closure_of_almost_minimizer_goal}
        \mathcal{E}[L_\infty; G] \leq \mathcal{E}[L'; G] + \Diss[L_\infty, L']
    \end{equation}
    for a $G$ chosen so that $L_\infty\Delta L' \Subset G \Subset B_r(x)\cap U$. Note that we can find such a $G$ as a finite union of open balls with radii less than $r$ by compactness. Moreover, we can vary the radii of the balls to guarantee the following genericity conditions:
    \begin{enumerate}
        \item $|\partial B\cap (\partial^*L'\cup \partial^*L_n)|$ for each ball $B$ and each $n$ 
        \item $\liminf_{n\to\infty} |\partial B\cap $
    \end{enumerate}
    
    Given this $G$, we form a competitor $L_n'$ to $L_n$ by
    \[ L_n' := (L'\cap G)\cup (L_n\setminus G) \]
    Then almost-minimality of $L_n$ yields:
    \begin{align*}
        \mathcal{E}[L_n; B_r(x)\cap U] &\leq \mathcal{E}[L_n'; B_r(x)\cap U] + \Lambda\vol(L_n\Delta L_n') + \Diss[L_n, L_n']
    \end{align*}
    We observe that the energy and dissipation satisfy:
    \begin{align*}
        &\mathcal{E}[L_n'; B_r(x)\cap \overline{\Cyl_R}] + \Diss[L_n, L_n'] 
        \\&= \mathcal{E}[L'; G] + \mathcal{E}[L_n; (B_r(x)\cap \overline{\Cyl_R})\setminus G] + |\mathrm{tr}_{\partial G}(L_n\Delta L')| + \Diss[L_n, L'; G]
    \end{align*}
    Thus, the almost-minimality inequality reduces to
    \begin{align*}
        \mathcal{E}[L_n; G] &\leq \mathcal{E}[L'; G] + \Lambda\vol(L_\infty\Delta L') + \Diss[L_\infty, L'] 
        \\&\dots+\Lambda\left(\vol((L_n\Delta L_\infty)\cap G) - \vol(L_\infty\Delta L')\right)
        \\&\dots+\Diss[L_n,L'] - \Diss[L_\infty, L']
        \\&\dots + |\mathrm{tr}_{\partial G}(L_n\Delta L')|
    \end{align*}
    Now we send $n\to\infty$. The second line goes to 0 since $L_n\to L_\infty$. The third line goes to 0 by the convergence of contact sets $\PS_{L_n}\to \PS_{L_\infty}$. The fourth line goes to 0 by the choice of $G$. Using lower semicontinuity of the energy by \lref{local_lower_semicontinuity} to handle the left side, we obtain \eref{closure_of_almost_minimizer_goal} in the limit.
    
    It remains only to discuss the convergence of free surface area and of energy. Proceeding as in \cite{Maggi_2012}*{Theorem 21.14}, one can show that $|\nabla\one_{L_n}|\overset{*}{\rightharpoonup} |\nabla \one_{L_\infty}|$ inside $U$. This implies convergence of area of the free surface, again using the estimate from \lref{perimeter_bound_near_contact_line} to prevent accumulation of the perimeter near the fixed boundary.
\end{proof}

As an application of \lref{almost-minimizer-cpt-closure}, we state an improved compactness result for sequences of globally stable profiles. Here, the main point is the stability of the limiting profile.

\begin{lemma}\label{l.glbl_stbl_cpt}
    Suppose we have a sequence of profiles $(L_n)$ and forcing $F_n\to F_\infty$ such that the $L_n$ are globally stable for $\mathcal{E}_R[\cdot; F_n]$ for some $R\in (1, \infty]$ and such that the plate contact sets of the $L_n$ converge: i.e. there exists $\PS_\infty$ such that
    \[ \lim_{n\to\infty} |\PS_{L_n}\Delta \PS_{\infty}| = 0 \]
    Then the $L_n$ are precompact in $L^1_{\mathrm{loc}}$. All subsequential limits $L_\infty$ are in $X^R_{F_\infty}$ with $\PS_{L_\infty} = \PS_\infty$, are globally stable (with volume constraint for $F_\infty$, if $R < \infty$), and have the same energy $\mathcal{E}_R[\cdot; F]$.
\end{lemma}
\begin{proof}
    By \lref{almost_minimality}, the $L_n$ are uniform almost minimizers. Thus, we may apply \lref{almost-minimizer-cpt-closure} to extract a subsequence $L_{n_k}$ and a subsequential limit $L_\infty$. This lemma gives local convergence of the $\PS_{L_{n_k}}$ to $\PS_{L_\infty}$, so we conclude that $\PS_\infty = \PS_{L_\infty}$. Also, in the case $R < \infty$, we can combine with \lref{width-bound} to conclude that the $L_{n_k}$ converge in volume, and hence $L_\infty$ solves the volume constraint for $F_\infty$.
    
    The lemma also gives local convergence of energy. If $R < \infty$, this implies global convergence, by using \lref{width-bound} and the boundedness of the $F_n$ to trap the free surfaces in a compact region. In the case $R = \infty$, a modified proof of \lref{width-bound} shows that the free surfaces can be trapped in a region whose height shrinks exponentially as $|x|\to \infty$. Combining this with local estimates from almost-minimality, we can interchange limits by dominated convergence.

    Finally, it remains to show stability of $L_\infty$. As soon as we have stability, the fact that all limiting profiles have the same energy follows from the fact that they share $\PS_\infty$. Let us consider a competitor $L'\in X^R_{F_\infty}$. Then we can take 
    \[L_n' :=L' + (F_n - F_\infty)e_z\in X^R_{F_n}\]
    as a competitor to $L_n$ (satisfying the volume constraint if $R < \infty$). From stability of $L_n$, we have
    \begin{align*}
        \mathcal{E}_R[L_n; F_n] &\leq \mathcal{E}_R[L_n';F_n] + \Diss[L_n, L_n']
        \\&= \mathcal{E}_R[L'; F_\infty] -\cos\YP (F_n - F_\infty)\mathcal{H}^{d-2}(\partial B_1) + \Diss[L_n, L_n']
        \\&\leq \mathcal{E}_R[L'; F_\infty] + \Diss[L_\infty, L'] 
        \\&\dots+ C|F_n - F_\infty| + \Diss[L_n, L_\infty]
    \end{align*}
    Here, in the last step we have used the dissipation triangle inequality \eref{diss_triangle_inequality} twice to replace $\Diss[L_n, L_n']$, absorbing the resulting $\Diss[L', L_n']$ into the $|F_n - F_\infty|$ term. Both terms in the last line disappear as $n\to\infty$. On the other hand, the left side converges to $\mathcal{E}_R[L_\infty; F_\infty]$ as already shown, so we obtain the stability inequality for $L_\infty$ in the limit.
\end{proof}

\subsection{Continuity of the pressure}\label{s.pressure_continuous}

\begin{lemma}\label{l.pressure_continuous}

Suppose $(L_n)$ are globally stable for $\mathcal{E}_{R_n}[\cdot; F_n]$ with $R_n\in (1,\infty]$, such that $R_n\to R_\infty \in (1,\infty]$ and $F_n\to F_\infty$. If $L_n\to L_\infty$ in $L^1_{\mathrm{loc}}$, then $P^*[L_n; F_n]\to P^*[L_\infty; F_\infty]$, where $P^*$ is the pressure from \pref{energy_taylor_exp_alt} for a fixed parameter $R_0$ and vector field $\mathcal{V}$.
\end{lemma}

\begin{proof}
    We recall
    \[ P^*[L; F]  = \int_{\FS_{L}\cap C_{R_0}} \nu^f_{L}\cdot D\mathcal{V}\nu^f_{L} + g(z-F)(\mathcal{V} - e_z)\cdot \nu^f_{L} \]
    where the support of $\mathcal{V}$ avoids $\thinCyl_1$. Let us observe that the integral is uniformly bounded given a bound for $F$, by using \lref{lambda_and_energy_bound} to control the measure of the free surface and \lref{width-bound} to ensure that the integrand is bounded.

    Write $\mu_n$ (resp. $\mu_\infty$) for the vector measure $\one_{\FS_{L_n}\cap \mathrm{spt}(\mathcal{V}-e_z)}\nu^f_{\FS_{L_n}} d\mathcal{H}^{d-1}$. By \lref{almost_minimality} and \lref{width-bound}, these measures have support trapped in a compact region, and the free surfaces are uniform perimeter almost-minimizers in the standard sense. Therefore, we can apply standard closure properties for $L^1$-convergent almost-minimizers \cite{Maggi_2012}*{Theorem 21.14} to have $\mu_n\overset{*}\rightharpoonup \mu_\infty$ and $|\mu_n|\overset{*}\rightharpoonup |\mu_\infty|$.

    The convergence of the gravity term in the pressure follows immediately from the weak* convergence of vector measures. The convergence of the $\nu^f_{L_0}\cdot D\mathcal{V}\nu^f_{L_0}$ term comes from a theorem of Reshetnyak (\cite{Reshetnyak1968}, see also the appendix in \cite{LuckhausModica}), which relies on the fact that we also have $|\mu_n|\overset{*}\rightharpoonup |\mu_\infty|$ and thus $|\mu_n|(U)\to |\mu_\infty|(U)$ for any $U$ with $|\partial U\cap \FS_{L_\infty}| = 0$.
\end{proof}

\section{Limits of rate-independent systems}\label{s.rate_independent_energy_solutions}

In this section we introduce some of the foundational elements which are used in studying the limits of rate independent systems. We recall for the convenience of the reader that we work with energetic rate-independent systems which satisfy the properties of Definition \ref{def:constrained_energy_sln} ($R < \infty$) or Definition \ref{def:limiting_energy_sln} ($R = \infty$). The results in this section will form the backbone of the arguments for later sections. In particular, in \sref{discrete_scheme_and_convergence} we will discuss precompactness of discrete minimizing movements schemes which approximate energy solutions to the evolution. In \sref{container_limit}, we will discuss the precompactness as $R \to \infty$ of energy solutions.

This section is mainly recalling results from the literature, especially \cite{AlbertiDeSimone}, and adapting them to our setting.  We summarize the framework in \cite{AlbertiDeSimone} condensing it into a clear list of criteria for convergence in \lref{energy_diss_lower_bound}. The ideas presented here should also be closely related to general evolutionary $\Gamma$ convergence type frameworks for limits of rate independent systems, see the \cite{mielke2015book} for further description and citations. However due to our nonlinear geometric setting where the forcing acts through a constraint we are unable to conveniently apply those abstract results set in linear spaces.

This section is organized as follows. First, we discuss the energy dissipation upper bound, the opposite inequality to \eref{constrained_edi}. This is a fact which depends only on global stability, and plays a key role in proving convergence results for the energy dissipation lower bound. Next, we discuss the time compactness of evolutions given by dissipation bounds. Hand-in-hand with this, we address lack of evolutionary compactness for the nondissipative part of the evolution, and the measurable selection techniques that we employ to work around this. Finally, the main result of this section is a proof of the energy dissipation lower bound for a limit of approximate solutions satisfying appropriate assumptions.

\subsection{Energy dissipation upper bound}\label{s.energy_diss_upper_bound}
Here we show that general paths of globally stable states satisfy an upper bound on the energy dissipated.

\begin{lemma}[Reverse EDI]\label{l.reverse_edi}
Fix $R\in (1, \infty]$, and suppose $L(t)$ is a Borel measurable map $[0,T]\to L^1(\R^d)$, such that $L(t)$ is $\mathcal{E}_R[\cdot; F(t)]$-globally stable for each $t$. Then we have the energy-dissipation upper bound:
\begin{equation}
    \mathcal{E}_R[L(0); F(0)] - \mathcal{E}_R[L(0); F(T)] + \int_{t_0}^{t_1} P^*_R[L(t); F(t)]\dot{F}(t)\,dt \leq \overline{\Diss}[L; [0, T]]
\end{equation}
where we recall that $\overline{\Diss}$ refers to the total variation of the specified path with respect to the dissipation distance. Here, $P^*_R$ is the pressure defined in \pref{energy_taylor_exp_alt}.
\end{lemma}
\begin{proof}
We follow the approach in \cite{AlbertiDeSimone}*{Lemma 4.4}.

Let $0 = t_0 < \dots < t_N = T$ be a $\delta$-fine and $\delta/2$-separated partition of $[0,T]$, where $\delta$ is sufficiently small to apply \pref{energy_taylor_exp_alt} for each pair of adjacent times. Let $\Phi_s$ be the volume altering flow as in the proof of \pref{energy_taylor_exp_alt}.

Then using global stability of $L(t_n)$ and the Taylor expansion about $L(t_{n+1})$, we have
\begin{align*}
    0 &\leq \mathcal{E}_R[\Phi_{F(t_n) - F(t_{n+1})}(L(t_{n+1})); F(t_n)] + \Diss[L(t_{n}), L(t_{n+1})] - \mathcal{E}_R[L(t_n); F(t_n)]
    \\&= \mathcal{E}_R[L(t_{n+1}); F(t_{n+1})] - \mathcal{E}_R[L(t_n); F(t_n)]
    \\& \quad \dots+ P^*_R[L(t_{n+1}); F(t_{n+1})]\dot{F}(t_{n+1})(t_n - t_{n+1})  + O(|t_n - t_{n+1}|^2)
    \\& \quad \dots+ \Diss[L(t_n), L(t_{n+1})]
\end{align*}
We note that the quadratic remainder term is controlled independent of $L(t)$ using the perimeter bound from \lref{lambda_and_energy_bound} and the height bound from \lref{width-bound}. Summing, we get
\begin{align*}
    \overline{\Diss}[L; [0,T]] &\geq \sum_{n=0}^{N-1} \Diss[L(t_n), L(t_{n+1})] 
    \\&\geq \mathcal{E}_R[L(0); F(0)] - \mathcal{E}_R[L(T); F(T)] 
    \\&\quad\dots+ \sum_{n=0}^{N-1} P^*_R[L(t_{n+1}); F(t_{n+1})]\dot{F}(t_{n+1})(t_{n+1} - t_{n}) 
    \\&\quad\dots+ NO(\delta^2)
\end{align*}
Note that only the last two lines depend on the choice of partition. The term $P^*_R[L(t); F(t)]\dot{F}$ is a measurable function of $t$, which can be bounded pointwise by another application of \lref{lambda_and_energy_bound}, so the theory of Lebesgue integration provides that we can choose a family of partitions parametrized by $\delta$ such that
\[ \lim_{\delta\to 0^+}\sum_{n=0}^{N-1} P^*_R[L(t_{n+1}); F(t_{n+1})]\dot{F}(t_{n+1})(t_{n+1} - t_{n})\to \int_{0}^T P^*[L(t); F(t)]\dot{F}dt \]
On the other hand, the quadratic remainder term vanishes as $\delta\to 0$ since $N \leq 2/\delta$ by assumption. Thus, we obtain the inequality in the limit.
\end{proof}

\subsection{Time compactness of contact sets and measurable selection of profiles for contact sets}\label{s.time_compactness_and_msb_sel}

We recall a version of Helly's selection principle, the basic compactness result for bounded variation paths into a metric space, adapted to our setting:

\begin{lemma}[Time compactness of contact sets / Helly's selection]\label{l.helly}
    Let $\PS_n(t)$ be a sequence of paths of contact sets with corresponding contact lines $\PL_n(t)$. If 
    \[\sup_n\overline{\Diss}(\PS_n; [t_0, t_1]) < \infty \]
    and
    \[  \sup_n\sup_t \sup_{(x,z)\in \PL_n(t)} |z| < +\infty\ \hbox{ and } \ \sup_n\sup_t \mathcal{H}^{d-2}(\PL_n(t)) <+\infty \]
    then there exists a subsequence along which $\PS_{n_k}(t)$ converges pointwise on $[t_0, t_1]$ in the $\Diss$ quasimetric.
\end{lemma}

Here, the notation $\overline{\Diss}$ refers to the total dissipation of a path
\begin{equation}
    \overline{\Diss}(\PS(t); [t_0, t_1]) = \sup_{\tau_0 < \tau_1 < \dots < \tau_n} \sum_{i=0}^{n-1} \Diss(\PS(\tau_{i}),  \PS(\tau_{i+1}))
\end{equation}
where the supremum is over all partitions of $[t_0, t_1]$. The assumed height and length bounds restrict us to a space which is compact with respect to the dissipation distance, by $L^1$ pre-compactness of sets of bounded perimeter in a bounded region. We will almost always verify these bounds using \lref{width-bound} and \tref{length_bound} respectively.

We will use Helly's selection principle to get evolutionary compactness of energy solutions, or discrete approximations thereof. However, this leads to an important technical issue which must be addressed:

\begin{remark}\label{rem:msb_sel}
The natural mode of convergence of energy solutions is pointwise in time with respect to the dissipation distance. As a result, this gives an evolutionary notion of convergence for the contact set, but not for the profile itself. This is possible due to the lack of unique extension of a given contact set to a globally stable free surface, and the lack of frictional dissipation associated with motions of the free surface. We can identify a limiting profile by choosing a convergent subsequence at any given time, but it is not necessarily possible to choose a subsequence for which the profile converges for every time. This kind of issue has been discussed before in the literature, in particular it also appears in \cite{AlbertiDeSimone} and we can follow a similar approach. We ensure that the selection of profiles along a time-dependent subsequence is measurable, and so we will treat this selection as the limiting evolution.
\end{remark}
The measurable-selection result we quote to justify Remark \ref{rem:msb_sel} is:
\begin{lemma}[Measurable selection, (\cite{srivastava1998course} Srivastava 5.2.5)]\label{lem:msb_sel}
Let $T$ be a measurable space, and $Y$ be a separable metric space. If $F:T\to Y$ is a measurable, compact-valued multifunction, then there exists a measurable selection of $F$: a measurable function $f:T\to Y$ such that $f(t)\in F(t)$ for each $t\in T$.

Here, the preimage of the multifunction, for purposes of measurability, is 
\[ F^{-1}(A) := \{ t\in T : F(t)\cap A\neq \emptyset \} \]
\end{lemma}

We present a corollary of Lemma \ref{lem:msb_sel} in the appendix as \lref{subseq_msb_sel}, to justify the claim that measurable selection of subsequential limits (or alternatively of subsequential limits extremizing a measurable functional) is possible.

Thus, it is possible in principle to select an evolution of profiles, corresponding to the evolution of contact sets. The final and most involved proof in this section will be to show that the selection can be made to satisfy the energy dissipation balance.

\subsection{Energy dissipation lower bound}\label{s.energy_diss_lower_bound}
Here we show that limits of sequences of paths which satisfy an approximately optimal energy dissipation lower bound (EDI) will also satisfy the EDI (and therefore the full dissipation balance as well). This is technically complicated by the fact that motions of the free surface carry no dissipation and therefore there is very little time compactness. As discussed already in \sref{time_compactness_and_msb_sel} the set of globally stable states $\mathcal{M}(\PS)$ with a given contact set $\PS$ may be nontrivial and so a measurable selection needs to occur at the limit level. This is additionally complicated by the fact that the pressure may not be constant on the set $\mathcal{M}(\PS)$, so it is not clear that all measurable selections of the path actually do satisfy the EDI. Instead we show that, by performing a selection of a restricted subset of the minimizers attained along good subsequences, there is at least one selection which does satisfy the EDI.

We state the assumptions of the next lemma in a very general setting, to allow us to apply it both to the discrete scheme and to limits of continuous-time solutions.

\begin{lemma}[Stability of the energy-dissipation balance]\label{l.energy_diss_lower_bound}
Let $[0, T] \ni t \mapsto L_n(t)\in  X^{R_n}_{F_n(t)}$ be a family of Borel-measurable paths, such that
\begin{enumerate}[label=(\roman*)]
    \item\label{hyp:edlb_r}  (Convergence of $R_n$) $R_n, R_\infty\in (1,\infty]$ with $\lim_{n\to\infty}R_n = R_\infty$.
    \item\label{hyp:edlb_f}  (Convergence of $F_n$) The $F_n(t)$ are bounded in $BV[t_0, t_1]$ and converge uniformly to a $C^1$ limit $F_\infty(t)$ (but $F_n$ may be discontinuous).
    \item\label{hyp:edlb_contact}  (Convergence of $\PS_{L_n}$) The family of paths $\Sigma_{L_n}^c(t)$ converges pointwise in time as $n\to\infty$ and has uniformly bounded total dissipation \[ \sup_n \overline{\Diss}(L_n; [0, T]) < \infty.\]
    \item\label{hyp:edlb_profile} (Pointwise compactness of $L_n$) For each $t$, the sequence $(L_n(t))$ is compact in $L^1_{\mathrm{loc}}$, with all subsequential limits contained in $X^{R_\infty}_{F_\infty(t)}$ and globally stable for $\mathcal{E}_{R_\infty}[\cdot, F_\infty(t)]$.
    
    \item\label{hyp:edlb_energy} (Relative energy convergence) For any $t_0, t_1$, and any subsequential limits $L_\infty(t_i)$ of the $L_n(t_i)$ for $i=0,1$, it holds
    \begin{align*}
        &\lim_{n\to\infty} \mathcal{E}_{R_{n}}[L_{n}(t_0); F_{n}(t_0)] - \mathcal{E}_{R_{n}}[L_{n}(t_1); F_{n}(t_1)] 
        \\&= \mathcal{E}_{R_{\infty}}[L_{\infty}(t_0); F_{\infty}(t_0)] - \mathcal{E}_{R_{\infty}}[L_{\infty}(t_1); F_{\infty}(t_1)]
    \end{align*}
    Note that the right-hand-side is independent of the choice of subsequential limit, since all subsequential limits are globally stable for the same contact set, hence have the same energy.
    
    \item\label{hyp:edlb_pressure}(Dominated approximation of pressure-forcing term) There exist functions $Q_n$ on $[0,T]$ which are dominated by an integrable $\hat{Q}$ and satisfy the next two properties. For any $t$ and for any subsequence $L_{n_k}(t)$ which converges in $L^1_\mathrm{loc}$ to some $L_\infty(t)$, it holds that
    \[ \lim_{k\to\infty }Q_{n_k}(t) = P^*_{R_\infty}[L_\infty(t), F_\infty(t)]\dot{F}_\infty(t) \]
    Moreover, the energy dissipation inequality holds in a limiting sense with the $Q_n$ replacing the pressure-forcing term: for any $(t_0, t_1)\subset [0, T]$, it holds that
    \begin{align*}
        \liminf_{k\to \infty} &\left[\mathcal{E}_{R_{n_k}}[L_{n_k}(t_0); F_{n_k}(t_0)] - \mathcal{E}_{R_{n_k}}[L_{n_k}(t_1); F_{n_k}(t_1)] + \int_{t_0}^{t_1} Q_{n_k}(t)dt \right]
        \\&\geq \liminf_{k\to\infty} \overline{\Diss}(L_{n_k}; [t_0, t_1])
    \end{align*}
\end{enumerate}
Under these assumptions, there exists a Borel-measurable selection $L(t)$ of subsequential limits of the $L_n(t)$ such that $L(t)$ is an energy solution as per Definition \ref{def:constrained_energy_sln}.
\end{lemma}
\begin{proof}
For any selection of subsequential limits $L$, we have that $L(t)\in X^{R_\infty}_{F_\infty(t)}$ and $L(t)$ is globally stable for $\mathcal{E}_{R_\infty}[\cdot, F_\infty(t)]$ by assumption. Thus, our task is to show that a measurable selection exists which satisfies the energy dissipation balance
\[ \mathcal{E}_R[L(t_0); F(t_0)] - \mathcal{E}_R[L(t_1); F(t_1)] + \int_{t_0}^{t_1} P^*_{R}[L(t); F(t)]\dot{F} = \overline{\Diss}(L; [t_0, t_1]). \]
We check this following the proof of \cite{AlbertiDeSimone}*{Theorem 4.3(viii)}. For each time $t$, we restrict to a subsequence along which $\limsup_n Q_n(t)$ is achieved and use the compactness assumption to identify a corresponding limiting profile $\overline{L}(t)$. By Lemma \ref{lem:msb_sel} and \lref{subseq_msb_sel}, the path $\overline{L}(t)$ can be selected to be Borel measurable. From assumption \ref{hyp:edlb_pressure}, we conclude $\limsup_n Q_n(t) = P^*_{R_\infty}[\overline{L}(t); F_\infty(t)]\dot{F}_\infty$.

By \lref{reverse_edi}, the reverse energy dissipation inequality for measurable paths of globally stable states, we obtain
    \[ \mathcal{E}_R[\overline{L}(t_0); F(t_0)] - \mathcal{E}_R[\overline{L}(t_1); F(t_1)] + \int_{t_0}^{t_1} P^*_{R}[\overline{L}(t); F(t)]\dot{F} \leq \overline{\Diss}(\overline{L}; [t_0, t_1]). \]
We deduce
    \begin{align}\label{eq:edlb_computation}
        \mathcal{E}_{R_\infty}&[\overline{L}(t_0); F_\infty(t_0)] - \mathcal{E}_{R_\infty}[\overline{L}(t_1); F_\infty(t_1)] + \int_{t_0}^{t_1} P^*_{R_\infty}[\overline{L}(t); F_\infty(t)]\dot{F}_\infty \notag
        \\&\leq \overline{\Diss}(\overline{L}; [t_0, t_1]) \notag
        \\&\leq \liminf_n \overline{\Diss}(L_n, [t_0, t_1]) \notag
        \\&\leq \liminf_n\left[\mathcal{E}_{R_n}[L_n(t_0); F_n(t_0)] - \mathcal{E}_{R_n}[L_n(t_1); F_n(t_1)] + \int_{t_0}^{t_1} Q_n(t)dt\right] \notag
        \\&= \mathcal{E}_{R_\infty}[\overline{L}(t_0); F_\infty(t_0)] - \mathcal{E}_{R_\infty}[\overline{L}(t_1); F_\infty(t_1)] + \liminf_n\int_{t_0}^{t_1} Q_n(t)
    \end{align}
Here, the first inequality comes from the previous step, the second inequality uses lower semicontinuity of the total dissipation and assumption \ref{hyp:edlb_contact}, the third inequality uses assumption \ref{hyp:edlb_pressure}, and the fourth inequality uses assumption \ref{hyp:edlb_energy}. We obtain
\[ \int_{t_0}^{t_1} P^*_{R_\infty}[\overline{L}(t); F_\infty(t)]\dot{F}_\infty = \int_{t_0}^{t_1}\limsup_n Q_n(t) \leq \liminf_n \int_{t_0}^{t_1} Q_n(t) \]
Consequently, we can apply the elementary result \lref{limsup_liminf_integral} to get that $Q_N\to \limsup_n Q_n$ in measure as $N\to\infty$, since assumption \ref{hyp:edlb_pressure} provides that the $Q_N$ are dominated. The energy dissipation balance for $\overline{L}(t)$ then follows from \eqref{eq:edlb_computation}, after replacing $\liminf_n \int_{t_0}^{t_1} Q_n$ by its value $\int_{t_0}^{t_1} P^*_{R_\infty}[\overline{L}(t); F_\infty(t)]\dot{F}_\infty$ by dominated convergence theorem.

\end{proof}

\section{Existence of minimizing movements energy solutions}\label{s.discrete_scheme_and_convergence}~

We now consider a discrete-time scheme for the evolution. Fix $\delta > 0$, and let $P_\delta = \{ t_0,\dots, t_N\}$ be a partition of times with $0 = t_0 < \dots < t_N = T$ and $0 < t_{i+1} - t_i < \delta$ for each $i$. Given the initial shape $L(t_0)$, we choose the subsequent elements of the discrete scheme for partition $P_\delta$ by:
\begin{equation}\label{e.stepwise-min}
L_\delta(t_{i+1})\in \operatorname*{argmin}_{L\in X^R_{F(t_{i+1})}} \bigg\{ \mathcal{E}_R[L; F(t_{i+1})] + \Diss[L_\delta(t_i), L] \bigg\}
\end{equation}
Recall that we have shown existence of minimizers for this functional in \lref{existence_of_minimizer} and that minimizers are automatically globally stable profiles. We extend this to a function of time via discontinuous interpolation:
\begin{equation}\label{eq:discrete_interp}
L_\delta(t) = L_\delta(t_{i}) \hbox{ for }t\in [t_i,t_{i+1})
\end{equation}
We will derive some basic estimates for the discrete scheme which show that we can send $\delta\to 0$ and obtain a limiting evolution which satisfies Definition \ref{def:constrained_energy_sln}. Our main result is:

\begin{corollary}\label{c.compactness_of_discrete_schemes}
Up to a subsequence $\delta_n\to 0$, we have that $\PS_{L_{\delta_n}(t)}$ converges in $L^1(d\mathcal{H}^{d-1})$ for each $t$, to a limit which we denote $\PS_{L(t)}$. For $L(t)$ the minimal energy profile associated to the contact set $\PS_{L(t)}$, we have that $L(t)$ is an energy solution as in Definition \ref{def:constrained_energy_sln}.
\end{corollary}

To prove this, we use the results of the previous section to verify global stability and energy dissipation balance for the limit. This requires the Taylor expansion of \pref{energy_taylor_exp_alt} to show that the discrete scheme obeys a discrete version of the energy-dissipation balance, which converges to the expected balance as $\delta\to 0$.

\begin{lemma}\label{l.discrete_scheme_energy_bound}
We have the total dissipation bound
\begin{align}\label{e.discrete_scheme_energy_diss_gronwall}
    \mathcal{E}_R&[L(t_{N_1}); F(t_{N_1})] - \mathcal{E}_R[L(t_0); F(t_0)] + \sum_{i=0}^{N_1} \Diss[L_\delta(t_i), L_\delta(t_{i+1})] \notag
    \\&\leq A_0[F]_{BV[t_0, t_{N_1}]}e^{C[F]_{BV[t_0, t_{N_1}]}}
\end{align}
for constants $A_0, C$ depending on $L_\delta(t_0)$ and $\osc F$.

In particular, we conclude the energy bound
\begin{equation}\label{e.discrete_scheme_energy_only_gronwall}
    \sup_{0 \leq i \leq N_1} (C_0 + \mathcal{E}_R[L(t_i); F(t_i)]) \leq A_0 e^{C[F]_{BV[t_0, t_{N_1}]}}
\end{equation}
where $C_0$ is an additive constant similar to the one appearing in \lref{energy_bdd_below}, such that $\mathcal{E}_R + C_0 > 0$.

\end{lemma}
\begin{proof} 
For $L_\delta(t)$ as in \eqref{e.stepwise-min}, \eqref{eq:discrete_interp}, write
\[ d_i = \mathcal{E}_R[L_\delta(t_{i+1}); F(t_{i+1})] + \Diss[L_\delta(t_i), L_\delta(t_{i+1})] - \mathcal{E}_R[L_\delta(t_i); F(t_i)] \]
We use \pref{energy_taylor_exp_alt} to obtain
\begin{align*}
    \sum_{i=0}^{N-1} \Diss[L_\delta(t_i), L_\delta(t_{i+1})] &= \mathcal{E}_R[L_\delta(t_0); F(t_0)] - \mathcal{E}_R[L_\delta(t_N); F(t_N)] + \sum_{i=0}^{N-1} d_i
    \\&\leq \mathcal{E}_R[L_\delta(t_0); F(t_0)] - \mathcal{E}_R[L_\delta(t_N); F(t_N)] 
    \\&\quad+ \sum_{i=0}^{N-1} P_R^*[L_\delta(t_i); F(t_i)](F(t_{i+1}) - F(t_i)) 
    \\&+ C(C_0 + \mathcal{E}_R[L_\delta(t_i); F(t_i)])(F(t_{i+1}) - F(t_i))^2
\end{align*}
Here, $C$ is the constant in front of the remainder term in the proposition, which is uniformly bounded as explained in the remark after the proposition. The energy is being used with an additive constant $C_0$ from the coercivity result \lref{energy_bdd_below} to control the term $|\FS_{L_0}\cap C_{R_0}|$ from the proposition. For conciseness, we write $\widetilde{\mathcal{E}}_R[L; F] := C_0 + \mathcal{E}_R[L;F]$ for the rest of the proof.

The pressure $P^*_R[L_\delta(t_i); F(t_i)]$ can also be bounded by $C\widetilde{\mathcal{E}}_R[L_\delta(t_i); F(t_i)])$ for a different $C$, so we merge together to get
\begin{align}
     &\widetilde{\mathcal{E}}_R[L_\delta(t_N); F(t_N)]-\widetilde{\mathcal{E}}_R[L_\delta(t_0); F(t_0)]+\sum_{i=0}^{N-1} \Diss[L_\delta(t_i), L_\delta(t_{i+1})] \notag\\
     & \qquad \leq   \sum_{i=0}^{N-1} C(1+\osc F)\widetilde{\mathcal{E}}_R[L_\delta(t_i); F(t_i)]|F(t_{i+1}) - F(t_i)| \label{e.pre-gronwall}
\end{align}
From here, we use Gronwall-type arguments to bound the energy and the total dissipation separately. First, for the energy, we have
\begin{align*}
    \widetilde{\mathcal{E}}_R[L_\delta(t_N); F(t_N)] &\leq \widetilde{\mathcal{E}}_R[L_\delta(t_0); F(t_0)] \\&\quad+ \sum_{i=0}^{N-1} C(1+\osc F)\widetilde{\mathcal{E}}_R[L_\delta(t_i); F(t_i)](F(t_{i+1}) - F(t_i))
\end{align*}
Let $f(n) = \widetilde{\mathcal{E}}_R[L_\delta(t_n); F(t_n)]$. We have shown $f > 0$ with
\[ f(n) \leq f(0) + C(1+\osc F)\sum_{i=0}^{n-1} f(i)|F(t_{i+1}) - F(t_i)| \]
We recall the discrete-time version of Gr\"onwall's inequality: we have $f(0) \leq f(0)$, and inductively, if 
\[ f(n) \leq f(0)\prod_{i=0}^{n-1} \left(1 + C(1+\osc F)|F(t_{i+1}) - F(t_i)|\right) \]
then
\begin{align*}
    &f(n+1) 
    \\&\leq f(0) + C(1+\osc F)\sum_{k=0}^{n} f(k)|F(t_{k+1}) - F(t_k)|
    \\&\leq f(0)\left(1 + \sum_{k=0}^n |F(t_{k+1})-F(t_k)|\prod_{i=0}^{k-1}  \left(1 + C(1 + \osc F)|F(t_{i+1}) - F(t_i)|\right)\right)
    \\&\leq f(0)\left(\prod_{k=0}^n \left(1 + C(1 + \osc F)|F(t_{k+1}) - F(t_k)|\right)\right)
\end{align*}
where the final inequality follows directly from matching terms. Using that $\prod (1 + a_k) \leq e^{\sum a_k}$ for $a_k \geq 0$, we conclude
\begin{align*}
    f(n) &\leq f(0)\exp\left(C(1+\osc F)\sum_{i=0}^{n-1}|F(t_{i+1}) - F(t_i)|\right)
    \\&\leq f(0)\exp\left(C(1+\osc F)[F]_{BV[t_0, t_{n}]}\right)
\end{align*}
This proves \eref{discrete_scheme_energy_only_gronwall}. To prove \eref{discrete_scheme_energy_diss_gronwall}, we then substitute the energy bound into the right-hand-side of \eref{pre-gronwall} to get
\begin{align*}
     &\mathcal{E}_R[L_\delta(t_N); F(t_N)]-\mathcal{E}_R[L_\delta(t_0); F(t_0)]+\sum_{i=0}^{N-1} \Diss[L_\delta(t_i), L_\delta(t_{i+1})] \notag\\
     & \quad \leq (C_0 + \mathcal{E}_R[L_\delta(t_0); F(t_0)])\exp\left(C(1+\osc F)[F]_{BV[t_0, t_{N}]}\right)
     \\&\quad\quad\cdot\sum_{i=0}^{N-1} C(1+\osc F)|F(t_{i+1}) - F(t_i)| 
\end{align*}
which is in the desired form.

\end{proof}

In particular, as a refinement of the Taylor expansions in the previous lemma, we have:
\begin{lemma}[Discrete approximation of pressure-forcing]\label{l.disc_pressure_forcing}
For any $0 < m < n < N$,
\begin{align*}
    &\mathcal{E}_R[L_\delta(t_m); F(t_m)] - \mathcal{E}_R[L_\delta(t_n); F(t_n)] + \sum_{i=m}^{n-1} P^*_R[L_\delta(t_i); F(t_i)](F(t_{i+1}) - F(t_i)) 
    \\&\geq \sum_{i=m}^{n-1} \Diss[L_\delta(t_i), L_\delta(t_{i+1})] - C\sum_{i=m}^{n-1} (F(t_{i+1}) - F(t_i))^2
\end{align*}
In particular, setting
\[ Q_\delta(t) = \sum_{i=0}^{N-1} P^*_R[L_\delta(t_i); F(t_i)]\frac{F(t_{i+1}) - F(t_i)}{t_{i+1} - t_i}\one_{[t_{i}, t_{i+1})}(t) \]
we have for any fixed $t_0, t_1$ that
\begin{align*}
    &\liminf_{\delta\to 0^+} \left[\mathcal{E}_R[L_\delta(t_0); F(t_0)] - \mathcal{E}_R[L_\delta(t_1); F(t_1)] + \int_{t_0}^{t_1} Q_\delta(t)dt\right] 
    \\&\quad\geq \liminf_{\delta\to 0^+}\overline{\Diss}(L_\delta, [t_0, t_1])
\end{align*}
Moreover, for any $t$, along any subsequence $\delta_k\to 0$ such that $L_{\delta_k}(t)$ converges in $L^1_{\mathrm{loc}}$ to some $L$, we have $Q_{\delta_k}(t) \to P^*_R[L; F(t)]\dot{F}(t)$.
\end{lemma}
\begin{proof}
    The first inequality follows directly from the computation in \lref{discrete_scheme_energy_bound}. The inequality of liminfs then follows from the fact that $\sum (F(t_{i+1}) - F(t_i))^2\to 0$ as partition fineness goes to 0 for any continuous BV function. 
    
    Fix a time $t$ and restrict to a subsequence along which $L_{\delta_k}(t)$ converges in $L^1_{\mathrm{loc}}$ to some $L$. We assume enough regularity for $F$ that $F(t_k)\to F(t)$, where $t_k$ is the actual time step of $L_{\delta_k}(t)$. Thus, \lref{pressure_continuous} gives
    \[ P^*_R[L_{\delta_k}(t); F(t_k)] \to P^*_R[L; F(t)] \]
    The convergence $Q_{\delta_k}(t)\to P^*_R[L; F(t)]\dot{F}(t)$ then follows.
\end{proof}

Finally, we prove the main result of this section: compactness of the discrete scheme as $\delta\to 0^+$, with all subsequential limits satisfying the properties of an energy solution as per Definition \ref{def:constrained_energy_sln}.
\begin{proof}[Proof of \cref{compactness_of_discrete_schemes}]
We use the version of Helly's section principle stated in \lref{helly} to obtain compactness. To apply this, we require \lref{discrete_scheme_energy_bound} to bound the energy and total dissipation of the discrete scheme independently of $\delta$. Then, with the energy bound and global stability of the discrete scheme profiles, we can apply \tref{length_bound} and \lref{width-bound} to conclude that the profiles satisfy contact line bounds and height bounds which are also independent of $\delta$. This verifies the hypotheses of the lemma, so given $(\PS_{L_\delta(t)})$, we can find a subsequence $\delta_n\to 0$ and a limiting flow of contact sets $\Sigma^{c}_\infty(t)$ such that $\lim_{n\to\infty}|\PS_{L_{\delta_n}(t)}\Delta \Sigma^{c}_\infty(t)| = 0$ for each $t$. 

This selects the contact sets of the limiting path, but not the profiles $L(t)$ themselves. We recall that by \lref{glbl_stbl_cpt}, the $L_{\delta_n}(t)$ are precompact in $L^1_{\mathrm{loc}}$ as $n\to\infty$ for each $t$, and all subsequential limits are globally stable for $\mathcal{E}_R[\cdot, F(t)]$ with the same energy $\lim_{n \to \infty}\mathcal{E}_R[L_{\delta_n}(t), F_{\delta_n}(t)]$ (and this limit exists along the full sequence). Here $F_\delta(t)$ is the piecewise constant interpolation of $F$ on the partition $P_\delta$ defined similarly to \eqref{e.stepwise-min}.

The existence of an energy solution $L(t)$ with contact set $\PS_{L(t)} = \Sigma^{c}_\infty(t)$ will follows from \lref{energy_diss_lower_bound} after we check its assumptions. We are working with fixed $R$, so \ref{hyp:edlb_r} is trivial. $F$ is sufficently regular for \ref{hyp:edlb_f} to hold. We have already established the convergence of $\PS_{L_{\delta_n}(t)}$ for \ref{hyp:edlb_contact}, the compactness of $L_{\delta_n}(t)$ for \ref{hyp:edlb_profile} and the energy convergence for \ref{hyp:edlb_energy}. Finally, we use \lref{disc_pressure_forcing} and its definition of $Q_{\delta_n}(t)$ to satisfy \ref{hyp:edlb_pressure}.

\end{proof}

\section{Limit as \texorpdfstring{$R \to \infty$}{R goes to infinity}}\label{s.container_limit}

We now turn toward establishing compactness of energy solutions as the outer container radius $R$ is allowed to increase to $\infty$. This is challenging due to the fact that the absolute energy \eqref{eq:energy} diverges with $R$, so an appropriate relativization must be found. To deal with this issue, we first show in Section \ref{s.lagrange-decomp} that the Lagrange multiplier $\lambda$ for a globally stable state has a decomposition $\lambda = a_0R^{-1} + O(R^{1-d})$, where crucially the coefficient of the first term is an explicit constant. In Section \ref{s.L-outer-approx} this allows us to use our viscosity theory to trap the globally stable state between two barriers that are very close, with the large term in the Lagrange multiplier only acting as a constant vertical shift.

Finally in Section \ref{s.limit-R} we apply this to establish asymptotics for the diverging part of the energy with respect to $R$. These are accurate to high enough order to allow us to relativize the energy, send $R\to\infty$, and obtain a limiting evolution with finite relative energy. The last part of the section then verifies that the limiting object is an energy solution for the expected limiting energy \eref{limiting_energy} in our main result \cref{compactness_of_energy_solutions_in_r}.

\subsection{An important decomposition of the Lagrange multiplier}\label{s.lagrange-decomp}

The decomposition of the Lagrange multiplier can be understood through the following formal computation:
\begin{align*}
    \lambda_L &= \frac{1}{|B_R\setminus B_1|}\int_{\FS_L} \lambda_L e_z\cdot \nu^f_L
    \\&= \frac{1}{|B_R\setminus B_1|}\int_{\FS_L} (-H_{\FS_L} + g(z-F))e_z\cdot \nu^f_L
    \\&= \frac{1}{|B_R\setminus B_1|}\left(\int_{\PL_L} \eta^p\cdot e_z + \int_{\CL_L} \eta^o\cdot e_z\right)
\end{align*}
where $\eta^p, \eta^o$ are conormal vectors, normal to the contact lines and tangent to the free surface. These integrals can then be estimated in terms of the contact angles, and in particular the integral over $\CL_L$ can be computed explicitly, using Stokes' theorem along with the fact that the contact angle is constant on $\CL_L$.

To avoid the technical issues that arise with integrating by parts near the contact line, we instead we use an energy perturbation argument to capture the same idea.

\begin{lemma}\label{l.lambda_decomposition}
 Let $L$ be globally stable for $\mathcal{E}_R[\cdot, F]$, $R\in (1,\infty)$. We have the following formula for the Lagrange multiplier $\lambda_{L} = -H_{\FS_{L}} + g(z-F)$:
\begin{equation}\label{eqn:lambda_decomposition}
    \lambda_{L} = -\frac{\cos\YC\mathcal{H}^{d-2}(\partial B_R)}{|B_R\setminus B_1|} + \frac{\tilde{\lambda}_{L}}{|B_R \setminus B_1|}
\end{equation}
    where 
\begin{equation}\label{eqn:lambda_decomposition_remainder}
    \tilde{\lambda}_{L} \in -\cos\YP\mathcal{H}^{d-2}(\partial B_1) + (\mu_-\vee \mu_+)\mathrm{Per}_{\thinCyl_1}(\PS_{L})[-1, 1]
\end{equation}
By \tref{length_bound}, we conclude that $\lambda_{L} = a_0 R^{-1} + a_1(R,L)R^{1-d}$, where $a_0 = \cos\YC \frac{\omega_{d-2}}{\omega_{d-1}}$ is a constant independent of $L$ and $R$, and $a_1$ is uniformly bounded in $L, R$.
\end{lemma}
\begin{proof}
We consider a test velocity $\mathcal{V} = -\frac{e_z}{|B_R\setminus B_1|} + \mathcal{V}_0$, for $\mathcal{V}_0$ as in \cref{volume_only_energy_taylor_exp}. Let $L_s$ be the result of flowing $L_0 := L$ under this velocity for time $s$.

We note that $\mathcal{C}_F[L_s] = \mathcal{C}_F[L_0] = 0$ by repeating the argument of \lref{volume-fix-flow-alt}. The energy change due to flowing by $\mathcal{V}_0$ for time $s$ is $\lambda_{L_0}s + O(s^2)$ by \cref{volume_only_energy_taylor_exp}. We subsequently must compute the energy change due to the constant vertical shift explicitly. The area of the free surface does not change, while the relative areas of the contact surfaces change by exactly $-s\frac{\mathcal{H}^{d-2}(\partial B_r)}{|B_R\setminus B_1|}$, for $r=1,R$ respectively. The change in gravity can also be computed (we assume $s > 0$ for demonstration, and write $c_R = |B_R\setminus B_1|^{-1}$):
\begin{align*}
    &\int_{L_s\Delta \Cyl_{R,F}} |z-F| - \int_{(L_s + c_Rse_z)\Delta \Cyl_{R,F}} |z-F|
    \\&= \int_{L_s\cap \{z \geq F\}} -c_Rs +\int_{\{z \leq F - c_Rs\}\setminus L_s} c_Rs 
    \\&\dots-\int_{L_s\cap \{F - c_Rs < z < F\}} z + c_R s 
    - F -\int_{\{F - c_R s < z < F\} \setminus L_s} z - F
    \\&= -c_R s\int_{L_s\Delta \Cyl_{R,F}} \mathrm{sign}(z-F)-\int_{\Cyl_{R,F}\setminus \Cyl_{R,F-c_R s}} z + c_R s 
    - F
\end{align*}
Here, the first term in the last line disappears since $\mathcal{C}_F[L_s] = 0$, while the second term is clearly $O(s^2)$. Thus, we absorb this into the other quadratic remainder term to obtain:
\begin{align}\label{e.energy_perturbation_for_lambda_decomp}
        \mathcal{E}[L_s; F] &= \mathcal{E}[L_0; F] + \lambda_L s + O(s^2) 
        \\&\dots+ s\frac{\cos\YP \mathcal{H}^{d-2}(\partial B_1)}{|B_R\setminus B_1|} + s\frac{\cos\YC \mathcal{H}^{d-2}(\partial B_R)}{|B_R\setminus B_1|}.\notag
\end{align}

To estimate the dissipation for this flow, we will use the perimeter bound on $\PS_L$. With an argument essentially identical to that of \cite{Maggi_2012}*{Lemma 17.9}, one can show with sharp constant that
\begin{align}\label{e.perimeter_diss_estimate_for_lambda_decomp}
    \Diss(L_0, L_s) &\leq (\mu_-\vee \mu_+)|(\PS_{L_0} - |B_R\setminus B_1|^{-1}s e_z)\Delta \PS_{L_0}| \notag
    \\&\leq |s|\frac{(\mu_-\vee \mu_+)\mathrm{Per}_{\thinCyl}(\PS_{L_0})}{|B_R\setminus B_1|}
\end{align}
We combine \eqref{e.energy_perturbation_for_lambda_decomp} and \eqref{e.perimeter_diss_estimate_for_lambda_decomp} with the global stability of $L_0$ to obtain:
\begin{align*}
    0 &\leq \lambda_{L_0} s + s\frac{\cos\YP \mathcal{H}^{d-2}(\partial B_1)}{|B_R\setminus B_1|} + s\frac{\cos\YC \mathcal{H}^{d-2}(\partial B_R)}{|B_R\setminus B_1|}
    \\&\quad\dots+|s|\frac{(\mu_-\vee\mu_+)\mathrm{Per}_{\thinCyl_1}(\PS_{L_0})}{|B_R\setminus B_1|} + O(s^2)
\end{align*}
Sending $s\to 0$, we obtain \eqref{eqn:lambda_decomposition} and \eqref{eqn:lambda_decomposition_remainder}.
\end{proof}

\subsection{Outer approximation of globally stable profiles}\label{s.L-outer-approx}

In this section, we develop approximations of the energy in the region $\Cyl_R\setminus \Cyl_{R_0}$, where $R_0\in (1,R)$ is a parameter. To do so, we show that globally stable profiles can generally be trapped between translates of a fixed $R$-dependent reference configuration using our viscosity theory, particularly \pref{sliding-comparison}. Subsequently, we show that approximating with the reference configuration yields good estimates for the energy.

Write $\mathsf{L}_{R}^0$ for the unconstrained minimizer of $\mathcal{E}_R[\cdot; 0]$ in $\Cyl_R$. We have the Euler-Lagrange equation:
\begin{equation}
\begin{cases}
    -H_{\FS_{\mathsf{L}_{R}^0}} + gz = 0 & \hbox{ on }\FS_{\mathsf{L}^0_{ R}} \\
    \nu^f_{\mathsf{L}_{R}^0}\cdot \nu^c_{\mathsf{L}_{R}^0} = -\cos\YC & \hbox{ on }\CL_{\mathsf{L}_{R}^0} \\
    \nu^f_{\mathsf{L}_{R}^0}\cdot e_r = -\cos\YP & \hbox{ on }\thinCyl_{1}\cap \FS_{\mathsf{L}_{R}^0}
\end{cases}
\end{equation}
If we consider vertical translations $\mathsf{L}^0_{ R} + t e_z$, the right-hand side of the first equation is replaced by $gt$, while the other two equations are unchanged.

We show that the decomposition for the Lagrange multiplier \eqref{eqn:lambda_decomposition} implies that shifts of $\mathsf{L}_{R}^0$ closely approximate general globally stable profiles far from the plate:

\begin{lemma}\label{l.oscillation_bound}
Let $L$ be globally stable for $\mathcal{E}_R[\cdot; F]$. Then there exist constants $C, c$ independent of $L, R, r$ such that
\begin{equation}
    L \Delta \left(\mathsf{L}_{R}^0 +(F+ \frac{a_0}{gR})e_z\right) \subset \FS_{\mathsf{L}^0_R + (F+ \frac{a_0}{gR})e_z} + [-h(r), h(r)]\,e_z \ \hbox{ in } \Cyl_{R} \setminus \Cyl_{r}
\end{equation}
for $h(r) = C\left(R^{1-d} + e^{-cr} + e^{-c(R - r)}\right)$. Here $a_0 = \cos\YC \frac{\omega_{d-2}}{\omega_{d-1}}$ is the same constant from Lemma~\ref{l.lambda_decomposition}.
\end{lemma}
\begin{proof}
    We apply \pref{sliding-comparison} with $r$ serving as the $R_0$ in the statement of that result. This yields that the inclusion holds for $h > \max(\frac{|\lambda_L - \lambda_{\mathsf{L}_{R}^0 + (F + \frac{a_0}{gR})e_z}|}{g}, 2h_0)$ where $h_0$ bounds the oscillation in height of $\mathrm{tr}_{\thinCyl_r}(\FS_{L})$ and $\mathrm{tr}_{\thinCyl_r}(\FS_{\mathsf{L}_{R}^0})$. We have $\lambda_{\mathsf{L}_{R}^0 + (F + \frac{a_0}{gR})e_z} = \frac{a_0}{R}$, so the first term in the maximum is $O(R^{1-d})$ by \lref{lambda_decomposition}. We bound the second term in the maximum by exponentials as a direct application of \lref{width-bound}.
\end{proof}

We will subsequently write 
\begin{equation}\label{e.reference_configuration}
    \mathsf{L}_{R} := \mathsf{L}_{R}^0 + (F+\frac{a_0}{gR})e_z,
\end{equation}
where the dependence on $F$ is not displayed explicitly to reduce notation.

\begin{lemma}\label{l.energy_asymptotics}
    Let $L$ be globally stable for $\mathcal{E}_R[\cdot; F]$, and $\mathsf{L}_{R}$ be as in \eref{reference_configuration}. Fix $R_0\in (1, R)$. We have estimates comparing the two profiles in the region $\Cyl_R\setminus \Cyl_{R_0}$, as follows. All constants $C, c$ will be independent of $R, R_0, L, F$, but may depend on all other parameters.

    The difference in volume is bounded as:
    \begin{equation}\label{eq:energy_asymptotics_volume}
        \vol((L\Delta \mathsf{L}_R)\cap (\Cyl_R\setminus \Cyl_{R_0})) \leq C
    \end{equation}
    The difference in container contact surface area is bounded as:
    \begin{equation}\label{eq:energy_asymptotics_container}
        \left|\relAreaC{\CS_{L}} - \relAreaC{\CS_{\mathsf{L}_{R}}}\right| \leq \frac{C}{R}
    \end{equation}
    The difference in gravitational energy is bounded as:
    \begin{align}\label{eq:energy_asymptotics_gravity}
        &\left|\int_{L\Delta (\Cyl_{R,F}\setminus C_{R_0,F})}|z-F| - \int_{\mathsf{L}_{R}\Delta (\Cyl_{R,F}\setminus C_{R_0,F})} |z-F|\right| \notag 
        \\&\quad\leq C\left(\frac{\log R}{R} + e^{-c R_0}\right)
    \end{align}
    The difference in free surface area is bounded as:
    \begin{equation}\label{eq:energy_asymptotics_perimeter}
        |\mathrm{Per}(L, \Cyl_R\setminus \Cyl_{R_0}) - \mathrm{Per}(L_{R,R_0}, \Cyl_R\setminus \Cyl_{R_0})| \leq\; C\left(\frac{\log R}{R} + e^{-c R_0}\right)
    \end{equation}
\end{lemma}
\begin{proof}
    First, we estimate the volume difference. We use \lref{oscillation_bound}, noting that the vertical oscillation of the symmetric difference $L\Delta \mathsf{L}_R$ inside the thin shell $\thinCyl_{r}$ can be bounded by 
    \[\min_{s \leq r} h(s) = \begin{cases}
        h(R/2) & r\geq R/2 \\ h(r) & r < R/2
    \end{cases} \]
    It follows that
    \begin{align*}
        &\vol((L\Delta \mathsf{L}_R)\cap (\Cyl_R\setminus \Cyl_{R_0})) 
        \\&\leq C \int_{R_0}^{R} (R^{1-d} + e^{-c(r\wedge R/2)})r^{d-2}dr
        \\&= C\frac{R^{d-1} - R_0^{d-1}}{R^{d-1}} + Ce^{-cR/2}\frac{R^{d-1} - (R_0\vee R/2)^{d-1}}{R^{d-1}} 
        \\&\quad\dots+ C\mathrm{sgn}_+(R/2 - R_0)\int_{R_0}^{R/2} e^{-cr}r^{d-2}dr \numberthis\label{eq:pf_of_energy_asym_volume}
    \end{align*}
    The last integral is bounded by $\Gamma(d-1)$ up to a constant factor, so the entire expression is bounded uniformly in $R_0, R$, giving \eqref{eq:energy_asymptotics_volume}.

    Next, we estimate the container contact surface. Once again, we use that the vertical oscillation of the symmetric difference $L\Delta \mathsf{L}_R$ inside $\thinCyl_R$ is bounded by $h(R/2) = R^{1-d} + e^{-cR/2}$. Combining with $\mathcal{H}^{d-2}(\partial B_R)\sim R^{d-2}$ gives \eqref{eq:energy_asymptotics_container}.

    Next, we estimate the difference in gravity. To improve on the estimate for the difference in volume, we must take advantage of the fact that the gravitational potential, $|z-F|$, is small far from the container. Thus, we split into cases. Let $R_1 = R_0\vee (R - c^{-1}\log R)$, where $c$ is the same constant appearing in the term $e^{-c(R - R_1)}$ in the width bound from \lref{width-bound}.
    
    In the outer region $\Cyl_R\setminus \Cyl_{R_1}$, we use the width bound to have $|z - F| = O(1)$. Since $R - R_1 = O(\log R)$, we are able to improve the estimate of the volume difference in \eqref{eq:pf_of_energy_asym_volume} to give $\vol((L\Delta \mathsf{L}_R)\cap (\Cyl_R\setminus \Cyl_{R_1})) =O(\tfrac{\log R}{R})$. Thus, the gravity difference in the outer region is $O(\tfrac{\log R}{R})$. In the inner region $\Cyl_{R_1}\setminus \Cyl_{R_0}$, we use the width bound to get $|z - F| = O(R^{-1} + e^{-cR_0})$. We combine this with the $O(1)$ estimate for the volume of $L\Delta \mathsf{L}_R$ to estimate the gravity difference in the inner region as $O(R^{-1} + e^{-cR_0})$. Combining the estimates for the two regions, we get the $O(\tfrac{\log R}{R} + e^{-cR_0})$ estimate for the difference of gravitational energies.

    Finally, to derive the free surface estimate, we use the stability properties of each profile. First, the unconstrained minimality of $\mathsf{L}_R$ for $\mathcal{E}_R[\cdot; F + \frac{a_0}{gR}]$ implies:
    \[ \mathcal{E}_R[\mathsf{L}_R; F + \frac{a_0}{gR}] \leq \mathcal{E}_R[((L\cap (\Cyl_R\setminus \Cyl_{R_0}))\cup (\mathsf{L}_R\cap \Cyl_{R_0})); F + \frac{a_0}{gR}] \]
    Expanding the energies and canceling common terms yields:
    \begin{align*}
        &|\FS_{\mathsf{L}_R}\cap (\Cyl_R\setminus \Cyl_{R_0})| + \int_{((\mathsf{L}_R\setminus \Cyl_{R_0}) \setminus \Cyl_{R,F}))\cup ((\Cyl_{R,F}\setminus \Cyl_{R_0})\setminus \mathsf{L}_R)} g|z - F - \frac{a_0}{gR}| 
        \\&\quad\dots- \cos\YC \relAreaC{\CS_{\mathsf{L}_R}}
        \\&\leq |\FS_{L}\cap (\Cyl_R\setminus \Cyl_{R_0})| + \int_{((L\setminus \Cyl_{R_0}) \setminus \Cyl_{R,F}))\cup ((\Cyl_{R,F}\setminus \Cyl_{R_0})\setminus L)} g|z - F - \frac{a_0}{gR}| 
        \\&\quad\dots- \cos\YC \relAreaC{\CS_{L}}+|\mathrm{tr}_{\thinCyl_{R_0}}(L\Delta \mathsf{L}_R)|
    \end{align*}
    Let us remark that the error term $|\mathrm{tr}_{\thinCyl_{R_0}}(L\Delta \mathsf{L}_R)|$ is $O(R_0^{d-2}R^{1-d}) \leq O(R^{-1})$, by the same argument used to obtain \eqref{eq:energy_asymptotics_container}. We also reorganize the gravity terms in the same step to obtain:
    \begin{align*}
        &|\FS_{\mathsf{L}_R}\cap (\Cyl_R\setminus \Cyl_{R_0})| + \int_{((\mathsf{L}_R\setminus \Cyl_{R_0}) \setminus \Cyl_{R,F}))\cup ((\Cyl_{R,F}\setminus \Cyl_{R_0})\setminus \mathsf{L}_R)} g|z - F|
        \\&\quad\dots- \cos\YC \relAreaC{\CS_{\mathsf{L}_R}}+ a_0R^{-1}\int \one_{((\mathsf{L}_R\setminus \Cyl_{R_0}) \setminus \Cyl_{R,F}))} - \one_{((\Cyl_{R,F}\setminus \Cyl_{R_0})\setminus \mathsf{L}_R)}
        \\&\leq |\FS_{L}\cap (\Cyl_R\setminus \Cyl_{R_0})| + \int_{((L\setminus \Cyl_{R_0}) \setminus \Cyl_{R,F}))\cup ((\Cyl_{R,F}\setminus \Cyl_{R_0})\setminus L)} g|z - F| 
        \\&\quad\dots- \cos\YC \relAreaC{\CS_{L}} - a_0R^{-1}\int\one_{((L\setminus \Cyl_{R_0}) \setminus \Cyl_{R,F}))} - \one_{((\Cyl_{R,F}\setminus \Cyl_{R_0})\setminus L)}
        \\&\quad\dots+O(R^{-1}) \numberthis\label{eq:pf_of_energy_asym_minimality_of_lr}
    \end{align*}
    We estimate the difference of the $a_0R^{-1}$ terms as $O(R^{-1})$, bounding the difference of integrals as $O(1)$ using \eqref{eq:energy_asymptotics_volume}. We estimate the difference of contact energy and gravitational energy terms using \eqref{eq:energy_asymptotics_container} and \eqref{eq:energy_asymptotics_gravity} respectively. Altogether, this yields
    \[ |\FS_{\mathsf{L}_R}\cap (\Cyl_R\setminus \Cyl_{R_0})| \leq |\FS_{L}\cap (\Cyl_R\setminus \Cyl_{R_0})|  + C\left(\frac{\log R}{R} + e^{-cR_0}\right). \]
    
    On the other hand, we can apply the global stability of $L$ against a competitor formed by (i) replacing $L$ by $\mathsf{L}_R$ inside $\Cyl_R\setminus \Cyl_{R_0}$ and (ii) flowing as in \cref{volume_only_energy_taylor_exp} to fix the volume constraint, with a velocity supported in $\Cyl_R\setminus \Cyl_{R_0}$. Once again, we obtain an $O(R^{-1})$ error term from the area of the free surface in $\thinCyl_{R_0}$. We recall that $\lambda_{\mathsf{L}_R} = O(R^{-1})$ and absorb the remainder term of the Taylor expansion. This yields
    \begin{align*}
        &|\FS_{L}\cap (\Cyl_R\setminus \Cyl_{R_0})| + \int_{((L\setminus \Cyl_{R_0}) \setminus \Cyl_{R,F}))\cup ((\Cyl_{R,F}\setminus \Cyl_{R_0})\setminus L)} g|z - F| 
        \\&\quad\dots- \cos\YC \relAreaC{\CS_{L}} 
        \\&\leq |\FS_{\mathsf{L}_R}\cap (\Cyl_R\setminus \Cyl_{R_0})| + \int_{((\mathsf{L}_R\setminus \Cyl_{R_0}) \setminus \Cyl_{R,F}))\cup ((\Cyl_{R,F}\setminus \Cyl_{R_0})\setminus \mathsf{L}_R)} g|z - F|
        \\&\quad\dots- \cos\YC \relAreaC{\CS_{\mathsf{L}_R}} + O(R^{-1})(1 + \vol((L\Delta \mathsf{L}_R)\cap (\Cyl_R\setminus \Cyl_{R_0}))). \numberthis\label{eq:pf_of_energy_asym-stability_of_l}
    \end{align*}
    Again combining the estimates \eqref{eq:energy_asymptotics_volume}, \eqref{eq:energy_asymptotics_container}, and \eqref{eq:energy_asymptotics_gravity}, we get
    \[ |\FS_{L}\cap (\Cyl_R\setminus \Cyl_{R_0})| \leq |\FS_{\mathsf{L}_R}\cap (\Cyl_R\setminus \Cyl_{R_0})| + C\left(\frac{\log R}{R} + e^{-cR_0}\right) \]
\end{proof}

\subsection{Relative energy and the $R \to \infty$ limit}\label{s.limit-R}

We continue to use $R_0$ to refer to a parameter in $(1, R)$ and $\mathsf{L}_R$ to refer to the reference configuration defined in \eref{reference_configuration}. We introduce the relative energy
\begin{equation}
    \widetilde{\mathcal{E}}_{R, R_0}[L; F] := \mathcal{E}_R[L; F] - \mathcal{E}[{\mathsf{L}_R}; \Cyl_R\setminus \Cyl_{R_0}] - |B_{R_0}\setminus B_1|
\end{equation}
To understand the relative energy, let us observe that the issues with sending $R\to\infty$ come in the following parts: (i) the energy of $L$ near the container, where one expects the area of the free surface, the gravitational energy, and the area of the container contact surface to diverge with $R$ and (ii) the energy of $L$ away from the container, where all terms have already been relativized in the definition of $\mathcal{E}_R[\cdot; F]$ except for the area of the free surface. Subtracting $|B_{R_0}\setminus B_1|$ handles (ii) and in a manner consistent with our definition for the $R=\infty$ energy \eref{limiting_energy}. To handle (i), we subtract $\mathcal{E}[{\mathsf{L}_R}; \Cyl_R\setminus \Cyl_{R_0}]$, relying on our proof in \lref{energy_asymptotics} that this closely approximates the energy of $L$ in the same region.

With this definition, we now show that we can send $R\to\infty$ and $R_0\to\infty$ in that order and obtain a meaningful limit for $\widetilde{\mathcal{E}}_{R, R_0}$.

\begin{lemma}\label{l.limiting_global_stability}
    Suppose for a sequence $R_n\to \infty$, we have profiles $L_{R_n}$ which are globally stable for $\mathcal{E}_{R_n}[\cdot, F]$ and which converge in $L^1_{\mathrm{loc}}$ to a set $L_\infty$. Then $L_\infty$ is globally stable for $\mathcal{E}_\infty[\cdot, F]$, with finite energy satisfying
    \begin{align*}
        \mathcal{E}_\infty[L_\infty; F] &= \lim_{R_0\to \infty}\limsup_{n\to\infty} \widetilde{\mathcal{E}}_{R_n, R_0}[L_{R_n}; F]
        \\&= \lim_{R_0\to \infty}\liminf_{n\to\infty} \widetilde{\mathcal{E}}_{R_n, R_0}[L_{R_n}; F].
    \end{align*}
    Moreover, if $L_{R_n}'$ is another sequence, globally stable for $\mathcal{E}_{R_n}[\cdot, F']$ and converging to $L_\infty'$, then we have
    \[ \lim_{n\to\infty} \mathcal{E}_{R_n}[L_{R_n}; F] - \mathcal{E}_{R_n}[L_{R_n}'; F'] = \mathcal{E}_\infty[L_\infty; F] - \mathcal{E}_\infty[L_\infty'; F'] \]
\end{lemma}
\begin{proof}
    First, we observe that for any $U\subset \R^{d-1}\setminus B_1$ bounded, the height bound on globally stable states ensures that $(L_{R_n}\cap (U\times \R))\Delta (U\times \{z \leq 0\})$ is contained in a compact set independent of $R_n$. Then by \lref{almost-minimizer-cpt-closure}, the uniform-in-$R$ almost-minimality of globally stable states implies that we have energy convergence $\mathcal{E}[L_{R_n}; U\times \R] \to \mathcal{E}[L_\infty; U\times \R]$ and convergence of contact sets in $U$ (under the assumption that the limiting free surface intersects $\partial U\times \R$ in a set of zero $\mathcal{H}^{d-1}$ measure, which holds for generic $U$).

    Now, considering $\widetilde{\mathcal{E}}_{R_n,R_0}[L_{R_n}; F]$, we split into the inner and outer parts of the energy
    \begin{align*}
        &\widetilde{\mathcal{E}}_{R_n,R_0}[L_{R_n}; F] 
        \\&\quad= \mathcal{E}[L_{R_n}; \Cyl_{R_0}] - |B_{R_0}\setminus B_1|\\
        &\quad \quad \cdots+\bigg(\mathcal{E}[L_{R_n}; \Cyl_{R_n}\setminus \Cyl_{R_0}] - \mathcal{E}[\mathsf{L}_{R_n}; \Cyl_{R_n}\setminus \Cyl_{R_0}]\bigg).
    \end{align*}
    For the inner part, the first term in parenthesis above, the above local energy convergence ensures that as $n\to\infty$, we have for a.e. $R_0$ that:
    \begin{align*}
        \lim_{n\to\infty} \mathcal{E}[L_{R_n}; \Cyl_{R_0}] - |B_{R_0}\setminus B_1| &= \mathcal{E}[L_\infty; \Cyl_{R_0}] - |B_{R_0}\setminus B_1|
    \end{align*}
    On the other hand for the outer energy, the energy asymptotics in \lref{energy_asymptotics} ensure 
    \begin{align*}
        &|\mathcal{E}[L_{R_n};\Cyl_{R}\setminus \Cyl_{R_0}] - \mathcal{E}[\mathsf{L}_{R_n};\Cyl_{R}\setminus \Cyl_{R_0}]| 
        \\&\quad\leq C\left( \frac{\log R_n}{R_n} + e^{-cR_0}\right)
    \end{align*}
    so
    \[\limsup_{n \to \infty} |\mathcal{E}[L_{R_n}; \Cyl_{R_n}\setminus \Cyl_{R_0}] - \mathcal{E}[\mathsf{L}_{R_n}; \Cyl_{R_n}\setminus \Cyl_{R_0}]| \leq Ce^{-cR_0}.\]
    Sending $R_0\to \infty$, this last error term disappears. Simultaneously, we have
    \[ \lim_{R_0\to\infty}\mathcal{E}[L_\infty; \Cyl_{R_0}] - |B_{R_0}\setminus B_1| = \mathcal{E}_\infty[L_\infty; F] \]
    since we defined $\mathcal{E}_\infty$ in \eqref{e.limiting_energy} by replacing the area of the free surface with the relativization $\int_{\FS_{L_\infty}} 1 - e_z\cdot\nu^f$. Thus, we conclude that
    \[ \mathcal{E}_\infty[L_\infty; F] = \lim_{R_0\to\infty}\limsup_{n\to\infty}\widetilde{\mathcal{E}}_{R_n,R_0}[L_{R_n}; F] \]
    The same computations hold with the limsup replaced by liminf.

    To verify that $\mathcal{E}_\infty[L_\infty; F]$ is indeed finite, we can use the width bound \lref{width-bound}. This implies that $\FS_{L_\infty}$ is exponentially close to $\{ z = F\}$ as $|x|\to \infty$. Consequently, we can use almost-minimality to show that the local energy in an $r$-ball centered at $(x,F)$ also decays exponentially as $|x|\to \infty$, from which we conclude that the energy is summable.

    Next, we show global stability. Supposing stability fails, there must be some globally stable $L_\infty'$ with 
    \[\mathcal{E}_\infty[L_\infty; F] - \mathcal{E}_\infty[L_\infty'; F] - \Diss[L_\infty, L_\infty'] =: \varepsilon > 0\]
    By a similar argument to that of the previous paragraph, we can show that $\mathcal{E}_\infty$ decays exponentially as $|x|\to\infty$ for $L_\infty'$ as well. In particular, this means that for all sufficiently large $R_0$, we have
    \[ \mathcal{E}[L_\infty; F, \Cyl_{R_0}] - \mathcal{E}[L_\infty'; F, \Cyl_{R_0}] - \Diss[L_\infty, L_\infty'] \geq \varepsilon/2 \]
    Then, from convergence of local energy and dissipation, we have
    \begin{equation}\label{eq:pf_of_limiting_glbl_stbl_with_eps}
        \liminf_{n\to\infty}\mathcal{E}[L_{R_n}; F, \Cyl_{R_0}] - \mathcal{E}[L_\infty'; F, \Cyl_{R_0}] - \Diss[L_n, L_\infty'] \geq \varepsilon/2
    \end{equation}
    
    We modify $L_\infty'$ to make a competitor to $L_{R_n}$. More precisely, we will keep the data of $L_\infty'$ up to a vertical shift on an annular region which scales with $R_0$ and avoids the container and plate, while replacing it elsewhere by $L_{R_n}$. Thus, we define:
    \[ L_{R_n}' := (L_{R_n}\cap ((\Cyl_{R_n}\setminus\Cyl_{R_0})\cup \Cyl_{\frac{R_0-1}{2}})\cup ((L_\infty'\cap (\Cyl_{R_0}\setminus \Cyl_{\frac{R_0-1}{2}}) + he_z) \]
    where $h$ is a parameter chosen to ensure the volume constraint $\mathcal{C}_{R_n}[L_{R_n}'] = 0$. Let us note that for fixed $R_0$, we can employ the local convergence of the $L_{R_n}$ to conclude that $|h|(|B_{R_0}| - |B_{\frac{R_0-1}{2}}|) \leq 1 + \vol((L_\infty\Delta L_\infty')\cap (\Cyl_{R_0}\setminus \Cyl_{\frac{R_0-1}{2}})$ for all sufficiently large $n$. This volume term is itself bounded independently of $R_0\gg 1$, again using the width bounds for $L_\infty$ and $L_\infty'$ which decay exponentially as $|x|\to \infty$, so we in fact have $|h| \leq C R_0^{1-d}$ for a uniform $C$ and all $n$ sufficiently large depending on $R_0$.

    Now, global stability of $L_{R_n}$ for $\mathcal{E}_{R_n}[\cdot; F]$ implies that
    \begin{align*}
        \mathcal{E}[L_{R_n}; \Cyl_{R_0}] 
        &\leq \mathcal{E}[L_{\infty}' + he_z; \Cyl_{R_0}] + \Diss(L_{R_n}, L_{\infty}' + he_z)
        \\&\quad\dots+|\mathrm{tr}_{\thinCyl_{R_0}}(L_{R_n}\Delta (L_\infty' + he_z))|
        \\&\leq \mathcal{E}[L_{\infty}'; \Cyl_{R_0}] + \Diss(L_{R_n}, L_{\infty}')
        \\&\quad\dots+\delta(R_n, R_0)
    \end{align*}
    where the remainder $\delta$ satisfies
    \begin{equation}\label{eq:pf_of_limiting_glbl_stbl_delta}
        \limsup_{n\to\infty} \delta(R_n, R_0) \leq CR_0^{-1}
    \end{equation}
    We achieve \eqref{eq:pf_of_limiting_glbl_stbl_delta} by straightforward estimates for the energy terms using the decay of $|h|$. The replacement produces free surface terms $|\mathrm{tr}_{\thinCyl_{r}}(L_{R_n}\Delta (L_\infty' + he_z))|$ for $r = R_0, \frac{R_0-1}{2}$. We bound these by using \lref{width-bound} to ensure that $\FS_{L_{R_n}}$ and $\FS_{L_\infty'}$ are within distance $C(R_n^{-1} +e^{-cR_0} + e^{-c(R_n - R_0)})$ of $\{ z = F\}$; combining then with $|h| \leq CR_0^{1-d}$ and $\mathcal{H}^{d-2}(\partial B_{R_0}) = CR_0^{d-2}$, we conclude that the areas of the free surface terms are at most $CR_0^{-1}$, provided $n$ is sufficiently large depending on $R_0$. For the change in gravity due to the shift, we estimate similarly to \lref{energy_asymptotics}, using that the change in volume due to the replacement is $O(1)$, while the gravitational potential $|z-F|$ in the modified region is controlled by $C(|h| + R_n^{-1} + e^{-cR_0} + e^{-c(R_n - R_0)}) \leq CR_0^{1-d}$ (again, for $n$ sufficiently large depending on $R_0$).

    Then from \eqref{eq:pf_of_limiting_glbl_stbl_delta}, it follows that
    \begin{align*}
        &\limsup_{R_0\to\infty}\limsup_{n\to\infty} \mathcal{E}[L_{R_n}; \Cyl_{R_0}] - \mathcal{E}[L_{\infty}'; \Cyl_{R_0}] - \Diss(L_{R_n}, L_{\infty}') = 0
    \end{align*}
    This contradicts \eqref{eq:pf_of_limiting_glbl_stbl_with_eps}, so we conclude that $L_\infty$ is globally stable.

    Now, we check the last part of the statement, regarding convergence of differences of nonrelativized energies. For any $R_0\in (1, R_n)$, we observe that
    \[ \mathcal{E}_{R_n}[L_{R_n}; F] - \mathcal{E}_{R_n}[L_{R_n}'; F'] = \widetilde{\mathcal{E}}_{R_n, R_0}[L_{R_n}; F] - \widetilde{\mathcal{E}}_{R_n, R_0}[L_{R_n}'; F'] \]
    since the relative energy simply subtracts a constant depending only on $R_n, R_0$. Thus, we have
    \begin{align*}
        &\limsup_{n\to\infty} \left(\mathcal{E}_{R_n}[L_{R_n}; F] - \mathcal{E}_{R_n}[L_{R_n}'; F'] \right)
        \\&\quad= \lim_{R_0\to\infty}\limsup_{n\to\infty} \left(\mathcal{E}_{R_n}[L_{R_n}; F] - \mathcal{E}_{R_n}[L_{R_n}'; F']\right)
        \\&\quad= \lim_{R_0\to\infty}\limsup_{n\to\infty} \left(\widetilde{\mathcal{E}}_{R_n, R_0}[L_{R_n}; F] - \widetilde{\mathcal{E}}_{R_n, R_0}[L_{R_n}'; F']\right)
        \\&\quad\leq \lim_{R_0\to\infty}\left(\limsup_{n\to\infty} \widetilde{\mathcal{E}}_{R_n, R_0}[L_{R_n}; F] - \liminf_{n\to\infty}\widetilde{\mathcal{E}}_{R_n, R_0}[L_{R_n}'; F']\right)
        \\&\quad= \mathcal{E}_\infty[L_\infty; F] - \mathcal{E}_\infty[L_\infty'; F']
    \end{align*}
    The $\liminf$ is bounded below in the same manner and by the same value, so the limit in $n$ exists.
\end{proof}

\begin{lemma}\label{l.uniform_in_r_total_diss_bound}
Let $L_R(t)$ be an energy solution in $\Cyl_R$ for forcing $F(t)$. Then it holds
\[ \overline{\Diss}(L_R(t); [0,T]) \leq C(1 + \|\dot{F}\|_{L^1[0,T]}) \]
for a constant $C$ independent of $L, R, F$.
\end{lemma}
\begin{proof}
    By the energy dissipation inequality, we have
    \[ \overline{\Diss}(L(t); [0,T]) \leq \mathcal{E}_R[L(0); F(0)] - \mathcal{E}_R[L(T); F(T)] + \int_0^T P^*_R[L(t); F(t)]\dot{F}(t)dt \]
    Using \lref{energy_asymptotics}, we take $R_0 = 2$ and bound the energy difference outside $\Cyl_{R_0}$ by a constant. Inside, we use the local perimeter bounds implied by \lref{almost_minimality} and the height bound from \lref{width-bound} to bound the energy. Similarly, $P^*_R[L(t); F(t)]$ is bounded independently of $R$ by combining the height and local perimeter bounds. Thus, we conclude that the entire right side is bounded as $C(1 + \|\dot{F}\|_{L^1})$.
\end{proof}

\begin{corollary}\label{c.compactness_of_energy_solutions_in_r}
If $L_{R_n}(t)$ are energy solutions in $\Cyl_{R_n}$ for forcing $F(t)$, with $R_n\to \infty$ as $n\to\infty$, then there exists a subsequence $(R_{n_k})$ and an energy solution $L_\infty(t)$ in $\Cyl_\infty$ such that $|\PS_{L_{R_{n_k}}(t)}\Delta \PS_{L_{\infty}(t)}|\to 0$ pointwise in $t$ as $k\to\infty$.
\end{corollary}
\begin{proof}

By the total dissipation bound in \lref{uniform_in_r_total_diss_bound}, the height bound in \lref{width-bound}, and the perimeter bound in \tref{length_bound}, we have the compactness to apply Helly's selection principle as in \lref{helly} for the $\PS_{L_n(t)}$. Without relabeling, we pass to a subsequence and identify a limit $\PS_\infty(t)$ such that
\[ \lim_{n\to\infty} |\PS_{L_n(t)}\Delta \PS_\infty| = 0 \hbox{ for all }t\in[0,T]. \]

Next, we select the limit profiles $L_\infty(t)$ for which $\PS_{L_\infty(t)} = \PS_\infty(t)$. As discussed in Remark \ref{rem:msb_sel}, the profiles need not converge in an evolutionary sense, so we instead apply \lref{energy_diss_lower_bound} to extract a measurable selection of subsequential limits which satisfy the properties of an energy solution. Thus, we must verify the six assumptions of this lemma.

The assumption \ref{hyp:edlb_r} that the $R_n$ converge is a hypothesis, while the forcing assumption \ref{hyp:edlb_f} is trivially satisfied since we do not vary $F$ with $n$. Our use of Helly's selection principle guarantees the dissipation bound and convergence of contact sets for assumption \ref{hyp:edlb_contact}. The compactness of the profiles, closure of global stability, and convergence of the relative energy needed for \ref{hyp:edlb_profile} and \ref{hyp:edlb_energy} has been verified in \lref{limiting_global_stability}. Finally, for assumption \ref{hyp:edlb_pressure} on the pressure-forcing term, we take $Q_{n}(t) = P^*_{R_{n}}[L_n(t); F(t)]\dot{F}(t)$. Note that $Q_n$ is uniformly bounded from the proof in \lref{uniform_in_r_total_diss_bound}, and $Q_n$ converges given convergence of profiles by \lref{pressure_continuous}. Therefore, we satisfiy all assumptions needed to conclude the result.
\end{proof}

\appendix
\section{Initial height bound}\label{app.initial_height_bound}

We present a proof that for globally stable $L$, it holds that $\FS_L$ is bounded in the vertical direction. The proof, based on \cite{AlbertiDeSimone}*{Proposition 5.4ii}, only relies on \lref{bv_trace_ineq} and \cref{volume_only_energy_taylor_exp} from the main paper. Thus, we can apply it to justify touching-from-above/below as we develop a viscosity theory for globally stable profiles in \sref{viscosity}.

\begin{proof}[Proof of \lref{qualitative_height_bound}]

We restrict to the case $R < \infty$. The proof for $R=\infty$ is similar, omitting the outer contact energy terms and the volume constraint.

Given $L$ globally stable for $\mathcal{E}_R[\cdot; F]$, we consider the truncation $L_h$ for $h > 0$ given by
\[ L_h = (L\cup\Cyl_{R,-h})\cap\Cyl_{R,h} \]
Let $a(h) = \vol(L\Delta L_h)$. We aim to show that $a(h) = 0$ for all $h$ sufficiently large by showing that $a$ satisfies a certain differential inequality. Note that
\[ a'(h) = -|S_h| \;\hbox{ where }\; S_h := (L\cap \{z = h\})\cup (\{z=-h\}\setminus L) \]
for a.e. $h$.

To make use of global stability, we must first fix the volume constraint for $L_h$, so that it is admissible as a competitor to $L$. Toward this end, we flow $L_h$ as in \cref{volume_only_energy_taylor_exp} to obtain a profile $\widetilde{L}_h$ satisfying $\mathcal{C}_F[\widetilde{L}_h] = 0$, $\Diss(L, \widetilde{L}_h) = \Diss(L,L_h)$, and
\begin{align*}
    \mathcal{E}_R[\widetilde{L}_h; F] &= \mathcal{E}_R[L_h; F] + \mathcal{C}_F[L_h]\int_{\FS_{L_h}} \nu^f\cdot D\mathcal{V}_0\nu^f + g(z-F)\mathcal{V}_0\cdot \nu^f\,d\mathcal{H}^{d-1} 
    \\&\quad\dots+ (1+h)|\FS_{L_h}\cap \mathrm{spt}(\mathcal{V}_0)|O(\mathcal{C}_F[L_h]^2)
    \\&\leq\mathcal{E}_R[L_h; F] + Ca(h)(1+h)(1 -a'(h))
\end{align*}
where in the second line, we bound $\mathcal{C}_F[L_h]$ by $a(h)$, $z-F$ by $h$, and all terms involving $\FS_{L_h} \subset \FS_{L}\cup S_h$ by $1 - a'$ (keeping in mind that anything purely in terms of $L$ or $F$ is $O(1)$ in $h$). We also assume that $h$ is chosen sufficiently large that $a(h) \ll 1$ in order to absorb the quadratic remainder into the linear term in the second line.

We now plug this into the global stability inequality, and expand $\mathcal{E}_R[L_h; F]$ in terms of $\mathcal{E}_R[L;F]$, using that $L_h$ is defined in terms of set operations on $L$. This yields:
\begin{align*}
    \mathcal{E}_R[L; F] &\leq \mathcal{E}_R[\widetilde{L}_h; F] + \Diss(L, \widetilde{L}_h)
    \\&\leq \mathcal{E}_R[L_h;F] + \Diss(L,L_h) + Ca(h)(1+h)(1 - a'(h))
    \\&\leq \mathcal{E}_R[L;F] + Ca(h)(1+h)(1 - a'(h))
    \\&\quad\dots+|S_h| - |\FS_L\cap \{ |z| > h\}| 
    \\&\quad\dots- g(h-F)\vol(L\cap \{ z > h\})
    \\&\quad\dots- g(h + F)\vol(\{ -z > h\}\setminus L)
    \\&\quad\dots+ (\cos\YP + \mu_-)|\PS_L\cap \{ z > h\}|
    \\&\quad\dots+ (- \cos\YP + \mu_+)|\{ -z > h\}\setminus \PS_L|
    \\&\quad\dots+ (\cos\YC - \mu_-)|\CS_L\cap \{ z > h\}|
    \\&\quad\dots+ (- \cos\YC + \mu_+)|\{ -z > h\}\setminus \CS_L|
\end{align*}
As a first step to deal with the proliferation of boundary terms on the right side, we use the isoperimetric inequality. Note that
\begin{align*}
    |\partial(L\Delta L_h)| &= |S_h| + |\FS_L\cap \{ |z| > h\}| 
    \\&\quad\dots+ |\PS_L\cap \{ z > h\}| + |\{ -z > h\}\setminus \PS_L|
    \\&\quad\dots+ |\CS_L\cap \{ z > h\}| + |\{ -z > h\}\setminus \CS_L|
    \\&\geq \frac{1}{C}\vol(L\Delta L_h)^{(d-1)/d}
\end{align*}
Let $\varepsilon > 0$ to be determined later. We add the nonnegative quantity $\varepsilon(|\partial(L\Delta L_h)| - \frac{1}{C}\vol(L\Delta L_h)^{(d-1)/d})$ to the right-hand side of the global stability inequality to obtain:
\begin{align*}
    0 &\leq Ca(h)(1+h)(1 -a'(h)) - \frac{\varepsilon}{C} a(h)^{(d-1)/d}
    \\&\quad\dots + (1 + \varepsilon)|S_h| - (1 - \varepsilon)|\FS_L\cap \{ |z| > h\}| 
    \\&\quad\dots- g(h-F)\vol(L\cap \{ z > h\})
    \\&\quad\dots- g(h + F)\vol(\{ -z > h\}\setminus L)
    \\&\quad\dots+ (\cos\YP + \mu_- + \varepsilon)|\PS_L\cap \{ z > h\}|
    \\&\quad\dots+ (- \cos\YP + \mu_+ + \varepsilon)|\{ -z > h\}\setminus \PS_L|
    \\&\quad\dots+ (\cos\YC - \mu_- + \varepsilon)|\CS_L\cap \{ z > h\}|
    \\&\quad\dots+ (- \cos\YC + \mu_+ + \varepsilon)|\{ -z > h\}\setminus \CS_L|
\end{align*}
Note that again $|S_h| = -a'(h)$ for a.e. $h$, so we can combine $|S_h|$ with the first term.

Next, we use \lref{bv_trace_ineq} to bound the four contact surface terms and absorb the result into the gravity and free surface coefficients. To ensure the sign of the free surface coeffcient does not change, we require that each of the coefficients $\cos\YP + \mu_- + \varepsilon$ and so on to be strictly smaller than $1-  \varepsilon$, which we can guarantee by choosing $\varepsilon$ sufficiently small.
In the application of the lemma, we of course must choose the parameter $r$ in the lemma sufficiently small that the inner and outer contact surface estimates do not interact. To ensure the sign of both gravity coefficients are negative, we require that $h$ is sufficiently large to absorb the constant from the lemma. We apply this, and discard the free surface and gravity terms with negative coefficients, obtaining as a result:
\begin{align*}
    0 &\leq -(1+\varepsilon)a(h) + Ca(h)(1+h)(1 -a'(h)) - \frac{\varepsilon}{C} a(h)^{(d-1)/d}
\end{align*}
which we rearrange to write as
\[ ((1+\varepsilon) + Ca(h)(1+h))a'(h) \leq - Ca(h)(1+h) - \frac{\varepsilon}{C} a(h)^{(d-1)/d} \]
Recall that by definition $a$ is decreasing in $h$, so $a'$ is nonpositive and all other variables are nonnegative. Since $L$ has finite energy, and \lref{energy_bdd_below} ensures that energy controls gravity up to an additive constant, it follows that $(1+h)a(h)\to 0$ as $h\to \infty$. Thus, by choosing $h$ sufficiently large and $\varepsilon$ sufficiently small, we may assume that the coefficient of $a'$ is at most $2$ and absorb the first term on the right into the second, giving
\[ a' \leq - \frac{1}{C} a(h)^{(d-1)/d} \]
Finally, we observe that solutions of this differential inequality with positive initial data become 0 in after finite time, and conclude that $\FS_L$ is vertically bounded.
\end{proof}

\section{Sliding comparison for general viscosity solutions}\label{app.sliding_comparison}
In this appendix we recall a comparison principle for level set viscosity solutions of the mean curvature flow with capillary boundary conditions due Barles \cite{Barles} and independently Ishii and Sato \cite{IshiiSato}. We then explain how to translate that result to our setting.
\subsubsection*{Comparison for curvature flows with capillary condition} In \cites{Barles,IshiiSato} they formulate the mean curvature flow with capillary contact angle condition in the level set form. For our purposes we just consider a level set PDE related, as we will explain below, to \eref{PDE-inc}. Let $\Omega$ be a $C^1$ bounded domain in $\R^n$, for us $n = d+1$. Split $\partial \Omega = \Gamma_N\cup \Gamma_D$ into two disjoint pieces, the Neumann boundary $\Gamma_N$ and the Dirichlet boundary $\Gamma_D$. We will take $\Gamma_N$ to be relatively open in $\partial \Omega$ and $\Gamma_D$ to be relatively closed. Then we will be considering level set PDE in the class
\begin{equation}\label{e.level-set-interior}
    u_t -\textup{tr}\left((I-\frac{\grad u \otimes \grad u}{|\grad u|^2})D^2u\right) +h(x,t)|\grad u| = 0 \ \hbox{ in } \ \Omega
\end{equation}
with a continuous forcing $h$ and the capillary boundary condition
\begin{equation}\label{e.level-set-bdry}
    \frac{\partial u}{\partial \nu} = \cos \YC |\grad  u| \ \hbox{ on } \ \Gamma_N.
\end{equation}

\begin{lemma}[\cite{IshiiSato}*{Theorem 2.1}, comparison principle for singular degenerate parabolic equations]\label{l.singular_degenerate_parabolic_comparison}
If $u$ and $v$ are respectively an upper semicontinuous subsolution and lower semicontinuous supersolution of \eref{level-set-interior} and \eref{level-set-bdry} on $(0,T) \times \Omega$ with $u(x,0) \leq v(x,0)$ on $\overline{\Omega}$ and $u(x,t) \leq v(x,t)$ on $\Gamma_D \times [0,T]$ then $u(x,t) \leq v(x,t)$ on $\overline{\Omega} \times (0,T)$.
\end{lemma}
Note that the Dirichlet boundary is not present in \cite{IshiiSato}*{Theorem 2.1}, but it is standard and does not add any difficulty to the proof there, since a positive maximum of $u - v$ on the Dirichlet boundary can be immediately ruled out by the assumed boundary ordering.

We recall the standard definition of viscosity subsolution of \eref{level-set-interior}: whenever $\varphi$ a smooth test function touches $u$ from above at $(x_0,t_0) \in \Omega \times (0,T)$ then
\[\left[\varphi_t -\textup{tr}^*\left((I-\frac{\grad \varphi \otimes \grad \varphi}{|\grad \varphi|^2})D^2\varphi\right) +h(x,t)|\grad \varphi|\right](x_0,t_0) \leq 0\]
where we adopted the shorthand notation
\[\textup{tr}^*\left((I-\frac{p \otimes p}{|p|^2})M\right) := \limsup_{\substack{q \to p \\ q \neq 0}}\textup{tr}^*\left((I-\frac{q \otimes q}{|q|^2})M\right)\]
The supersolutions are defined analogously with $\textup{tr}_*$. See \cite{CrandallIshiiLions}*{Section 9} for further discussion about the case of touching with degenerate slope.

\subsubsection*{Viscosity property for sliding level sets} Now we explain that viscosity solutions of \eref{PDE}, in the sense of Definitions \ref{def:visc_subsln_with_test_region} and \ref{def:visc_supsln_with_test_region}, can be interpreted in level-set form. We also add in a sliding factor to take advantage of the parabolic comparison principle \lref{singular_degenerate_parabolic_comparison}.

\begin{lemma}\label{l.sliding-viscosity-soln}
Let $f$ be a smooth increasing function of $t$, and $L$ be a subsolution of \eref{PDE}. Then the sliding level set function
\[ u(t, x, z):={\bf 1}_{L^*}(x,z - f(t)) \]
is a subsolution on $\Cyl_R\times (0, T]$ of
\[\begin{cases}
    \partial_t u -\textup{tr}\left((I-\frac{\grad u \otimes \grad u}{|\grad u|^2})D^2u\right) +g\big(z -F-\tfrac{\lambda_L + gf(t) + f'(t)}{g}\big)|\grad u| \leq 0 & \hbox{on }\Cyl_R \\
    \frac{\partial u}{\partial \nu} = \cos \YC |\grad  u| & \hbox{on }\thinCyl_R
\end{cases}\]
Similarly, for $f$ decreasing and $L$ a supersolution, $\one_{L_*}(x, z - f(t))$ is a supersolution.
\end{lemma}
\begin{proof}
We verify the claim for the subsolution property. Suppose $\varphi$ is a smooth test function, such that $u - \varphi$ has a strict maximum at $(t_0, x_0, z_0)$ with $t_0 > 0$. Then $\one_{L^*} - \varphi(t, x, z + f(t))$ has the same at $(t_0, x_0, z_0 - f(t_0))$, and we may split into four cases, depending on whether $(x_0, z_0 - f(t_0))\in \partial L^*$ and whether $(x_0, z_0 - f(t_0))\in \thinCyl_R$.

Starting with the trivial case $(x_0, z_0 - f(t_0)) \notin \partial L^*\cup \thinCyl_R$, at $(t_0, x_0, z_0-f(t_0)$ we have
\[ \partial_t \varphi = f'\partial_z\varphi,\quad \nabla \varphi = 0, \quad D^2\varphi \geq 0\]
The curvature term in \eqref{e.level-set-interior} is interpreted in the sense of upper semicontinuous envelopes, as
\[ \sup_{|\nu| = 1}\mathrm{tr}\left((I - \nu\otimes \nu)D^2\varphi\right) \]
Therefore, we directly obtain
\begin{align*}
    &\partial_t \varphi - \mathrm{tr}\left((I - \frac{\nabla\varphi\otimes\nabla\varphi}{|\nabla\varphi|^2})D^2\varphi\right) + g(z_0 - f(t_0) - F - \tfrac{\lambda_L}{g})|\nabla \varphi| 
    \\&= - \mathrm{tr}\left((I - \frac{\nabla\varphi\otimes\nabla\varphi}{|\nabla\varphi|^2})D^2\varphi\right) 
    \\&\leq 0
\end{align*}
In the trivial boundary case, $(x_0, z_0) \in \thinCyl_R\setminus \partial L^*$, we similarly have that for $\one_{L^*} - \varphi(t,x,z - f(t))$ to have a boundary maximum, we must have
\[ \partial_r \varphi \leq 0, \quad \nabla_{\thinCyl_R}\varphi = 0,\quad\partial_t\varphi = 0 \]
Hence, from $\cos\alpha\in (-1,1)$ we conclude that
\[ \partial_r\varphi \leq \cos\alpha |\nabla\varphi| = \cos\alpha |\partial_r\varphi| \]

Proceeding to the nontrivial interior case $(x_0, z_0) \in \partial(L^* + f(t_0)e_z)$, we assume that $\nabla \varphi(t_0, x_0, z_0) \neq 0$. Letting $\Phi := \{ (x, z - f(t_0)) : \varphi(t_0, x, z) \geq 1 \}$, we note that $\Phi$ has a smooth parametrization near $(x_0, z_0)$, and can be redefined elsewhere if necessary to be a smooth test region which touches $L^*$ from outside at $(x_0, z_0 - f(t_0))$. Applying Definition \ref{def:visc_subsln_with_test_region}, we obtain that either $-H_\Phi(x_0, z_0 - f(t_0)) + g(z_0 - f(t_0) - F) \leq \lambda_{L}$ or $(x_0, z_0) \in \gamma_\Phi$ and $\THG{\Phi}(x_0, z_0 - f(t_0)) \geq \alpha$.

We recall that the outward normal to $\Phi$ is $-\frac{\nabla \varphi}{|\nabla \varphi|}$, and with our sign convention the mean curvature is negative where $\Phi$ is convex (e.g. where $\varphi$ is concave). Therefore, 
\begin{align*}
    -H_\Phi &= \nabla\cdot(-\tfrac{\nabla\varphi}{|\nabla\varphi|})
    \\&= \frac{-1}{|\nabla\varphi|}\left(I - \frac{\nabla\varphi \otimes \nabla\varphi}{|\nabla\varphi|^2}\right) : D^2\varphi
\end{align*}
from which we conclude that in the first case, we have at $(x_0, z_0)$ that
\[ -\textup{tr}\left((I-\frac{\grad \varphi \otimes \grad \varphi}{|\grad \varphi|^2})D^2\varphi\right) +g(z_0 - f(t_0)-F-\tfrac{\lambda_L}{g})|\grad \varphi| \leq 0 \]
Since $\one_{L^*}(x,z) - \varphi(t,x,z-f(t))$ has a local maximum at $(t_0, x_0, z_0)$, we also obtain that $\partial_t\varphi = f'\partial_z\varphi$ at $(t_0,x_0,z_0)$. Bounding $\partial_z \varphi \leq |\nabla \varphi|$, we conclude that
\[ \partial_t\varphi-\textup{tr}\left((I-\frac{\grad \varphi \otimes \grad \varphi}{|\grad \varphi|^2})D^2\varphi\right) +g(z_0 - f(t_0)-F-\tfrac{\lambda_L}{g})|\grad \varphi| \leq f'(t_0)|\nabla \varphi| \]

In the second case, we can compute the contact angle with our sign convention as \[ \cos\THG{\Phi}(x_0, z_0 - f(t_0)) = -(-\tfrac{\nabla\varphi}{|\nabla\varphi|}\cdot \tfrac{x_0}{|x_0|}) \leq \cos\alpha  \]
Therefore, we have the boundary inequality
\[ \partial_r\varphi(x_0, z_0) \leq |\nabla\varphi(x_0, z_0)|\cos\alpha \]
where $r$ is the cylindrical radius coordinate.

The degenerate case $\nabla \varphi(t_0, x_0,z_0) = 0$ can be handled by a standard tilting argument, since we only need to check, for interior touching,
\[\left[-\textup{tr}^*\left((I-\frac{\grad \varphi \otimes \grad \varphi}{|\grad \varphi|^2})D^2\varphi\right) +g(z-F-\tfrac{\lambda}{g})|\grad \varphi|\right]_{(x_0,z_0)} \leq 0,\]
and a similar condition for boundary touching. Thus, we conclude.

\end{proof}

\section{Some measure theory results}

\begin{lemma}\label{l.limsup_liminf_integral}
    Let $Q_n$ be a sequence of functions on a finite-measure set $I$ which are dominated by an integrable function $\hat{Q}$. If 
    \[ \int_I \limsup_n Q_n \leq  \liminf_n \int_I Q_n \]
    then the $Q_n$ converge in measure as $n\to \infty$ to $\limsup_n Q_n$.
\end{lemma}

\begin{proof}
    First note that
    \[ \lim_N\one_{\{Q_N > \varepsilon + \limsup_n Q_n\}} = 0 \]
    pointwise in time for any $\varepsilon > 0$ by definition. Bounded convergence implies
    \[ |\{Q_N > \varepsilon + \limsup_n Q_n\}| \to 0 \]
    On the other hand, we have
    \[ Q_N + \varepsilon\one_{\{ Q_N < -\varepsilon + \limsup_n Q_n\}} \leq Q_N\vee\limsup_n Q_n \]
    We integrate both sides in time and take the limsup in $N$, to get
    \begin{align*}
        \liminf_N & \int_I Q_Ndt + \limsup_N \varepsilon|\{ Q_N < -\varepsilon + \limsup_n Q_n\}| \\&\leq \limsup_N\left(\int_I Q_Ndt +  \varepsilon|\{ Q_N < -\varepsilon + \limsup_n Q_n\}|\right)
        \\&\leq \limsup_N \int_I Q_N\vee \limsup_n Q_n\,dt
        \\&\leq \int_I \limsup_N(Q_N\vee \limsup_n Q_n) dt \tag{Reverse Fatou}
        \\&= \int_I \limsup_N Q_N dt
        \\&\leq \liminf_N\int_I Q_Ndt \tag{Hypothesis}
    \end{align*}
    The only step that requires justification is the reverse Fatou inequality. This holds since $Q_N\vee \limsup_n Q_n$ is dominated by $\hat{Q}$ (since $\hat{Q}$ dominates the $Q_N$).

    With this, we cancel $\liminf_N\int_I Q_Ndt$ from both sides of the inequality and conclude that
    \[ \limsup_N|\{ Q_N < -\varepsilon + \limsup_n Q_n\}| = 0 \]
    so the convergence in measure holds.
\end{proof}

\begin{lemma}\label{l.subseq_msb_sel}
    Let $Y$ be a separable metric space (equipped with the Borel $\sigma$-algebra), $T$ a measurable space, and $L_k(t)$ a sequence of measurable functions $T\to Y$. Let $\mathcal{M}: T \to 2^Y$ be such that $\mathcal{M}(t)$ is the set of subsequential limits of the $L_k(t)$, for each $t$. Then $L(t)$ is measurable in the sense of Lemma \ref{lem:msb_sel}.

    Let $F:T\times Y\to \R$ be measurable. Let $a\in \R$, and let $\mathcal{M}_{F,a}(t)$ denote the set of subsequential limits $L$ of the $L_k(t)$ such that $F(t,L) = a$. Then $\mathcal{M}_{F,a}$ is also measurable.
\end{lemma}
\begin{proof}
Note that
\[\mathcal{M}(t) = \bigcap_i\bigcap_j\bigcup_{k>j} \{L: d(L_k(t),L) < 1/i\} =: \bigcap_i\bigcap_j\bigcup_{k>j} E_{ijk}(t)\]
We then have for any $A$ that
\begin{align*}
    \{t: \mathcal{M}(t) \cap A \neq \emptyset\} &= \{ t: \bigcap_i\bigcap_j\bigcup_{k>j} (E_{ijk}(t) \cap A) \neq \emptyset\}
    \\&= \bigcap_i\bigcap_j\bigcup_{k>j} \{ t: (E_{ijk}(t) \cap A) \neq \emptyset\}
    \\&= \bigcap_i\bigcap_j\bigcup_{k>j} L_k^{-1}(B_{1/i}(A))
\end{align*}
where the last line is measurable as a countable union and intersection of measurable sets, using that $L_k$ is measurable and $B_{1/i}(A)$ is open.

For the second part, write $G:T\to 2^Y\times 2^\R$ as $G(t) = (\mathcal{M}(t), \{F(t, L): L\in \mathcal{M}(t)\})$. Note that the notions of measurability for multivalued functions compose, and thus $G$ is measurable. For any $A$, we have $\mathcal{M}_{F,a}^{-1}(A) = G^{-1}(A\times \{a\})$, where $A\times \{a\}$ is Borel in $Y\times \R$ if $A$ is Borel in $Y$. We conclude that $\mathcal{M}_{F,a}$ is measurable.
\end{proof}

\section{Local regularity of pinned contact lines}\label{a.local-contact-line-reg}

We work in an abstract local framework which encompasses almost-minimizers of the type in Definition \ref{def:almost_minimizer}. In contrast to the rest of the paper, in this section we use the notation $B_r(x)$ to refer to $d$-dimensional balls.

Let $\Gamma$ be a bounded $C^2$-embedded $(d-1)$-surface in $\R^d$, with unit normal $\nu$, and write $\Gamma_{s_0}$ for the one-sided tubular neighborhood $\{ y + s\nu : y\in \Gamma, s\in (0,s_0)\}$. We recall that the projection $\Gamma_r\ni p\mapsto (y,s)$ is well-defined when $s$ is sufficiently small, depending on the $C^2$ character of $\Gamma$, so we will assume this to be the case. Define the energy $E$ for $U\in \mathrm{Cacc}(\Gamma_{s_0})$:
\begin{equation}
    E[U] = \mathrm{Per}(U; \Gamma_{s_0}) - \cos\theta |\mathrm{tr}_\Gamma(U)|
\end{equation}
and consider $(\Lambda_0, \mu_\pm, r_0)$-minimizers of $E$, defined to satisfy (for any $V$ with $U\Delta V \Subset \{ |x - x_0| < r\}\cap (\Gamma_{s_0}\cup \Gamma)$ for some $p_0\in \R^d$),
\begin{equation}\label{e.abstract_almost_minimizer}
    E[U] \leq E[V] + \Lambda_0\vol(U\Delta V) + \Diss[U, V]
\end{equation}
where
\begin{equation}
    \Diss[U, V] = \mu_+|\mathrm{tr}_\Gamma(V\setminus U)| + \mu_- |\mathrm{tr}_\Gamma(U\setminus V)|
\end{equation}
subject to the assumption $[\cos\theta - \mu_+, \cos\theta + \mu_-]\subset (-1, 1)$.

We first present density bounds for almost-minimizers, as a modification of \cite{DePhilippisMaggi}*{Lemma 2.8}.

\begin{lemma}
    Let $U$ be a $(\Lambda_0, \mu_\pm, r_0)$-minimizer in $\Gamma_{s_0}$ in the sense of \eref{abstract_almost_minimizer}. Then for any $x$ in the measure-theoretic support of $U$, there exists $c = c(x) > 0$ and $r_x > 0$ such that for each $r\in (0,r_x)$, $\vol(U\cap B_r(x)) \geq c\vol(B_r(x))$. Here, the parameters are uniform for $x$ avoiding $\partial \Gamma_{s_0}\setminus \Gamma$.
\end{lemma}

Recall that the measure theoretic support of $U$ is the set of points $x\in\R^d$ such that $\vol(U \cap B_r(x)) >0$ for all $r>0$.

\begin{proof}

For $x\in \Gamma_{s_0}$, set
\[ r_x = \min(r_0, \mathrm{dist}(x,\partial\Gamma_{s_0}\setminus \Gamma))  \]
and set
\[ m_x(r) = \vol(U\cap B_r(x)),\quad r\in (0,r_x) \]
Observe that by the isoperimetric inequality:
\begin{align*}
    m_x(r)^{(d-1)/d} &= \vol(U\cap B_r(x))^{(d-1)/d}
    \\&\leq C_{isop}\mathrm{Per}(U\cap B_r(x))
    \\&=C_{isop}\left(\mathrm{Per}(U; B_r(x)\cap\Gamma_{s_0}) + |\mathrm{tr}_\Gamma(U\cap B_r(x))| + |\mathrm{tr}_{\partial B_r(x)\setminus \Gamma}(U)| \right)
\end{align*}
Here, the terms are related by the almost-minimality inequality for competitor $U\setminus B_r(x)$:
\[ \mathrm{Per}(U; B_r(x)\cap\Gamma_{s_0}) \leq \Lambda_0\vol(U\cap B_r(x)) + (\mu_- + \cos\theta)|\mathrm{tr}_\Gamma(U\cap B_r(x))| \]
and the trace inequality from \lref{bv_trace_ineq}:
\[ |\mathrm{tr}_\Gamma(U\cap B_r(x))| \leq \mathrm{Per}(U; B_r(x)\cap \Gamma_{s_0}) + |\mathrm{tr}_{\partial B_r(x)\setminus \Gamma}(U)| + C\vol(U\cap B_r(x)) \]
We add $(1 + \varepsilon)$ times the first inequality to the second to get:
\begin{align*}
    &\varepsilon\mathrm{Per}(U; B_r(x)\cap\Gamma_{s_0}) + (1 - (1 + \varepsilon)(\mu_- + \cos\theta))|\mathrm{tr}_\Gamma(U\cap B_r(x))| 
    \\&\quad\leq \left((1 + \varepsilon)\Lambda_0 + C\right)\vol(U\cap B_r(x)) + |\mathrm{tr}_{\partial B_r(x)\setminus \Gamma}(U)|
\end{align*}
Here, we make use of the assumption $\mu_- + \cos\theta < 1$ to give room to choose $\varepsilon > 0$ so that both coefficients on the left are positive. Subsequently, we can return to the original computation to obtain
\begin{align*}
    m_x(r)^{(d-1)/d} &\leq Cm_x'(r) + Cm_x(r)
\end{align*}
using that $m_x$ is absolutely continuous with $m_x'(r) = |\mathrm{tr}_{\partial B_r(x)\setminus \Gamma}(U)|$ for a.e. $r$. Up to shrinking $r_x$ based on the explicit value of the constants, we can absorb the $m_x(r)$ into the $m_x(r)^{(d-1)/d}$. Assuming $x$ is in the measure-theoretic support of $U$, we have $m_x$ positive for all $r$, and we can divide to obtain:
\[ 1 \leq C(m_x^{1/d})'. \]
Integrating, we get
\[ c\vol(B_r\cap \Gamma_{s_0}) \leq m_x(r). \]
By the symmetric argument carried out for $\Gamma_{s_0}\setminus U$, we obtain for $x$ in the measure-theoretic support of $\Gamma_{s_0}\setminus U$:
\[ m_x(r) \leq (1-c)\vol(B_r\cap \Gamma_{s_0}). \]
\end{proof}

With similar modifications, the following perimeter estimate can also be shown. We omit the proof.

\begin{lemma}[\cite{DePhilippisMaggi}*{Lemma 2.8}]\label{l.perimeter_bound_near_contact_line}
Let $U$ be a $(\Lambda_0, \mu_\pm, r_0)$-minimizer in $\Gamma_{s_0}$ in the sense of \eref{abstract_almost_minimizer}. Let $x \in \Gamma$ then for $r \leq r_0$ and $t \leq s_0/2$
    \[\textup{Per}(U;B_r(x) \cap \Gamma_t) \leq C r^{d-2}t.\]
\end{lemma}

We recall that $\mathrm{tr}_{\Gamma}(U)$ can be defined for $U\in \mathrm{Cacc}(\Gamma_{s_0})$ as $\lim_{s\to 0^+} \one_U(\cdot,s)$, where the slices $\one_U(\cdot,s)$, interpreted as functions on $\Gamma$, are Cauchy in $L^1(\Gamma, d\mathcal{H}^{d-1})$ \cite{Giusti1984}*{Lemma 2.4, generalized by incorporating the Jacobian for the tubular neighborhood parametrization}. Moreover, the trace satisfies the following distributional interpretation: for all $\varphi\in C^1(\R^d; \R^d)$,
\begin{equation}\label{e.distributional_trace}
    \int_{U} \nabla \cdot\varphi = \int_{\Gamma_{s_0}\cap \partial^*U} \varphi\cdot \nu_U d\mathcal{H}^{d-1} - \int_{\mathrm{tr}_\Gamma(U)} \varphi\cdot\nu_\Gamma d\mathcal{H}^{d-1}
\end{equation}
We use this along with the estimate in \lref{perimeter_bound_near_contact_line} to estimate the length of the contact line.

\begin{lemma}\label{l.local_length_bound}
Let $U$ be a $(\Lambda_0, \mu_\pm, r_0)$-minimizer in $\Gamma_{s_0}$. Then
\begin{equation}
    \mathrm{Per}(\mathrm{tr}_{\Gamma}(U); \Gamma) := \sup_\varphi \int_{\Gamma} \one_{\mathrm{tr}_\Gamma(U)}\nabla_{\Gamma}\cdot \varphi \leq C(\Lambda_0, \mu_\pm, r_0, s_0)
\end{equation}
where the supremum is taken over (tangent) vector fields on $\Gamma$ with compact support. In particular, we bound the length of the contact line as:
\[ \mathcal{H}^{d-2}(\gamma_U) = \mathcal{H}^{d-2}(\partial^*_{\Gamma}\mathrm{tr}_\Gamma(U)) \leq C(\Lambda_0, \mu_\pm, r_0, s_0). \]
Here, $\partial^*_\Gamma$ is the reduced boundary operator on $\Gamma$, which selects boundary points where $\mathrm{tr}_\Gamma(U)$ has a measure-theoretic normal inside $\Gamma$.
\end{lemma}
\begin{proof}
Write the tubular coordinate functions as $Y(x) = \mathrm{argmin}_{y\in \Gamma} |x - y|$ and $N(x) = |x - Y(x)|$. Set $U_s = Y(U\cap N^{-1}(\{ s\}))$, the slice of $U$ at distance $s$ from $\Gamma$, projected to $\Gamma$. We identify $U_0$ with $\mathrm{tr}_\Gamma(U)$.

The slices of a $BV$ function with respect to a one-dimensional coordinate are $L^1$-continuous, so we have that $\lim_{s\to 0^+}|U_s \Delta U_0| = 0$. It follows by lower semicontinuity of perimeter, that we have
\[ \mathrm{Per}_\Gamma(U_0) \leq \liminf_{s\to 0^+} \mathrm{Per}_\Gamma(U_s) \]
where we write $\mathrm{Per}_\Gamma$ to emphasize that the perimeter of a subset of $\Gamma$ is computed by testing against compactly supported tangent vector fields to $\Gamma$.

The coarea formula \cite{Maggi_2012}*{Theorem 18.8} gives for $s_1 < s_2$ that
\begin{align*}
    \int_{s_1}^{s_2} \mathrm{Per}_\Gamma(U_s) ds&\leq C\int_{s_1}^{s_2} \mathcal{H}^{d-2}(\partial^*U\cap N^{-1}(\{s\}))ds
    \\&= C\int_{\partial^* U\cap (\Gamma_{s_2}\setminus \Gamma_{s_1})} |\nabla^{\partial^* U}N| d\mathcal{H}^{d-1} 
    \\&\leq C\mathrm{Per}(U; \Gamma_{s_2}\setminus \Gamma_{s_1})
\end{align*}
where here the implicit constant in the first line appears from estimating the Jacobian for the change of variables, and is modified in the last line from estimating $N$ in $C^1$. Applying \lref{perimeter_bound_near_contact_line}, with a suitable covering argument, we have
\[ \mathrm{Per}(U; \Gamma_{s_2}) \leq C s_2 \]
By sending $s_1\to 0$, we conclude that for any $s_2$, we have
\[ \frac{1}{s_2}\int_{0}^{s_2} \mathrm{Per}_\Gamma(U_s) ds \leq C \]
and hence there exists $s\in (0, s_2)$ such that $\mathrm{Per}_\Gamma(U_s) \leq C$. It follows that
\[ \mathrm{Per}_\Gamma(U_0) \leq \liminf_{s\to 0^+} \mathrm{Per}_\Gamma(U_s) \leq C \]
\end{proof}

\bibliography{wilhelmy-articles}

\end{document}